\documentclass[a4paper,11pt,reqno]{amsart}
\usepackage{amssymb,amsmath,amsthm}
\numberwithin{equation}{section}
\usepackage{natbib}
\usepackage{paralist}               
\usepackage[margin=0.75in]{geometry}
\usepackage[titletoc]{appendix}     %%  Designing the APPENDIX
\usepackage{mathrsfs}
\usepackage{setspace}
\usepackage[final]{graphicx}
\usepackage{epstopdf}
\usepackage{subcaption}
\usepackage{float}
\usepackage{enumitem}
\usepackage[bookmarks,bookmarksnumbered,bookmarksopen,breaklinks,linktocpage]{hyperref}
\newcommand{\esssup}{\operatorname*{ess\,sup}}
\newtheorem{Theorem}{Theorem}[section]

\newtheorem{Lemma}{Lemma}[section]

\theoremstyle{definition}           %% \theoremstyle{plain}, or \theoremstyle{remark}

\theoremstyle{remark}
\newtheorem{Remark}{Remark}[section]
\usepackage{multirow}
\usepackage{color,soul}
\usepackage{soul}
\usepackage{array}
\newcolumntype{L}[1]{>{\raggedright\let\newline\\\arraybackslash\hspace{0pt}}m{#1}}
\newcolumntype{C}[1]{>{\centering\let\newline\\\arraybackslash\hspace{0pt}}m{#1}}
\newcolumntype{R}[1]{>{\raggedleft\let\newline\\\arraybackslash\hspace{0pt}}m{#1}}

\begin{document}
	\title[Machine Learning for Control PDEs]{Computational Control of Nonlinear Partial Differential Equations Using Machine Learning}
	
	\author[Maximilian Kurbanov
	]{Maximilian Kurbanov
	}
	
	\address[Maximilian Kurbanov
	]{Department of Mathematics \\
		Texas A\&M University, College Station, TX, 77843 USA}
	
	\email[Maximilian Kurbanov
	]{\href{mailto:maxkurbanov@tamu.edu}{maxkurbanov@tamu.edu}}

	\author[Minh-Nhat Phung]{Minh-Nhat Phung }
	
	\address[Minh-Nhat Phung]{Department of Mathematics \\
		Texas A\&M University, College Station, TX, 77843 USA}
	
	\email[Minh-Nhat Phung]{\href{mailto:pmnt1114@tamu.edu}{pmnt1114@tamu.edu}}
	
	\author[Minh-Binh Tran]{Minh-Binh Tran}
	\address[Minh-Binh Tran]{Department of Mathematics \\
		Texas A\&M University, College Station, TX, 77843 USA}
	
	\email[Minh-Binh Tran]{\href{mailto:minhbinh@tamu.edu}{minhbinh@tamu.edu}}
	\thanks{ M.-N. Phung and M.-B. T are  funded in part by     NSF CAREER  DMS-2303146, and NSF Grants DMS-2204795, DMS-2305523,  DMS-2306379. }
	
	\begin{abstract}
		The numerical reconstruction of controls for nonlinear partial
differential equations (PDEs) remains a challenging and relatively
underdeveloped problem, despite the extensive literature on
controllability theory.  In this work, we introduce an
operator-decomposed physics-informed neural network framework,
called WeightedPINN, for approximating controls in nonlinear PDE
settings.  The method is designed for both internal and bilinear
control problems and incorporates the governing equation, boundary and
initial conditions, and terminal control constraints directly into the
training objective.

The main feature of WeightedPINN is that the different components of
the controlled PDE residual are weighted separately.  In particular,
the time derivative, directional diffusion terms, nonlinear response,
and control term are assigned independent adaptive space--time weights,
and the same weighted formulation is applied to the boundary, initial,
and terminal constraints.  This produces a control-aware residual
metric that is more sensitive to operator-level imbalance and to the
mechanism through which the unknown control enters the equation.

We provide a convergence analysis for the proposed method and present
numerical experiments for semilinear heat and wave equations with
internal and bilinear controls.  The high-dimensional experiments
demonstrate improved residual-based testing errors compared with the
standard PINN baseline, while lower-dimensional manufactured-solution
benchmarks show improved direct reconstruction errors for both the
state and the control against several adaptive and control-oriented
PINN methods.  The results suggest that WeightedPINN is particularly
effective in regimes where componentwise residual imbalance,
anisotropy, variable coefficients, or control-identification
sensitivity play a significant role.
	\end{abstract}
	
	\maketitle
	%%%%%%%%%%%%%%%%%%%%%%%%%%%%%%%%%%%%%%%%%%%%%%%%%%%%%%%%%%%%%%%%%%%%%%%%%%%%%%%%%%%%%%%%%%%%%%%%%%%%
	\allowdisplaybreaks
	%\tableofcontents

	\section{Introduction}

The control theory of nonlinear partial differential equations (PDEs)
has been the subject of sustained and intensive investigation over the
past several decades; we refer, for instance, to
\cite{Coron2007,Zuazua2005,Lions1988} and the extensive bibliography
therein. Despite this substantial body of work, the literature remains
comparatively limited on the numerical reconstruction of controls for
nonlinear systems. In particular, while many controllability results,
including local, global, approximate, and null controllability, have
been established under diverse structural assumptions, considerably
less is known about how to compute reliable and accurate numerical
approximations of the corresponding control functions.

This difficulty is not merely technical. The intrinsic nonlinearity of
the dynamics typically precludes the use of superposition principles
and may lead to instabilities, loss of compactness, or sensitivity with
respect to perturbations. These features complicate both the analytical
and numerical reconstruction of controls. Consequently, the design of
stable, consistent, and convergent numerical methods for control
reconstruction remains a challenging problem, especially in genuinely
nonlinear regimes.

This question appears to have been investigated only relatively
recently, starting with the initial contribution
\cite{MunchTrelat2022}, where a theoretical proof of a numerical scheme was provided without numerical results. Since then, it has begun to develop into an
emerging research direction, although it remains far from fully
established and is still comparatively unexplored. Constructive and
least-squares approaches have been proposed for semilinear wave and
heat equations, including distributed and boundary control problems, as
well as space--time methods for the approximation of controls; see
\cite{BhandariLemoineMunch2023,LemoineMarinGayteMunch2021,
LemoineMunch2023,BottoisLemoineMunch2023,Munch2023}. These
contributions highlight both the difficulty and the importance of
obtaining accurate numerical approximations of controls in nonlinear
PDE settings. It is also worth noting that classical numerical methods
for PDEs may face increasing computational challenges in
high-dimensional problems, a phenomenon commonly referred to as the
curse of dimensionality. This provides further motivation for
developing mesh-free or sampling-based methods for high-dimensional
PDE control problems.

Physics-informed neural networks (PINNs) are neural-network-based
methods for solving differential equations by incorporating the
governing equations directly into the training loss. Introduced in
\cite{Raissi2017PartI,Raissi2017PartII,Raissi2019}, PINNs approximate
the solution by a neural network trained with data, boundary
conditions, and PDE residuals computed by automatic differentiation.
They have become a central tool in scientific machine learning, with
applications to forward and inverse problems, particularly in
data-sparse settings \cite{Raissi2019}. From a computational
perspective, substantial effort has been devoted to improving the
robustness and accuracy of PINNs, especially in connection with
optimization difficulties, stiffness, and loss imbalance
\cite{WangTengPerdikaris2021}. Important developments include
conservative formulations, domain decomposition methods, adaptive
activation functions, adaptive weighting strategies, and
self-adaptive PINNs; see, for instance,
\cite{JagtapKharazmiKarniadakis2020,JagtapKawaguchiKarniadakis2020,
McClennyBragaNeto2023}. On the theoretical side, convergence and
generalization properties have begun to be understood for certain
classes of PDEs \cite{ShinDarbonKarniadakis2020,MishraMolinaro2023}.
We also refer to the recent surveys
\cite{RaissiAhmadiPerdikarisKarniadakis2024,DeRyckMishra2024} for
further developments and perspectives.

\subsection{Positioning with respect to existing PINN methodologies}
\label{Sec:PriorPINN}

The PINN literature has grown rapidly since
\cite{Raissi2017PartI,Raissi2017PartII,Raissi2019}, and many
methodological refinements have been proposed. Since the present work
introduces an adaptive PINN framework designed specifically for the
numerical reconstruction of controls for nonlinear PDEs, it is useful
to position our contribution within the broader landscape of PINN
variants. We organize the discussion into eight categories. In each
case, we explain both the principal mechanism of the existing method
and the aspects of the control reconstruction problem that remain
unaddressed.

\paragraph{(A) Loss-balancing methods at the level of whole loss terms.}
A standard PINN loss is typically written as a weighted sum
\[
\mathcal L_{\mathrm{PINN}}
=
\mathcal L_{\mathrm{PDE}}
+
\lambda_{\mathrm{bd}}\mathcal L_{\mathrm{bd}}
+
\lambda_{\mathrm{ic}}\mathcal L_{\mathrm{ic}}
+\cdots .
\]
The different components of this loss may have very different
magnitudes and gradient scales during training, leading to severe
imbalance. This issue is analyzed through gradient-flow pathologies in
\cite{WangTengPerdikaris2021} and from a neural tangent kernel
perspective in \cite{WangYuPerdikaris2022_NTKPINN}. This line of work
also proposes dynamic update rules for the scalar coefficients
\(\lambda_\bullet\). The inverse Dirichlet weighting method of
\cite{MadduSturmMuellerSbalzarini2022_InverseDirichlet} provides an
alternative rebalancing rule based on residual variance. These methods
improve the scaling among different loss components, but they operate
at the level of whole loss terms. They do not distinguish the
individual operator components inside the PDE residual itself, such as
the time derivative, the diffusion operator, the nonlinear response,
or the control term. Nor do they separately treat the control
constraint that defines the controllability problem.

\paragraph{(B) Self-adaptive PINNs with pointwise weights.}
A second family attaches trainable per-collocation-point weights to
the residual, replacing
\[
\sum_j |R[u_\theta](t_j,x_j)|^2
\]
by
\[
\sum_j \lambda_j^2 |R[u_\theta](t_j,x_j)|^2,
\]
where the weights \(\lambda_j\) are updated to emphasize
ill-resolved points \cite{McClennyBragaNeto2023}. This creates a soft
attention mechanism over the collocation points and improves local
accuracy where the residual is large. Such adaptivity, however, is
attached to points or samples rather than to the differential
operator. It does not identify which component of the PDE residual,
for example the time derivative, diffusion term, nonlinear term, or
control term, is responsible for the local error. It is also not
specifically designed for reconstructing controls that enter the
equation through a particular structural mechanism.

\paragraph{(C) Residual-based adaptive sampling.}
A third family modifies neither the residual norm nor the loss
weights, but instead changes the distribution of collocation points.
Residual-adaptive refinement (RAR) adds collocation points where the
residual is large; the family of RAR, RAR-D, and RAD strategies has
been systematized in
\cite{WuZhuTan2023_AdaptiveSamplingPINN} and is implemented in the
DeepXDE framework \cite{LuMengMaoKarniadakis2021_DeepXDE}. Related
ideas include retain-and-resample (\(R^3\)) sampling
\cite{DawBuWangPerdikarisKarniadakis2023_R3}. These methods improve
the resolution of localized features by changing where the residual
is evaluated. However, the residual itself remains a single scalar
quantity with no operator-level decomposition, and the method is
agnostic to the mechanism through which the control enters the PDE.

\paragraph{(D) Weighted residual and variational PINNs.}
A fourth family replaces \(R[u_\theta]\) by
\(\phi(t,x)R[u_\theta]\), where \(\phi\) is a suitable test or weight
function; equivalently, the residual is measured in a weighted norm.
Variational PINNs (VPINNs) and their \(hp\)-adaptive versions
(hp-VPINNs) test the residual against piecewise polynomial spaces
\cite{KharazmiZhangKarniadakis2019_VPINN,
KharazmiZhangKarniadakis2021_hpVPINN}. Gradient-enhanced PINNs
(gPINNs) \cite{YuLuMengKarniadakis2022_gPINN} add gradients of the
residual to the loss. In these approaches, the weight, test function,
or derivative term acts on the residual as a whole. It does not
separately weight the components
\[
\partial_t u_\theta,\qquad
\partial_{x_i x_i}u_\theta,\qquad
v(u_\theta),\qquad
f_{\bar\theta}\mathbf 1_\omega,
\qquad\text{or}\qquad
f_{\bar\theta}u_\theta.
\]
This distinction is essential in the control setting, because the
unknown control is identified through the balance among these
individual components.

\paragraph{(E) Causal and curriculum-style PINNs.}
For time-dependent PDEs, several works improve the temporal structure
of training. Causal PINNs reweight the residual at later times so that
earlier times must be resolved first
\cite{WangSankaranPerdikaris2024_CausalPINN}, while time-marching and
sequence-to-sequence formulations train the network on growing time
windows. These methods address temporal causality and improve
training stability in evolutionary problems. However, they do not
distinguish among operator components at a fixed point in space-time,
and they have not been formulated specifically for the reconstruction
of controls.

\paragraph{(F) Domain decomposition methods.}
Conservative PINNs (cPINNs) \cite{JagtapKharazmiKarniadakis2020},
extended PINNs (XPINNs) \cite{JagtapKarniadakis2020_XPINN},
finite-basis PINNs (FBPINNs)
\cite{MoseleyMarkhamNissenMeyer2023_FBPINN}, and parallel PINN
variants split the space-time domain into subdomains, each with its
own neural representation. These methods are powerful for multiscale
problems, but they primarily address the geometry and decomposition of
the training domain rather than the structure of the PDE residual.
They do not provide an operator-decomposed adaptive mechanism for
control reconstruction.

\paragraph{(G) Architecture-level adaptivity and hard constraints.}
Locally adaptive activation functions (LAAF)
\cite{JagtapKawaguchiKarniadakis2020} and adaptive activation slopes
modify the neural representation itself. Hard-constrained PINNs
(hPINNs) \cite{LuPestourieYaoWangVerdugoJohnson2021_hPINN} embed
constraints directly into the network architecture, while
augmented-Lagrangian PINNs (AL-PINNs)
\cite{SonChoHwang2023_ALPINN} treat boundary and initial conditions
as constraints rather than penalty terms. Physics-constrained neural
networks (PCNNs), exact boundary lifting, and related formulations
share a similar philosophy. These methods are effective for enforcing
constraints, but their primary objective is not the numerical
reconstruction of controls, and they do not address the
operator-component imbalance that arises when an unknown control is
identified through a multiplicative or localized term.

\paragraph{(H) Other specialized variants.}
Bayesian PINNs (B-PINNs) \cite{YangMengKarniadakis2021_BPINN}
introduce uncertainty quantification; fractional PINNs (fPINNs)
\cite{PangLuKarniadakis2019_fPINN} handle fractional derivatives; and
PINNs combined with Kolmogorov--Arnold networks (PIKANs)
\cite{ToscanoOommenVargheseZouAhmadiKarniadakis2024_PIKAN} explore
alternative neural representations. These contributions are important
and complementary to ours, but they do not directly address the
control reconstruction problem considered in this paper.

\medskip

To the best of our knowledge, none of the adaptive PINN methodologies
described above has been systematically developed for the numerical
reconstruction of controls for nonlinear PDEs,
with internal or bilinear controls, treated here.

\subsection{The control reconstruction setting and the WeightedPINN
formulation}
\label{Sec:WPDescription}

The control setting presents a different and more delicate difficulty
from that of a forward PDE problem. In a forward problem, the primary
objective is to approximate the state \(u\). In a control
reconstruction problem, by contrast, both the state and the control
are unknown, and the control is identified only indirectly through the
state equation, the boundary and initial conditions, and the terminal
target. For example, in the internal control setting one has
\[
\partial_t u-\Delta u+v(u)=f\mathbf 1_\omega,
\]
whereas in the bilinear control setting one has
\[
\partial_t u-\Delta u+v(u)=fu.
\]
Thus, the control enters through a specific structural mechanism. In
the bilinear case, the control appears through the product \(fu\), so
errors in the state and errors in the control interact
multiplicatively. Consequently, a small averaged PDE residual does not
necessarily imply that the recovered control is accurate. Likewise, an
improved approximation of the state alone may still leave the control
poorly reconstructed.

Although PINN-based methodologies have recently been applied to
optimal deterministic control problems
\cite{ZhangLiuAllaDarbonKarniadakis2026,YongLuoSun2024,
GarciaCerveraKesslerPeriago2023,AllaBertagliaCalzola2025}, as well as
to stochastic control settings
\cite{BensoussanNguyenTranTu2026,BensoussanLiNguyenTranYamZhou2022},
our objective here is different. We develop a new PINN-based approach,
which we call \emph{WeightedPINNs}; see
Subsection~\ref{Sec:Illu} for an illustrative explanation. The method
is tailored to the computation of reliable and accurate numerical
approximations for the control reconstruction problem initiated in
\cite{MunchTrelat2022}.

The central idea is to introduce adaptive weights directly into the
PINN formulation at the level of the controlled PDE itself. Instead of
the standard residual
\[
\partial_t u_\theta
-\Delta u_\theta
+v(u_\theta)
-f_{\bar\theta}\mathbf 1_\omega,
\]
we consider weighted residuals of the form
\[
\phi_{\theta^1}\partial_t u_\theta
-\sum_{i=1}^d \phi_{\bar\theta^i}\partial_{x_i x_i}u_\theta
+\phi_{\theta^2} v(u_\theta)
-\phi_{\theta^3} f_{\bar\theta}\mathbf 1_\omega .
\]
For bilinear control, the control component is instead weighted as
\[
-\phi_{\theta^3} f_{\bar\theta}u_\theta .
\]
Thus, the time derivative, the directional diffusion terms, the
nonlinear response, and the control term are weighted independently.
The same principle is applied to the boundary, initial, and terminal
constraints. The resulting method is therefore not merely a scalar
reweighting of a PINN loss. It is an operator-decomposed and
control-aware adaptive residual method.

\subsection{Structural features of WeightedPINN}
\label{Sec:WPFeatures}

The position of WeightedPINN relative to existing PINN methodologies
has been discussed in Subsection~\ref{Sec:PriorPINN}. We now summarize
the structural features that, taken together, define the proposed
framework and justify its use for control reconstruction.

\medskip
\noindent
\textit{(i) Operator-level adaptivity.}
In a standard PINN formulation, the PDE residual is treated as a
single quantity, so that the time derivative, the diffusion operator,
the nonlinear term, and the control term are all penalized through the
same residual norm. This can lead to imbalance during training,
particularly in nonlinear or stiff regimes. WeightedPINN introduces
space--time-dependent weights at the level of each operator component,
so that the differential, nonlinear, and control contributions can be
rebalanced individually during training.

\medskip
\noindent
\textit{(ii) Control-sensitive residual metric.}
In control reconstruction, the unknown control is not observed
directly. It is identified through the state equation together with
the initial, boundary, and terminal constraints. Therefore, a small
averaged PDE residual does not necessarily imply an accurate
reconstructed control. This issue is especially delicate in bilinear
control problems, where the control appears through the product
\(f_{\bar\theta}u_\theta\) and where errors in the state and errors in
the control interact multiplicatively. WeightedPINN addresses this
difficulty by assigning an independent adaptive weight to the control
term, making the loss more sensitive to the part of the residual
through which the control is identified.

\medskip
\noindent
\textit{(iii) Directional and space--time weighting.}
The Laplacian is decomposed into its directional components
\(\partial_{x_i x_i}u_\theta\), and each component is weighted
independently. This produces an anisotropic residual metric that
distinguishes errors associated with different spatial directions.
Because the weights are also functions of space and time, the method
can emphasize regions where the approximation is less accurate or
where the control is more difficult to recover, for instance near the
control region, near transition layers, or in portions of the domain
where the state is small and control identification becomes
ill-conditioned.

\medskip
\noindent
\textit{(iv) Penalized min--max formulation consistent with the
original PINN.}
The state and control networks minimize the weighted residuals, while
the weight networks amplify poorly resolved components of the equation
and of the constraints. A quadratic penalty keeps the weights close to
\(1\), preventing the learned weighted operator from drifting too far
from the original PDE. In particular, when all weights are equal to
\(1\), the WeightedPINN loss reduces exactly to the standard PINN
loss. Thus, the standard PINN is contained as the limiting case in
which the penalty dominates, and the proposed method may be viewed as
a controlled adaptive perturbation of the standard PINN formulation.

\medskip
\noindent
\textit{(v) Structured treatment of control constraints.}
The weighted formulation is also applied to the boundary, initial, and
terminal conditions. This is essential in controllability problems,
because the terminal condition is not auxiliary data; it is the target
condition that defines the control problem. A neural pair
\((u_\theta,f_{\bar\theta})\) that satisfies the PDE residual but
fails to reach the prescribed terminal state cannot be regarded as a
reliable reconstructed control. By weighting both the dynamical
residual and the control constraints in a structured manner, the
method enforces the governing equation and the target condition more
effectively.

\medskip

Overall, the main contribution of this work is not simply the
introduction of adaptive weights into PINNs. Rather, we introduce an
operator-decomposed and control-aware WeightedPINN framework for
nonlinear PDE control reconstruction. The method is new at the
methodological level because the adaptive weights act on individual
differential, nonlinear, control, and constraint components. It is
also new in the control setting, because previous adaptive PINN
methodologies have not been systematically developed for the numerical
reconstruction of controls for nonlinear PDEs of the type considered
here.

The WeightedPINN method is introduced in
Section~\ref{sec:ConandPINN}, and its theoretical analysis is presented
in Section~\ref{Sec:Theory}.  The numerical results in
Section~\ref{Sec:Numerical} indicate that WeightedPINN improves over
the standard PINN baseline on the high-dimensional control benchmarks
considered in this work.  In dimension \(d=10\), the method produces
smaller residual-based testing errors for the internal and bilinear
control problems governed by both heat and wave equations.  The
improvement is most pronounced in the equation residual, where the
adaptive weighting mechanism frequently reduces the PDE error
substantially, while keeping the boundary and endpoint errors
comparable or smaller.

The lower-dimensional manufactured-solution benchmarks in
Section~\ref{Sec:NumericalOtherMethodFull} further show that
WeightedPINN attains the smallest mean reconstruction errors for the
state and the control among the methods tested.  The numerical evidence suggests that a principal advantage of
WeightedPINN, in the benchmarks considered here, appears in regimes
where the residual contains componentwise imbalances, such as singular
perturbations, anisotropic or variable-coefficient operators, and
control-identification mechanisms in which the unknown control enters
through a localized or multiplicative term.  This is consistent with the design of the method: different
components of the PDE residual are weighted separately, allowing the
training procedure to respond to componentwise residual imbalances
that are not explicitly treated by scalar loss reweighting, adaptive
sampling, or causal weighting.

	\section*{Acknowledgments}
	The authors would like to thank Emmanuel Tr\'elat for fruitful discussions on the topic.

	\section{Control problems, PINNs and WeightedPINNs}\label{sec:ConandPINN}
	
	We first recall the definition of a feedforward neural network mapping $\mathbb{R}^d$ to $\mathbb{R}$ with width $\mathcal{W}$ and depth $\mathcal{L}$. The width structure of the network is specified by a sequence $(\mathcal{W}_i)_{i=0}^{\mathcal{L}+1}$ such that $\mathcal{W}_0 = d$, $\mathcal{W}_{\mathcal{L}+1} = 1$, and \(\mathcal W := \max_{1\le i\le\mathcal L}\mathcal W_i\).

	Given parameters $\theta = (w_0, b_0, \dots, w_{\mathcal{L}}, b_{\mathcal{L}})$, where $w_i \in \mathbb{R}^{\mathcal{W}_{i+1} \times \mathcal{W}_i}$ and $b_i \in \mathbb{R}^{\mathcal{W}_{i+1}}$, we define the affine transformations
	\begin{align*}
		\ell_i(x) = w_i x + b_i, \quad i = 0, \dots, \mathcal{L}.
	\end{align*}
	A neural network $F_\theta$ with activation function $\sigma : \mathbb{R} \to \mathbb{R}$ is defined by
	\begin{align*}
		F_{\theta}(x) := \ell_{\mathcal{L}} \circ \sigma_{\mathcal{L}} \circ \ell_{\mathcal L-1}\circ \cdots \circ \sigma_{1} \circ \ell_0(x),
	\end{align*}
	where, for each $i = 1, \dots, \mathcal{L}$, the map $\sigma_i : \mathbb{R}^{\mathcal{W}_i} \to \mathbb{R}^{\mathcal{W}_i}$ denotes the componentwise application of $\sigma$, i.e.,
	\begin{align*}
		\sigma_{i}\big((x_1, \dots, x_{\mathcal{W}_i})^\top\big)
		= \big(\sigma(x_1), \dots, \sigma(x_{\mathcal{W}_i})\big)^\top.
	\end{align*}
	
	Suppose we aim to solve the partial differential equation (PDE)
	\begin{equation}\label{eq:general_PDE}
		\mathcal{D}u(x) = f(x), \quad \text{in } Q,
		\qquad 
		\mathcal{B}u(x) = g(x), \quad \text{on } \Gamma,
	\end{equation}
	where $x \in Q$ denotes the spatial variable, and may also include a temporal component $t$ when appropriate. 
	Here, $\mathcal{D}u = f$ represents the governing PDE, and $\mathcal{B}u = g$ represents the associated boundary and/or initial conditions. The sets $Q$ and $\Gamma$ denote the corresponding domain and boundary, respectively.
	
	In a standard PINN approach, we seek a neural network $u(x;\theta)$ that approximates the solution $u$ of \eqref{eq:general_PDE}. 
	The residuals associated with the PDE and the boundary/initial conditions are defined by
	\begin{equation*}
		\mathcal{R}_{Q}(u(x;\theta)) := \mathcal{D}u(x;\theta) - f(x), 
		\qquad 
		\mathcal{R}_{\Gamma}(u(x;\theta)) := \mathcal{B}u(x;\theta) - g(x).
	\end{equation*}
	
	The PDE can then be solved by seeking parameters $\theta$ that minimize the residual errors \cite{walton2022deep}:
	\begin{equation*} 
		\min_{\theta} \left\|\mathcal{R}_{Q}(u(x;\theta))\right\|_{Q}^{2} 
		+ \lambda \left\|\mathcal{R}_{\Gamma}(u(x;\theta))\right\|_{\Gamma}^{2},
	\end{equation*}
	where $\|\cdot\|$ typically denotes the $L^{2}$-norm, and $\lambda > 0$ is a weighting parameter. For example, this minimization problem can be written as
	\begin{equation}\label{StandardNET}
		\min_{\theta} \mathbb{E}_{x \in Q} \left[ \left|\mathcal{D}u(x;\theta) - f(x)\right|^{2} \right] 
		+ \lambda \, \mathbb{E}_{x \in \Gamma} \left[ \left|\mathcal{B}u(x;\theta) - g(x)\right|^{2} \right].
	\end{equation}
	
	In practice, the $L^2$-norm is often used for simplicity, although the PDE solution may require higher regularity. 
	
	Solving the minimization problem \eqref{StandardNET} yields the standard Physics-Informed Neural Network (PINN) formulation.

	Below, we consider control problems for semilinear heat and wave equations to illustrate our method. First, we formally state the control problems and then apply the PINN method to approximate the associated control. Next, our new WeightedPINN method introduces adaptability to the learning process.
	We consider both the internal control for the heat and wave equations and the bilinear control for these equations. We introduce the notation \(\mathcal{R}_Q^{\mathcal I},\mathcal{R}_\Gamma^{\mathcal I}\), with \(\mathcal{I}\in \{1,2,3,4\}\), to distinguish residuals corresponding to different PDE and control settings. Similarly, we denote by \(\bar{\mathcal{R}}_Q^{\mathcal I},\bar{\mathcal{R}}_\Gamma^{\mathcal I}\) for the residuals in WeightedPINN settings. We refer to Section \ref{Sec:internalsettings} and Section \ref{Sec:bilinearsettings} for the detailed definitions.

	\subsection{Internal control problems for heat and wave equations in high dimensions}~~\label{Sec:internalsettings}

	Let $\Omega \subset \mathbb{R}^d$ be a bounded domain with smooth boundary, let $\omega \subset \Omega$ be an open subset, and let $T>0$. We define
	\[
	Q_T := (0,T) \times \Omega, \quad
	q_T := (0,T) \times \omega, \quad
	\Sigma_T := (0,T) \times \partial\Omega,
	\]
	and
	\[
	\Gamma := \Sigma_T \cup (\{0,T\} \times \Omega).
	\]
	We denote by $\mathbf 1_\omega$ the characteristic function of $\omega$.

	For a given globally Lipschitz function $v:\mathbb{R}\to\mathbb{R}$, we consider the semilinear heat equation
	\begin{align}\label{Heat1}
		\partial_t u - \Delta u + v(u) &= f \mathbf 1_\omega \quad \text{in } Q_T, \nonumber\\
		u &= 0 \quad \text{on } \Sigma_T, \\
		u(0) &= u_0 \quad \text{in } \Omega, \nonumber
	\end{align}
	where $u_0 \in L^2(\Omega)$.
	
	Equation \eqref{Heat1} is said to be exactly controllable at time $T$ if, for any $u_0 \in L^2(\Omega)$ and $z_0 \in L^2(\Omega)$, there exists a control $f \in L^2(q_T)$ such that the corresponding solution satisfies $u(T)=z_0$.
	
	The existence of such controls is not guaranteed in general; see \cite{FernandezZuazua2000,fernandez-cara_guerrero_2006}. In this work, we assume that a control exists for \eqref{Heat1} for given initial and terminal conditions.
	
	\begin{Remark}
		In \cite{FernandezZuazua2000}, controllability is established with controls \(f \in L^\infty(q_T)\), whereas in \cite{fernandez-cara_guerrero_2006}, the control is taken in \(L^2(q_T)\). In Theorem \ref{thm:errorHeat}, we need higher regularity and adopt the setting \(f\in L^\infty(q_T)\). In Theorem \ref{thm:error_estHeat}, we only need \(f\in L^2(q_T)\).
	\end{Remark}
	
	We also consider the semilinear wave equation
	\begin{align}\label{Wave1}
		\partial_{tt} u - \Delta u + v(u) &= f \mathbf 1_\omega \quad \text{in } Q_T, \nonumber\\
		u &= 0 \quad \text{on } \Sigma_T, \\
		(u(0), \partial_t u(0)) &= (u_0,u_1) \quad \text{in } \Omega, \nonumber
	\end{align}
	where \(v:\mathbb R \to \mathbb R\) is globally Lipschitz and $(u_0,u_1)\in H^1(\Omega)\times L^2(\Omega)$.
	
	The wave equation \eqref{Wave1} is said to be exactly controllable at time $T$ if, for any initial data $(u_0,u_1)\in H^1(\Omega)\times L^2(\Omega)$ and target data $(z_0,z_1)\in H^1(\Omega)\times L^2(\Omega)$, there exists a control $f \in L^2(q_T)$ such that \((u(T), \partial_t u(T)) = (z_0,z_1).\)
	
	We assume the existence of such controls for \eqref{Wave1} under given initial and terminal conditions. We refer to \cite{zuazua2024exactcontrollabilitystabilizationwave,fu_yong_zhang_2007} for relevant conditions. Specifically, we discuss the conditions further in Remark \ref{rmk:con_wave}.
	
	\begin{itemize}
		\item \textbf{First approach: standard PINN method}

Hereafter, let \(u_\theta\) and \(f_{\bar{\theta}}\) denote neural networks with parameters \(\theta\) and \(\bar{\theta}\) that approximate, respectively, the solution and the control for the model equation.

		\textbf{Heat equation:} 
		We consider the control problem \eqref{Heat1} with initial condition \(u_0\) and terminal condition \(z_0\). 
		First, we recall the definition of $H^{1/2}$-norm on \(\Sigma_T\), which is
		\begin{align*}
			\|\varphi\|^2_{H^{1/2}(\Sigma_T)}:=\|\varphi\|^2_{L^2(\Sigma_T)}+[\varphi]^2_{H^{1/2}(\Sigma_T)},
		\end{align*}
		where
		\begin{align*}
			[\varphi]^2_{H^{1/2}(\Sigma_T)}:=\int_{\Sigma_T}\int_{\Sigma_T}\frac{(\varphi(x)-\varphi(y))^2}{\|x-y\|^{d+1}}dS_x dS_y.
		\end{align*}
		The PINN method is associated with the minimization formulation
		\begin{align*}
			\inf_{\theta,\bar\theta}\ &\frac{1}{|Q_T|}\int_{Q_T}
			\left( \partial_t u_{\theta} - \Delta u_{\theta} + v(u_{\theta}) - f_{\bar\theta} \mathbf 1_\omega \right)^2 dx\,dt\\
			&+\frac{\lambda}{|\Gamma|}
			\Big(
				\|u_\theta\|_{H^{1/2}(\Sigma_T)}^2
			+\|u_\theta(0,\cdot)-u_0\|_{L^2(\Omega)}^2
			+\|u_\theta(T,\cdot)-z_0\|_{L^2(\Omega)}^2
			\Big),
		\end{align*}
		where \(\lambda>0\) is a given constant, $|Q_T|$ and $|\Gamma|$ denote the $(d+1)$- and $d$-dimensional measures of $Q_T$ and $\Gamma$, respectively.

		\begin{Remark}\label{rmk:H1bound}
			From a theoretical perspective, even though Equation \eqref{Heat1} has homogeneous boundary conditions, the approximation \(u_\theta\) is non-homogeneous. Thus, \(u_\theta\) must possess sufficient regularity to approximate the solution of the heat equation.

			In \cite[Chapter 4, Section 15.5]{LionsMagenesVol2}, when we need a weak solution in \(L^2(0,T;H^1(\Omega))\), the trace space for parabolic equation is at least
			\[L^2(0,T;H^{1/2}(\partial \Omega))\cap H^{1/4}(0,T;L^2(\partial \Omega)).\]

			For the ease of the formulation, we utilized the stronger \(H^{1/2}\)-norm on the boundary in the PINN formulation, this norm controls the anisotropic norm on the trace space. In practice, however, we adopt only the \(L^2\)-norm for simplicity.
		\end{Remark}

		We define the equation residual by
		\begin{align*}
			\|\mathcal{R}_{Q}^{1}(u_{\theta},f_{\bar\theta})\|_{Q}^2
			:=\frac{1}{|Q_T|}\int_{Q_T}
			\left( \partial_t u_{\theta} - \Delta u_{\theta} + v(u_{\theta}) - f_{\bar\theta} \mathbf 1_\omega \right)^2 dx\,dt,
		\end{align*}
		and define the boundary residual by
		\begin{align*}
			\|\mathcal{R}_{\Gamma}^{1}(u_{\theta})\|_{\Gamma}^2
			:=\frac{1}{|\Gamma|}
			\Big(
				\|u_\theta\|_{H^{1/2}(\Sigma_T)}^2
			+\|u_\theta(0,\cdot)-u_0\|_{L^2(\Omega)}^2
			+\|u_\theta(T,\cdot)-z_0\|_{L^2(\Omega)}^2
			\Big).
		\end{align*}
		
		The standard PINN formulation for the internal control on the heat equation, with \(\lambda > 0\), can be written in the following compact form:
		\begin{align*}
			\inf_{\theta,\bar{\theta}}
			\|\mathcal{R}_{Q}^{1}(u_{\theta},f_{\bar{\theta}})\|_{Q}^2
			+\lambda \|\mathcal{R}_{\Gamma}^{1}(u_{\theta})\|_{\Gamma}^2.
		\end{align*}

		\textbf{Wave equation:}
		We now consider the control problem in \eqref{Wave1} with initial condition \( (u_0,u_1)\) and terminal condition \( (z_0,z_1)\). The definition of $H^{3/2}$-norm on \(\Sigma_T\) is as follow
		\begin{align*}
			\|\varphi\|^2_{H^{3/2}(\Sigma_T)}:=\|\varphi\|^2_{L^2(\Sigma_T)}+\|\partial_t \varphi\|_{H^{1/2}(\Sigma_T)}^2+\sum_{i=1}^{d-1}\bigl\|\frac{\partial\varphi}{\partial \sigma_i}\bigr\|_{H^{1/2}(\Sigma_T)}^2,
		\end{align*}
		where \( (\sigma_1,\dots,\sigma_{d-1})\) is a smooth local coordinate of \(\partial \Omega\).

		The PINN method in this control problem is associated with the minimization formulation
		\begin{align*}
			\inf_{\theta,\bar\theta}\ &\frac{1}{|Q_T|}\int_{Q_T}
			\left( \partial_{tt} u_{\theta} - \Delta u_{\theta} + v(u_{\theta}) - f_{\bar\theta} \mathbf 1_\omega \right)^2 dx\,dt\\
			&+\frac{\lambda}{|\Gamma|}
			\Big(
				\|u_\theta\|_{H^{3/2}(\Sigma_T)}^2
			+\|(u_\theta(0,\cdot),\partial_t u_\theta(0,\cdot))-(u_0,u_1)\|_{H^1(\Omega)\times L^2(\Omega)}^2 \nonumber\\
			&\qquad\quad
			+\|(u_\theta(T,\cdot),\partial_t u_\theta(T,\cdot))-(z_0,z_1)\|_{H^1(\Omega)\times L^2(\Omega)}^2
			\Big),
		\end{align*}
		where \(\lambda>0\) is a given constant.

\begin{Remark}\label{rmk:H1initer}
	In the same spirit as Remark \ref{rmk:H1bound}, one would expect the boundary conditions has higher regularity than \(L^2\). For wave equation, we choose \(H^{3/2}\)-norm for boundary conditions.

	Furthermore, the initial and terminal data \(u_0\) and \(z_0\) belong to \(H^1(\Omega)\). Therefore, from a theoretical standpoint, the approximation should be carried out in the \(H^1\)-norm.

	In practice, however, we adopt the \(L^2\)-norm for boundary, initial and terminal conditions.
		\end{Remark}
		
		The equation residual is defined by
		\begin{align*}
			\|\mathcal{R}_{Q}^{2}(u_{\theta},f_{\bar\theta})\|_{Q}^2
			:=\frac{1}{|Q_T|}\int_{Q_T}
			\left( \partial_{tt} u_{\theta} - \Delta u_{\theta} + v(u_{\theta}) - f_{\bar\theta} \mathbf 1_\omega \right)^2 dx\,dt,
		\end{align*}
		and the boundary residual is defined by
		\begin{align*}
			\|\mathcal{R}_{\Gamma}^{2}(u_{\theta})\|_{\Gamma}^2
			&:=\frac{1}{|\Gamma|}
			\Big(
				\|u_\theta\|_{H^{3/2}(\Sigma_T)}^2
			+\|(u_\theta(0,\cdot),\partial_t u_\theta(0,\cdot))-(u_0,u_1)\|_{H^1(\Omega)\times L^2(\Omega)}^2 \nonumber\\
			&\qquad\quad
			+\|(u_\theta(T,\cdot),\partial_t u_\theta(T,\cdot))-(z_0,z_1)\|_{H^1(\Omega)\times L^2(\Omega)}^2
			\Big).
		\end{align*}

		The standard PINN formulation for the internal control on the wave equation, with \(\lambda > 0\), can be written in the following compact form:
		\begin{align*}
			\inf_{\theta,\bar\theta}
			\|\mathcal{R}_{Q}^{2}(u_{\theta},f_{\bar\theta})\|_{Q}^2
			+\lambda \|\mathcal{R}_{\Gamma}^{2}(u_{\theta})\|_{\Gamma}^2.
		\end{align*}

		\item \textbf{Second approach: WeightedPINN method}
		
		We introduce additional neural networks $\phi_{\theta}:\mathbb R^{d+1}\to \mathbb R$ that act as adaptive weights. We impose on the weight networks a uniform boundedness assumption, which will be specified later in \ref{con:weight}.

		\textbf{Heat equation:} 
		We again consider the control problem \eqref{Heat1} with initial condition \(u_0\) and terminal condition \(z_0\). The WeightedPINN method is associated with the following min--max formulation:
		\begin{align}\label{eq:lossheat}
			\inf_{\theta,\bar\theta}\,\sup_{(\theta^i)_{i=1}^{8},(\bar\theta^i)_{i=1}^d}\ 
			&\frac{1}{|Q_T|}\int_{Q_T}
			\Big[
			\phi_{\theta^1}\partial_t u_{\theta}
			- \sum_{i=1}^{d}\phi_{\bar{\theta}^i}\partial_{x_i x_i} u_{\theta}
			+ \phi_{\theta^{2}} v(u_{\theta})
			- \phi_{\theta^{3}} f_{\bar\theta} \mathbf 1_\omega
			\Big]^2 dx\,dt\\
			&+\frac{\lambda}{|\Gamma|}
			\Big(
				\|\phi_{\theta^4}u_\theta\|_{H^{1/2}(\Sigma_T)}^2
			+ \|\phi_{\theta^5}(0,\cdot)u_{\theta}(0,\cdot)-\phi_{\theta^6}(0,\cdot)u_0\|_{L^2(\Omega)}^2\nonumber\\
			&\qquad\quad
			+ \|\phi_{\theta^7}(T,\cdot)u_{\theta}(T,\cdot)-\phi_{\theta^{8}}(T,\cdot)z_0\|_{L^2(\Omega)}^2
			\Big)\nonumber\\
			&-\rho\Big(
			\sum_{i=1}^{8}\|\phi_{\theta^i}-1\|^2_{Q^i}
			+\frac{1}{|Q_T|}\sum_{i=1}^{d}\|\phi_{\bar{\theta}^i}-1\|^2_{L^2(Q_T)}	
			\Big)\nonumber,
		\end{align}
		where \(\lambda,\rho>0\) are given constants, \(\|\cdot\|^2_{Q^i}=\frac{\|\cdot\|^2_{L^2(Q_T)}}{|Q_T|}\) for \(i=1,2,3\), \(\|\cdot\|^2_{Q^4}=\frac{\|\cdot\|^2_{L^2(\Sigma_T)}}{|\Gamma|}\), \(\|\phi\|^2_{Q^i}=\frac{\|\phi(0,\cdot)\|^2_{L^2(\Omega)}}{|\Gamma|}\) for \(i=5,6\), \(\|\phi\|^2_{Q^i}=\frac{\|\phi(T,\cdot)\|^2_{L^2(\Omega)}}{|\Gamma|}\) for \(i=7,8\).

		In the proposed formulation, the additional neural networks serve as adaptive weights, designed to maximize discrepancies in the equation and boundary residuals, thereby enhancing the effectiveness of the minimization. To regulate these weights, we introduce additional \(L^2\)-norm penalization terms. These terms ensure that the weights remain controlled and do not deviate significantly from the constant function \(1\).
		
		\begin{Remark}\label{rmk:weightwithtime}
			The weights \(\phi_{\theta^5}\) and \(\phi_{\theta^6}\) can be viewed as functions from \(\mathbb{R}^d\) to \(\mathbb{R}\), i.e., without explicit dependence on \(t\) at \(t=0\). For simplicity, we employ the same network architecture as for the other weight functions, namely neural networks defined on \(\mathbb{R}^{d+1}\). These networks are then restricted to \(t=0\).
			
			The weights \(\phi_{\theta^7}, \phi_{\theta^{8}}\) are treated analogously at \(t=T\).
		\end{Remark}

		As in the standard PINN framework, we adopt a compact notation for the WeightedPINN formulation by introducing weighted residuals. For readability, we omit their explicit dependence on the neural network parameters.
		
		We define the weighted residual for the heat equation by
		\begin{align}\label{eq:barRQ1}
			\|\bar{\mathcal{R}}_{Q}^{1}\|_{Q}^2
			:=\frac{1}{|Q_T|}\int_{Q_T}
			\Big[
			\phi_{\theta^1}\partial_t u_{\theta}
			- \sum_{i=1}^{d}\phi_{\bar{\theta}^i}\partial_{x_i x_i} u_{\theta}
			+ \phi_{\theta^{2}} v(u_{\theta})
			- \phi_{\theta^{3}} f_{\bar\theta} \mathbf 1_\omega
			\Big]^2 dx\,dt.
		\end{align}
		
		The weighted boundary residual is
		\begin{align}\label{eq:barRG1}
			\|\bar{\mathcal{R}}_{\Gamma}^{1}\|_{\Gamma}^2
			&:=\frac{1}{|\Gamma|}
			\Big(
				\|\phi_{\theta^4}u_\theta\|_{H^{1/2}(\Sigma_T)}^2
			+ \|\phi_{\theta^5}(0,\cdot)u_{\theta}(0,\cdot)-\phi_{\theta^6}(0,\cdot)u_0\|_{L^2(\Omega)}^2\\
			&\qquad\quad
			+ \|\phi_{\theta^7}(T,\cdot)u_{\theta}(T,\cdot)-\phi_{\theta^{8}}(T,\cdot)z_0\|_{L^2(\Omega)}^2
			\Big).\nonumber
		\end{align}
		
		In WeightedPINN approach, we also define the auxiliary neural residual by
		\begin{align}\label{eq:RNN1}
			\|\mathcal{R}_{NN}^{1}\|_{NN}^2
			:=\sum_{i=1}^{8}\|\phi_{\theta^i}-1\|^2_{Q^i}
			+\frac{1}{|Q_T|}\sum_{i=1}^{d}\|\phi_{\bar{\theta}^i}-1\|^2_{L^2(Q_T)}.
		\end{align}
		Here, we recall \(\|\cdot\|^2_{Q^i}=\frac{\|\cdot\|^2_{L^2(Q_T)}}{|Q_T|}\) for \(i=1,2,3\), \(\|\cdot\|^2_{Q^4}=\frac{\|\cdot\|^2_{L^2(\Sigma_T)}}{|\Gamma|}\), \(\|\phi\|^2_{Q^i}=\frac{\|\phi(0,\cdot)\|^2_{L^2(\Omega)}}{|\Gamma|}\) for \(i=5,6\), \(\|\phi\|^2_{Q^i}=\frac{\|\phi(T,\cdot)\|^2_{L^2(\Omega)}}{|\Gamma|}\) for \(i=7,8\). 
		
		The compact form for WeightedPINN formulation for the internal control on the heat equation, with \(\lambda,\rho>0\), is given by
		\begin{align*}
			\inf_{\theta,\bar\theta}\,\sup_{(\theta^i)_{i=1}^8,(\bar\theta^i)_{i=1}^d}
			\;\|\bar{\mathcal{R}}_{Q}^{1}\|_{Q}^2
			+\lambda \|\bar{\mathcal{R}}_{\Gamma}^{1}\|_{\Gamma}^2
			-\rho \|\mathcal{R}_{NN}^{1}\|_{NN}^2.
		\end{align*}
		
		\textbf{Wave equation:} 
		We extend the idea of WeightedPINN to control problems of the wave equation. Let us consider \eqref{Wave1} with initial condition \( (u_0,u_1)\) and \( (z_0,z_1)\).
		The WeightedPINN is associated with the min--max formulation
		\begin{align}\label{eq:losswave}
			\inf_{\theta,\bar\theta}\,\sup_{(\theta^i)_{i=1}^{12},(\bar\theta^i)_{i=1}^d,(\tilde\theta^i)_{i=1}^d,(\hat\theta^i)_{i=1}^{2d},(\check\theta^i)_{i=1}^{2d}}\ 
			&\frac{1}{|Q_T|}\int_{Q_T}
			\Big[
			\phi_{\theta^1}\partial_{tt}u_{\theta}
			- \sum_{i=1}^{d}\phi_{\bar{\theta}^i}\partial_{x_i x_i} u_{\theta}
			+ \phi_{\theta^{2}} v(u_{\theta})
			- \phi_{\theta^{3}} f_{\bar\theta} \mathbf 1_\omega
			\Big]^2 dx\,dt\\
			&+\frac{\lambda}{|\Gamma|}
			\Big(
			\|\phi_{\theta^4}u_\theta\|_{L^2(\Sigma_T)}^2
			+\|\phi_{\tilde\theta^1}\partial_t u_\theta\|^2_{H^{1/2}(\Sigma_T)}
			+\sum_{i=2}^{d}\bigl\|\phi_{\tilde\theta^i}\frac{\partial u_\theta}{\partial \sigma_{i-1}}\bigr\|^2_{H^{1/2}(\Sigma_T)}\nonumber\\
			&\qquad\quad
			+\|\phi_{\theta^5}(0,\cdot)u_{\theta}(0,\cdot)-\phi_{\theta^6}(0,\cdot)u_0\|^2_{L^2(\Omega)}\nonumber\\
			&\qquad\quad
			+\|\phi_{\theta^{9}}(0,\cdot)\partial_t u_{\theta}(0,\cdot)-\phi_{\theta^{10}}(0,\cdot)u_1\|^2_{L^2(\Omega)}\nonumber\\
			&\qquad\quad
			+\|\phi_{\theta^7}(T,\cdot)u_{\theta}(T,\cdot)-\phi_{\theta^{8}}(T,\cdot)z_0\|^2_{L^2(\Omega)}\nonumber\\
			&\qquad\quad+\|\phi_{\theta^{11}}(T,\cdot)\partial_t u_{\theta}(T,\cdot)-\phi_{\theta^{12}}(T,\cdot)z_1\|^2_{L^2(\Omega)}\nonumber\\
			&\qquad\quad
			+\sum_{i=1}^{d}\|\phi_{\hat{\theta}^{2i-1}}(0,\cdot)\partial_{x_i}u_{\theta}(0,\cdot)-\phi_{\hat{\theta}^{2i}}(0,\cdot)\partial_{x_i}u_0\|^2_{L^2(\Omega)}\nonumber\\
			&\qquad\quad
			+\sum_{i=1}^{d}\|\phi_{\check{\theta}^{2i-1}}(T,\cdot)\partial_{x_i}u_{\theta}(T,\cdot)-\phi_{\check{\theta}^{2i}}(T,\cdot)\partial_{x_i}z_0\|^2_{L^2(\Omega)}\Big)\nonumber\\
			&-\rho\Big(
			\sum_{i=1}^{12}\|\phi_{\theta^i}-1\|^2_{Q^i}\\
			&\qquad\quad
			+\frac{1}{|Q_T|}\sum_{i=1}^{d}\|\phi_{\bar{\theta}^i}-1\|^2_{L^2(Q_T)}+\frac{1}{|\Gamma|}\sum_{i=1}^{d}\|\phi_{\tilde{\theta}^i}-1\|^2_{L^2(\Sigma_T)}\nonumber\\
			&\qquad\quad
			+\frac{1}{|\Gamma|}\sum_{i=1}^{2d}\|\phi_{\hat{\theta}^i}(0,\cdot)-1\|^2_{L^2(\Omega)}
			+\frac{1}{|\Gamma|}\sum_{i=1}^{2d}\|\phi_{\check{\theta}^i}(T,\cdot)-1\|^2_{L^2(\Omega)}\nonumber
			\Big),
		\end{align}
		where \(\lambda,\rho>0\) are given constants, \((\sigma_1,\dots,\sigma_{d-1})\) is a smooth local coordinate of \(\partial\Omega\), \(\|\cdot\|^2_{Q^i}=\frac{\|\cdot\|^2_{L^2(Q_T)}}{|Q_T|}\) for \(i=1,2,3\), \(\|\cdot\|^2_{Q^4}=\frac{\|\cdot\|^2_{L^2(\Sigma_T)}}{|\Gamma|}\), \(\|\phi\|^2_{Q^i}=\frac{\|\phi(0,\cdot)\|^2_{L^2(\Omega)}}{|\Gamma|}\) for \(i=5,6,9,10\), \(\|\phi\|^2_{Q^i}=\frac{\|\phi(T,\cdot)\|^2_{L^2(\Omega)}}{|\Gamma|}\) for \(i=7,8,11,12\). 
		
		Similar to the control problem for heat equation, as in Remark \ref{rmk:weightwithtime}, we take \(\phi_{\theta^5},\phi_{\theta^6},\phi_{\theta^{9}},\phi_{\theta^{10}}\) and \(\phi_{\hat\theta^i}\) to be neural networks from \(\mathbb R^{d+1}\) to \(\mathbb R\),  restricted to \(t = 0\). We treat  \(\phi_{\theta^7},\phi_{\theta^{8}},\phi_{\theta^{11}},\phi_{\theta^{12}}\) and \(\phi_{\check\theta^i}\) analogously, with restriction to \(t = T\).
		
		For the internal control on the wave equation, the weighted  residuals are defined by
		\begin{align}\label{eq:barRQ2}
			\|\bar{\mathcal{R}}_{Q}^{2}\|_{Q}^2
			:=\frac{1}{|Q_T|}\int_{Q_T}
			\Big[
			\phi_{\theta^1}\partial_{tt}u_{\theta}
			- \sum_{i=1}^{d}\phi_{\bar{\theta}^i}\partial_{x_i x_i} u_{\theta}
			+ \phi_{\theta^{2}} v(u_{\theta})
			- \phi_{\theta^{3}} f_{\bar\theta} \mathbf 1_\omega
			\Big]^2 dx\,dt,
		\end{align}
		and
		\begin{align}\label{eq:barRG2}
			\|\bar{\mathcal{R}}_{\Gamma}^2\|_{\Gamma}^{2}
			&:=\frac{1}{|\Gamma|}
			\Big(
			\|\phi_{\theta^4}u_\theta\|_{L^2(\Sigma_T)}^2
			+\|\phi_{\tilde\theta^0}\partial_t u_\theta\|^2_{H^{1/2}(\Sigma_T)}
			+\sum_{i=1}^{d}\|\phi_{\tilde\theta^i}\partial_{x_i}u_\theta\|^2_{H^{1/2}(\Sigma_T)}\\
			&\qquad\quad
			+\|\phi_{\theta^5}(0,\cdot)u_{\theta}(0,\cdot)-\phi_{\theta^6}(0,\cdot)u_0\|^2_{L^2(\Omega)}
			+\|\phi_{\theta^{9}}(0,\cdot)\partial_t u_{\theta}(0,\cdot)-\phi_{\theta^{10}}(0,\cdot)u_1\|^2_{L^2(\Omega)}\nonumber\\
			&\qquad\quad
			+\|\phi_{\theta^7}(T,\cdot)u_{\theta}(T,\cdot)-\phi_{\theta^{8}}(T,\cdot)z_0\|^2_{L^2(\Omega)}
			+\|\phi_{\theta^{11}}(T,\cdot)\partial_t u_{\theta}(T,\cdot)-\phi_{\theta^{12}}(T,\cdot)z_1\|^2_{L^2(\Omega)}\nonumber\\
			&\qquad\quad
			+\sum_{i=1}^{d}\|\phi_{\hat{\theta}^{2i-1}}(0,\cdot)\partial_{x_i}u_{\theta}(0,\cdot)-\phi_{\hat{\theta}^{2i}}(0,\cdot)\partial_{x_i}u_0\|^2_{L^2(\Omega)}\nonumber\\
			&\qquad\quad
			+\sum_{i=1}^{d}\|\phi_{\check{\theta}^{2i-1}}(T,\cdot)\partial_{x_i}u_{\theta}(T,\cdot)-\phi_{\check{\theta}^{2i}}(T,\cdot)\partial_{x_i}z_0\|^2_{L^2(\Omega)}\Big).\nonumber
		\end{align}
		
		We also define the neural residual by
		\begin{align}\label{eq:RNN2}
			\|\mathcal{R}_{NN}^{2}\|_{NN}^2
			&:=\sum_{i=1}^{12}\|\phi_{\theta^i}-1\|^2_{Q^i}
			+\frac{1}{|Q_T|}\sum_{i=1}^{d}\|\phi_{\bar{\theta}^i}-1\|^2_{L^2(Q_T)}
			+\frac{1}{|\Gamma|}\sum_{i=1}^{d}\|\phi_{\tilde{\theta}^i}-1\|^2_{L^2(\Sigma_T)}\\
			&\quad+\frac{1}{|\Gamma|}\sum_{i=1}^{2d}\|\phi_{\hat{\theta}^i}(0,\cdot)-1\|^2_{L^2(\Omega)}
			+\frac{1}{|\Gamma|}\sum_{i=1}^{2d}\|\phi_{\check{\theta}^i}(T,\cdot)-1\|^2_{L^2(\Omega)},\nonumber
		\end{align}
		where \(\|\cdot\|^2_{Q^i}=\frac{\|\cdot\|^2_{L^2(Q_T)}}{|Q_T|}\) for \(i=1,2,3\), \(\|\cdot\|^2_{Q^4}=\frac{\|\cdot\|^2_{L^2(\Sigma_T)}}{|\Gamma|}\), \(\|\phi\|^2_{Q^i}=\frac{\|\phi(0,\cdot)\|^2_{L^2(\Omega)}}{|\Gamma|}\) for \(i=5,6,9,10\), \(\|\phi\|^2_{Q^i}=\frac{\|\phi(T,\cdot)\|^2_{L^2(\Omega)}}{|\Gamma|}\) for \(i=7,8,11,12\).
		
		Finally, for given parameters \(\lambda,\rho>0\), the WeightedPINN approach for the wave equation has the compact form
		\begin{align*}
			\inf_{\theta,\bar{\theta}}\,
			\sup_{(\theta^i)_{i=1}^{12},(\bar\theta^i)_{i=1}^d,(\tilde\theta^i)_{i=1}^{d},(\hat\theta^i)_{i=1}^{2d},(\check\theta^i)_{i=1}^{2d}}
			\;\|\bar{\mathcal{R}}_Q^{2}\|_{Q}^2+\lambda\|\bar{\mathcal{R}}_{\Gamma}^{2}\|_{\Gamma}^2-\rho\|\mathcal{R}_{NN}^{2}\|_{NN}^2.
		\end{align*}

	\end{itemize}

	\subsection{Bilinear control problems for heat and wave equations in high dimensions}~~\label{Sec:bilinearsettings}
	
	We use the same notations for the spatial domains and boundaries as in the internal control setting. For completeness, we recall them below.
	
	Let $\Omega \subset \mathbb{R}^d$ be a bounded domain with smooth boundary, and let $T>0$. We define
	\[
	Q_T := (0,T) \times \Omega, \quad
	\Sigma_T := (0,T) \times \partial\Omega, \quad
	\Gamma := \Sigma_T \cup (\{0,T\} \times \Omega).
	\]
	
	Let \(v:\mathbb R\to\mathbb R\) be a globally Lipschitz function. We consider the heat equation with bilinear control:
	\begin{align}\label{Heat1b}
		\partial_t u - \Delta u + v(u) &= f u \quad \text{in } Q_T, \nonumber\\
		u &= 0 \quad \text{on } \Sigma_T, \\
		u(0) &= u_0 \quad \text{in } \Omega, \nonumber
	\end{align}
	
	We say that \eqref{Heat1b} is bilinearly controllable at time $T$ from $u_0 \in L^2(\Omega)$ to a target $z_0 \in L^2(\Omega)$ if there exists a control $f \in L^\infty(Q_T)$ such that the corresponding solution satisfies $u(T)=z_0$. Sufficient conditions for such controllability are given in \cite{Lin2006}.
	
	\begin{Remark}
		The case of homogeneous boundary conditions is known to lead to a loss of controllability for the bilinear heat equation in general; see \cite{Lin2006}. Therefore, we assume the existence of a control associated with the prescribed initial and terminal data \(u_0\) and \(z_0\).
	\end{Remark}
	
	We next consider the wave equation with the globally Lipschitz function \(v:\mathbb R \to \mathbb R\) and bilinear control:
	\begin{align}\label{Wave1b}
		\partial_{tt} u - \Delta u + v(u) &= f u \quad \text{in } Q_T, \nonumber\\
		u &= 0 \quad \text{on } \Sigma_T, \\
		(u(0), \partial_t u(0)) &= (u_0,u_1) \quad \text{in } \Omega. \nonumber
	\end{align}
	
	We say that \eqref{Wave1b} is bilinearly controllable at time $T$ from $(u_0,u_1) \in H^1(\Omega)\times L^2(\Omega)$ to $(z_0,z_1) \in H^1(\Omega)\times L^2(\Omega)$ if there exists a control $f \in L^\infty(Q_T)$ such that the solution satisfies \((u(T), \partial_t u(T)) = (z_0,z_1).\)
	Relevant results on bilinear controllability of the wave equation can be found in \cite{Beauchard2011Local,Ouzahra2013}.

	\begin{itemize}
		\item \textbf{First approach: standard PINN method}
		
		\textbf{Heat equation:} 
		The standard PINN formulation of bilinear control on the heat equation is given by
		\begin{align*}
			\inf_{\theta,\bar\theta}\ &\frac{1}{|Q_T|}\int_{Q_T}
			\left( \partial_t u_{\theta} - \Delta u_{\theta} + v(u_{\theta}) - f_{\bar\theta}  u_\theta \right)^2 dx\,dt\\
			&+\frac{\lambda}{|\Gamma|}
			\Big(
				\|u_\theta\|_{H^{1/2}(\Sigma_T)}^2
			+\|u_\theta(0,\cdot)-u_0\|_{L^2(\Omega)}^2
			+\|u_\theta(T,\cdot)-z_0\|_{L^2(\Omega)}^2
			\Big),
		\end{align*}
		where \(\lambda>0\) is a given constant.
		We recall that $|Q_T|$ and $|\Gamma|$ denote the $(d+1)$- and $d$-dimensional measures of $Q_T$ and $\Gamma$, respectively.
		
		The equation residual is given by
		\begin{align*}
			\|\mathcal{R}_{Q}^{3}(u_{\theta},f_{\bar\theta})\|_{Q}^2
			:=\frac{1}{|Q_T|}\int_{Q_T}
			\left( \partial_t u_{\theta} - \Delta u_{\theta} + v(u_\theta) - f_{\bar\theta} u_{\theta} \right)^2 dx\,dt,
		\end{align*}
		and the boundary residual is
		\begin{align*}
			\|\mathcal{R}_{\Gamma}^{3}(u_{\theta})\|_{\Gamma}^2
			:=\frac{1}{|\Gamma|}
			\Big(
				\|u_\theta\|_{H^{1/2}(\Sigma_T)}^2
			+\|u_\theta(0,\cdot)-u_0\|_{L^2(\Omega)}^2
			+\|u_\theta(T,\cdot)-z_0\|_{L^2(\Omega)}^2
			\Big).
		\end{align*}
		
		The standard training formulation for bilinear controls on the heat equation, with $\lambda>0$, is given in compact form by
		\begin{align*}
			\inf_{\theta,\bar\theta}
			\|\mathcal{R}_{Q}^{3}(u_{\theta},f_{\bar\theta})\|_{Q}^2
			+\lambda \|\mathcal{R}_{\Gamma}^{3}(u_{\theta})\|_{\Gamma}^2.
		\end{align*}

		\textbf{Wave equation:} 
		Similar to the heat equation, the PINN formulation for bilinear control on the wave equation is given by
		\begin{align*}
			\inf_{\theta,\bar\theta}\ &\frac{1}{|Q_T|}\int_{Q_T}
			\left( \partial_{tt} u_{\theta} - \Delta u_{\theta} + v(u_{\theta}) - f_{\bar\theta} u_\theta \right)^2 dx\,dt\\
			&+\frac{\lambda}{|\Gamma|}
			\Big(
				\|u_\theta\|_{H^{3/2}(\Sigma_T)}^2
			+\|(u_\theta(0,\cdot),\partial_t u_\theta(0,\cdot))-(u_0,u_1)\|_{H^1(\Omega)\times L^2(\Omega)}^2 \nonumber\\
			&\qquad\quad
			+\|(u_\theta(T,\cdot),\partial_t u_\theta(T,\cdot))-(z_0,z_1)\|_{H^1(\Omega)\times L^2(\Omega)}^2
			\Big),
		\end{align*}
		where \(\lambda>0\) is a given constant.
		
		The residuals are defined as follows:
		\begin{align*}
			\|\mathcal{R}_{Q}^{4}(u_{\theta},f_{\bar\theta})\|_{Q}^2
			&:=\frac{1}{|Q_T|}\int_{Q_T}
			\left( \partial_{tt} u_{\theta} - \Delta u_{\theta} + v(u_\theta) - f_{\bar\theta} u_{\theta} \right)^2 dx\,dt,
		\end{align*}
		and
		\begin{align*}
			\|\mathcal{R}_{\Gamma}^{4}(u_{\theta})\|_{\Gamma}^2
			&:=\frac{1}{|\Gamma|}
			\Big(
				\|u_\theta\|_{H^{3/2}(\Sigma_T)}^2
			+\|(u_\theta(0,\cdot),\partial_t u_\theta(0,\cdot))-(u_0,u_1)\|_{H^1(\Omega)\times L^2(\Omega)}^2 \\
			&\qquad\quad
			+\|(u_\theta(T,\cdot),\partial_t u_\theta(T,\cdot))-(z_0,z_1)\|_{H^1(\Omega)\times L^2(\Omega)}^2
			\Big).\nonumber
		\end{align*}
		
		The compact form of PINN formulation for bilinear control on the wave equation is given by
		\begin{align*}
			\inf_{\theta,\bar\theta}
			\|\mathcal{R}_{Q}^{4}(u_{\theta},f_{\bar\theta})\|_{Q}^2
			+\lambda \|\mathcal{R}_{\Gamma}^{4}(u_{\theta})\|_{\Gamma}^2.
		\end{align*}

		\item {\bf Second approach: WeightedPINN method}. 
		We further define a weighted formulation of the residual errors for the heat and wave equations in the context of bilinear control problems. Similar to internal control settings, for readability, we omit the dependence of the residuals on the parameters of the related neural networks.
		
		{\bf Heat equation.} 
		We consider the min--max formulation that defines the WeightedPINN for bilinear control in \eqref{Heat1b}:
		\begin{align*}
			\inf_{\theta,\bar\theta}\,\sup_{(\theta^i)_{i=1}^8,(\bar\theta^i)_{i=1}^d}\ 
			&\frac{1}{|Q_T|}\int_{Q_T}
			\Big[
			\phi_{\theta^1}\partial_t u_{\theta}
			- \sum_{i=1}^{d}\phi_{\bar{\theta}^i}\partial_{x_i x_i} u_{\theta}
			+ \phi_{\theta^{2}} v(u_{\theta})
			- \phi_{\theta^{3}} f_{\bar\theta} u_\theta
			\Big]^2 dx\,dt\\
			&+\frac{\lambda}{|\Gamma|}
			\Big(
				\|\phi_{\theta^4}u_\theta\|_{H^{1/2}(\Sigma_T)}^2
			+ \|\phi_{\theta^5}(0,\cdot)u_{\theta}(0,\cdot)-\phi_{\theta^6}(0,\cdot)u_0\|_{L^2(\Omega)}^2\\
			&\qquad\quad
			+ \|\phi_{\theta^7}(T,\cdot)u_{\theta}(T,\cdot)-\phi_{\theta^8}(T,\cdot)z_0\|_{L^2(\Omega)}^2
			\Big)\\
			&-\rho\Big(
			\sum_{i=1}^{8}\|\phi_{\theta^i}-1\|^2_{Q^i}
			+\frac{1}{|Q_T|}\sum_{i=1}^{d}\|\phi_{\bar{\theta}^i}-1\|^2_{L^2(Q_T)}
			\Big),
		\end{align*}
where \(\lambda,\rho>0\) are given constants, \(\|\cdot\|^2_{Q^i}=\frac{\|\cdot\|^2_{L^2(Q_T)}}{|Q_T|}\) for \(i=1,2,3\), \(\|\cdot\|^2_{Q^4}=\frac{\|\cdot\|^2_{L^2(\Sigma_T)}}{|\Gamma|}\), \(\|\phi\|^2_{Q^i}=\frac{\|\phi(0,\cdot)\|^2_{L^2(\Omega)}}{|\Gamma|}\) for \(i=5,6\), \(\|\phi\|^2_{Q^i}=\frac{\|\phi(T,\cdot)\|^2_{L^2(\Omega)}}{|\Gamma|}\) for \(i=7,8\).

		To shorten the formulation, we define
		\begin{align}\label{eq:barRQ3}
			\|\bar{\mathcal{R}}_Q^3\|_Q^2
			:=\frac{1}{|Q_T|}\int_{Q_T}\bigg[
			\phi_{\theta^1}\partial_{t}u_{\theta}
			- \sum_{i=1}^{d}\phi_{\bar{\theta}^{i}}\partial_{x_ix_i} u_{\theta}
			+ \phi_{\theta^2}v(u_\theta)
			- \phi_{\theta^{3}}f_{\bar\theta} u_{\theta}
			\bigg]^2\,dx\,dt.
		\end{align}
		
		Next, we define the weighted boundary residual:
		\begin{align}\label{eq:barRG3}
			\|\bar{\mathcal{R}}_{\Gamma}^3\|_{\Gamma}^2
			&:=\frac{1}{|\Gamma|}
			\Big(
				\|\phi_{\theta^4}u_\theta\|_{H^{1/2}(\Sigma_T)}^2
			+ \|\phi_{\theta^5}(0,\cdot)u_{\theta}(0,\cdot)-\phi_{\theta^6}(0,\cdot)u_0\|_{L^2(\Omega)}^2\\
			&\quad\quad
			+ \|\phi_{\theta^7}(T,\cdot)u_{\theta}(T,\cdot)-\phi_{\theta^8}(T,\cdot)z_0\|_{L^2(\Omega)}^2
			\Big).\nonumber
		\end{align}
		
		The auxiliary neural residual is
		\begin{align}\label{eq:RNN3}
			\|\mathcal{R}_{NN}^{3}\|_{NN}^2
			:=\sum_{i=1}^{8}\|\phi_{\theta^i}-1\|^2_{Q^i}
			+\frac{1}{|Q_T|}\sum_{i=1}^{d}\|\phi_{\bar{\theta}^i}-1\|^2_{L^2(Q_T)},
		\end{align}
		where \(\|\cdot\|^2_{Q^i}=\frac{\|\cdot\|^2_{L^2(Q_T)}}{|Q_T|}\) for \(i=1,2,3\), \(\|\cdot\|^2_{Q^4}=\frac{\|\cdot\|^2_{L^2(\Sigma_T)}}{|\Gamma|}\), \(\|\phi\|^2_{Q^i}=\frac{\|\phi(0,\cdot)\|^2_{L^2(\Omega)}}{|\Gamma|}\) for \(i=5,6\), \(\|\phi\|^2_{Q^i}=\frac{\|\phi(T,\cdot)\|^2_{L^2(\Omega)}}{|\Gamma|}\) for \(i=7,8\).
		
		Thus, the compact form of the WeightedPINN in this case is
		\begin{align*}
			\inf_{\theta,\bar{\theta}}\,
			\sup_{(\theta^i)_{i=1}^8,(\bar\theta^i)_{i=1}^d}
			\|\bar{\mathcal{R}}_Q^{3}\|_{Q}^2
			+\lambda\|\bar{\mathcal{R}}_{\Gamma}^{3}\|_{\Gamma}^2
			-\rho\|\mathcal{R}_{NN}^{3}\|_{NN}^2.
		\end{align*}
		
		{\bf Wave equation.} 
		We consider the min--max formulation that defines the WeightedPINN for bilinear control in \eqref{Wave1b}:
		\begin{align*}
			\inf_{\theta,\bar\theta}\,\sup_{(\theta^i)_{i=1}^{12},(\bar\theta^i)_{i=1}^d,(\tilde\theta^i)_{i=1}^d,(\hat\theta^i)_{i=1}^{2d},(\check\theta^i)_{i=1}^{2d}}\ 
			&\frac{1}{|Q_T|}\int_{Q_T}
			\Big[
			\phi_{\theta^1}\partial_{tt}u_{\theta}
			- \sum_{i=1}^{d}\phi_{\bar{\theta}^i}\partial_{x_i x_i} u_{\theta}
			+ \phi_{\theta^{2}} v(u_{\theta})
			- \phi_{\theta^{3}} f_{\bar\theta} u_\theta
			\Big]^2 dx\,dt\\
			&+\frac{\lambda}{|\Gamma|}
			\Big(
			\|\phi_{\theta^4}u_\theta\|_{L^2(\Sigma_T)}^2
			+ \|\phi_{\tilde\theta^1}\partial_t u_\theta\|_{H^{1/2}(\Sigma_T)}^2 + \sum_{i=2}^{d} \bigl\|\phi_{\tilde\theta^i}\frac{\partial u_\theta}{\partial \sigma_{i-1}}\bigr\|_{H^{1/2}(\Sigma_T)}^2\\
			&\qquad\quad
			+\|\phi_{\theta^5}(0,\cdot)u_{\theta}(0,\cdot)-\phi_{\theta^6}(0,\cdot)u_0\|^2_{L^2(\Omega)}\\
			&\qquad\quad
			+\|\phi_{\theta^{9}}(0,\cdot)\partial_t u_{\theta}(0,\cdot)-\phi_{\theta^{10}}(0,\cdot)u_1\|^2_{L^2(\Omega)}\nonumber\\
			&\qquad\quad
			+\|\phi_{\theta^7}(T,\cdot)u_{\theta}(T,\cdot)-\phi_{\theta^8}(T,\cdot)z_0\|^2_{L^2(\Omega)}\\
			&\qquad\quad
			+\|\phi_{\theta^{11}}(T,\cdot)\partial_t u_{\theta}(T,\cdot)-\phi_{\theta^{12}}(T,\cdot)z_1\|^2_{L^2(\Omega)}\nonumber\\
			&\qquad\quad
			+\sum_{i=1}^{d}\|\phi_{\hat{\theta}^{2i-1}}(0,\cdot)\partial_{x_i}u_{\theta}(0,\cdot)-\phi_{\hat{\theta}^{2i}}(0,\cdot)\partial_{x_i}u_0\|^2_{L^2(\Omega)}\nonumber\\
			&\qquad\quad
			+\sum_{i=1}^{d}\|\phi_{\check{\theta}^{2i-1}}(T,\cdot)\partial_{x_i}u_{\theta}(T,\cdot)-\phi_{\check{\theta}^{2i}}(T,\cdot)\partial_{x_i}z_0\|^2_{L^2(\Omega)}\Big)\nonumber\\
			&-\rho\Big(
			\sum_{i=1}^{12}\|\phi_{\theta^i}-1\|^2_{Q^i}\\
			&\qquad\quad
			+\frac{1}{|Q_T|}\sum_{i=1}^{d}\|\phi_{\bar{\theta}^i}-1\|^2_{L^2(Q_T)}
			+\frac{1}{|\Gamma|}\sum_{i=1}^{d}\|\phi_{\tilde{\theta}^i}-1\|^2_{L^2(\Sigma_T)}\\
			&\qquad\quad
			+\frac{1}{|\Gamma|}\sum_{i=1}^{2d}\|\phi_{\hat{\theta}^i}(0,\cdot)-1\|^2_{L^2(\Omega)}
			+\frac{1}{|\Gamma|}\sum_{i=1}^{2d}\|\phi_{\check{\theta}^i}(T,\cdot)-1\|^2_{L^2(\Omega)}\nonumber
			\Big),
		\end{align*}
		where \(\lambda,\rho>0\) are given constants, \( (\sigma_1,\dots,\sigma_{d-1}) \) is a smooth local coordinate of \(\partial\Omega\), \(\|\cdot\|^2_{Q^i}=\frac{\|\cdot\|^2_{L^2(Q_T)}}{|Q_T|}\) for \(i=1,2,3\), \(\|\cdot\|^2_{Q^4}=\frac{\|\cdot\|^2_{L^2(\Sigma_T)}}{|\Gamma|}\), \(\|\phi\|^2_{Q^i}=\frac{\|\phi(0,\cdot)\|^2_{L^2(\Omega)}}{|\Gamma|}\) for \(i=5,6,9,10\), \(\|\phi\|^2_{Q^i}=\frac{\|\phi(T,\cdot)\|^2_{L^2(\Omega)}}{|\Gamma|}\) for \(i=7,8,11,12\). 
		
		To write a compact form for the formulation, we define
		\begin{align}\label{eq:barRQ4}
			\|\bar{\mathcal{R}}_Q^4\|_Q^2
			:=\frac{1}{|Q_T|}\int_{Q_T}\bigg[
			\phi_{\theta^1}\partial_{tt}u_{\theta}
			- \sum_{i=1}^{d}\phi_{\bar{\theta}^{i}}\partial_{x_ix_i} u_{\theta}
			+ \phi_{\theta^2}v(u_\theta)
			- \phi_{\theta^{3}}f_{\bar\theta} u_{\theta}
			\bigg]^2\,dx\,dt.
		\end{align}
		
		The boundary residual for the wave equation is given by
		\begin{align}\label{eq:barRG4}
			\|\bar{\mathcal{R}}_{\Gamma}^4\|_{\Gamma}^2
			&:=\frac{1}{|\Gamma|}
			\Big(
			\|\phi_{\theta^4}u_\theta\|_{L^2(\Sigma_T)}^2
			+ \|\phi_{\tilde\theta^0}\partial_t u_\theta\|_{H^{1/2}(\Sigma_T)}^2 + \sum_{i=1}^{d} \|\phi_{\tilde\theta^i}\partial_{x_i}u_\theta\|_{H^{1/2}(\Sigma_T)}^2\\
			&\qquad\quad
			+\|\phi_{\theta^5}(0,\cdot)u_{\theta}(0,\cdot)-\phi_{\theta^6}(0,\cdot)u_0\|^2_{L^2(\Omega)}
			+\|\phi_{\theta^{9}}(0,\cdot)\partial_t u_{\theta}(0,\cdot)-\phi_{\theta^{10}}(0,\cdot)u_1\|^2_{L^2(\Omega)}\nonumber\\
			&\qquad\quad
			+\|\phi_{\theta^7}(T,\cdot)u_{\theta}(T,\cdot)-\phi_{\theta^8}(T,\cdot)z_0\|^2_{L^2(\Omega)}
			+\|\phi_{\theta^{11}}(T,\cdot)\partial_t u_{\theta}(T,\cdot)-\phi_{\theta^{12}}(T,\cdot)z_1\|^2_{L^2(\Omega)}\nonumber\\
			&\qquad\quad
			+\sum_{i=1}^{d}\|\phi_{\hat{\theta}^{2i-1}}(0,\cdot)\partial_{x_i}u_{\theta}(0,\cdot)-\phi_{\hat{\theta}^{2i}}(0,\cdot)\partial_{x_i}u_0\|^2_{L^2(\Omega)}\nonumber\\
			&\qquad\quad
			+\sum_{i=1}^{d}\|\phi_{\check{\theta}^{2i-1}}(T,\cdot)\partial_{x_i}u_{\theta}(T,\cdot)-\phi_{\check{\theta}^{2i}}(T,\cdot)\partial_{x_i}z_0\|^2_{L^2(\Omega)}\Big).\nonumber
		\end{align}

		For the wave equation, the auxiliary neural residual is given by
		\begin{align}\label{eq:RNN4}
			\|\mathcal{R}_{NN}^{4}\|_{NN}^2
			&:=\sum_{i=1}^{12}\|\phi_{\theta^i}-1\|^2_{Q^i}+\frac{1}{|Q_T|}\sum_{i=1}^{d}\|\phi_{\bar{\theta}^i}-1\|^2_{L^2(Q_T)}+\frac{1}{|\Gamma|}\sum_{i=1}^{d}\|\phi_{\tilde{\theta}^i}-1\|^2_{L^2(\Sigma_T)}\\
			&\quad+\frac{1}{|\Gamma|}\sum_{i=1}^{2d}\|\phi_{\hat{\theta}^i}(0,\cdot)-1\|^2_{L^2(\Omega)}+\frac{1}{|\Gamma|}\sum_{i=1}^{2d}\|\phi_{\check{\theta}^i}(T,\cdot)-1\|^2_{L^2(\Omega)},\nonumber
		\end{align}
		where \(\|\cdot\|^2_{Q^i}=\frac{\|\cdot\|^2_{L^2(Q_T)}}{|Q_T|}\) for \(i=1,2,3\), \(\|\cdot\|^2_{Q^4}=\frac{\|\cdot\|^2_{L^2(\Sigma_T)}}{|\Gamma|}\), \(\|\phi\|^2_{Q^i}=\frac{\|\phi(0,\cdot)\|^2_{L^2(\Omega)}}{|\Gamma|}\) for \(i=5,6,9,10\), \(\|\phi\|^2_{Q^i}=\frac{\|\phi(T,\cdot)\|^2_{L^2(\Omega)}}{|\Gamma|}\) for \(i=7,8,11,12\).

		Finally, the WeightedPINN formulation for bilinear control on the wave equation takes the compact form
		
		\begin{align*}
			\inf_{\theta,\bar{\theta}}\,
			\sup_{(\theta^i)_{i=1}^{12},(\bar\theta^i)_{i=1}^d,(\tilde\theta^i)_{i=1}^d,(\hat\theta^i)_{i=1}^{2d},(\check\theta^i)_{i=1}^{2d}}
			\|\bar{\mathcal{R}}_Q^{4}\|_{Q}^2
			+\lambda\|\bar{\mathcal{R}}_{\Gamma}^{4}\|_{\Gamma}^2
			-\rho\|\mathcal{R}_{NN}^{4}\|_{NN}^2.
		\end{align*}

	\end{itemize}
	
	\subsection{Illustrative examples in two space dimensions}\label{Sec:Illu}
	
	We present a representative example in two space dimensions to illustrate the behavior of the standard PINN method and the proposed WeightedPINN formulation for bilinear problems. The emphasis is on how the residual structure is modified and how this impacts the reconstruction of the control.
	
	\subsubsection{Example: Two-dimensional heat equation with bilinear control}
	
	Let $Q_T := (0,T)\times \Omega$, where $\Omega=(0,1)^2$. For simplicity, we take \(v \equiv 0\) in this example. We consider the bilinear control problem associated with the heat equation
	\begin{align*}
		\partial_t u - \Delta u = f(x)\,u,
		\qquad (t,x)\in Q_T,
	\end{align*}
	supplemented with homogeneous boundary condition
	\begin{align*}
		u = 0 \quad \text{on } \Sigma_T := (0,T)\times \partial\Omega,
	\end{align*}
	and prescribed initial and terminal data
	\begin{align*}
		u(0,x)=u_0(x), \qquad u(T,x)=z_0(x).
	\end{align*}
	
	In the standard PINN approach, one introduces neural network approximations $u_\theta$ and $f_{\bar\theta}$, and minimizes the loss
	\begin{align*}
		\mathcal{L}_{\mathrm{PINN}}(\theta,\bar\theta)
		=
		\frac{1}{|Q_T|}\int_{Q_T}
		\left(
		\partial_t u_\theta - \Delta u_\theta - f_{\bar\theta} u_\theta
		\right)^2\,dx\,dt
		+\lambda \|\mathcal{R}_\Gamma(u_\theta)\|^2_{\Gamma},
	\end{align*}
	where
	\begin{align*}
		\|\mathcal{R}_\Gamma(u_\theta)\|^2_{\Gamma}
		=
		\frac{1}{|\Gamma|}
		\Big(
			\|u_\theta\|^2_{H^{1/2}(\Sigma_T)}
		+\|u_\theta(0,\cdot)-u_0\|^2_{L^2(\Omega)}
		+\|u_\theta(T,\cdot)-z_0\|^2_{L^2(\Omega)}
		\Big).
	\end{align*}
	This corresponds to a uniform least-squares enforcement of the PDE and boundary conditions. In this formulation, the residual is measured globally over the whole space-time domain, so the training process is driven by an averaged error criterion. As a consequence, the approximation may fit the state reasonably well while still leaving noticeable inaccuracies in the recovered control, especially in our control settings where the dependence on the coefficient $f$ is sensitive.
	
	In the WeightedPINN approach, one introduces additional neural networks $\phi$ acting as adaptive weights. In two space dimensions, the weighted PDE residual becomes
	\begin{align*}
		\bar{\mathcal{R}}_Q
		=
		\phi_{\theta^1}\,\partial_t u_\theta
		-\phi_{\bar\theta^1}\,\partial_{x_1x_1}u_\theta
		-\phi_{\bar\theta^2}\,\partial_{x_2x_2}u_\theta
		-\phi_{\theta^3}\,f_{\bar\theta}u_\theta.
	\end{align*}
	
	We emphasize that, in two space dimensions, the Laplacian operator is decomposed into its directional components, namely
	\[
	\Delta u_\theta = \partial_{x_1x_1} u_\theta + \partial_{x_2x_2} u_\theta,
	\]
	and the WeightedPINN formulation assigns independent weights to each second-order derivative. This leads to the terms
	\[
	-\phi_{\bar\theta^1}\,\partial_{x_1x_1}u_\theta
	\quad \text{and} \quad
	-\phi_{\bar\theta^2}\,\partial_{x_2x_2}u_\theta,
	\]
	which are treated separately in the residual. 
	
	This directional weighting provides additional flexibility compared to a single scalar weight in front of the Laplacian. In particular, it allows the training procedure to distinguish anisotropic features of the approximation error, as the contributions of the second-order derivatives along the $x_1$ and $x_2$ directions may differ significantly across the domain. As a result, the maximization step can selectively amplify discrepancies associated with specific spatial directions, thereby improving the sensitivity of the loss to localized and direction-dependent errors.
	
	From a functional viewpoint, this corresponds to replacing the standard isotropic $L^2$ residual norm by an adaptive anisotropic metric, in which each differential component of the operator is weighted independently. This refinement is particularly relevant in control problems, where the identifiability of the coefficient $f$ may depend differently on variations along different spatial directions.

	In the same spirit, the weighted formulation for the boundary, initial, and terminal conditions is given by
	\begin{align*}
		\|\bar{\mathcal R}_\Gamma\|_{\Gamma}^2
		&= \frac{1}{|\Gamma|}\bigg(
			\|\phi_{\theta^4} u_\theta\|^2_{H^{1/2}(\Sigma_T)}
		+ \|\phi_{\tilde{\theta}^1} \partial_{x_1} u_\theta\|^2_{L^2(\Sigma_T)}
		+ \|\phi_{\tilde{\theta}^2} \partial_{x_2} u_\theta\|^2_{L^2(\Sigma_T)} \\
		&\quad\quad + \|\phi_{\theta^5}(0,\cdot) u_{\theta}(0,\cdot)
		- \phi_{\theta^6}(0,\cdot) u_{0}\|_{L^2(\Omega)}^2 \\
		&\quad\quad + \|\phi_{\theta^7}(T,\cdot) u_{\theta}(T,\cdot)
		- \phi_{\theta^8}(T,\cdot) z_{0}\|_{L^2(\Omega)}^2
		\bigg).
	\end{align*}

	The corresponding loss is embedded into the min--max problem
	\begin{align*}
		\inf_{\theta,\bar\theta}\sup_{(\theta^i)_{i=1}^8,(\bar\theta^i)_{i=1}^2}
		\Bigg(
		&\|\bar{\mathcal{R}}_Q\|_Q^2
		+\lambda \|\bar{\mathcal{R}}_\Gamma\|_\Gamma^2 \\
		&-\rho \Big(
			\frac{1}{|Q_T|}\big(\|\phi_{\theta^1}-1\|^2_{L^2(Q_T)}
	+\|\phi_{\theta^3}-1\|^2_{L^2(Q_T)}\\
		&\quad\quad
		+\|\phi_{\bar\theta^1}-1\|^2_{L^2(Q_T)}
		+\|\phi_{\bar\theta^2}-1\|^2_{L^2(Q_T)}\big)\\
		&\quad\quad
		+\frac{1}{|\Gamma|}\big(\|\phi_{\theta^4}-1\|^2_{L^2(\Sigma_T)}
		+\|\phi_{\theta^5}(0,\cdot)-1\|^2_{L^2(\Omega)}
		+\|\phi_{\theta^6}(0,\cdot)-1\|^2_{L^2(\Omega)}\\
		&\quad\quad
		+\|\phi_{\theta^7}(T,\cdot)-1\|^2_{L^2(\Omega)}
		+\|\phi_{\theta^8}(T,\cdot)-1\|^2_{L^2(\Omega)}
	\big)
		\Big)
		\Bigg).
	\end{align*}
	
	When $\|\bar{\mathcal{R}}_Q\|_Q=\|\bar{\mathcal R}_\Gamma\|_\Gamma=0$ and
	\[
	\phi_{\theta^1}=\phi_{\bar\theta^1}=\phi_{\bar\theta^2}=\phi_{\theta^3}=\phi_{\theta^4}=\phi_{\theta^5}=\phi_{\theta^6}=\phi_{\theta^7}=\phi_{\theta^8}=1,
	\]
	the formulation reduces to the original PDE. This formulation changes the role of the residual in the training process. Instead of being measured in a fixed norm, it is evaluated in an adaptive metric determined by the weight networks. The maximization step increases the influence of those regions where the equation is less accurately satisfied, while the penalization term prevents the weights from drifting too far from $1$ when such amplification is unnecessary. In this way, the method allocates more attention to the most informative parts of the domain and produces a more accurate recovery of both the state and the control.

	\section{Theoretical error estimate}\label{Sec:Theory}
	
	\subsection{PINNs and WeightedPINNs for internal control problems}
	
	The internal control problem for heat and wave equations has been extensively studied in the literature. Representative results can be found in \cite{FernandezZuazua2000, fernandez-cara_guerrero_2006, zuazua2024exactcontrollabilitystabilizationwave, fu_yong_zhang_2007}.
	The main result of this section establishes a convergence estimate for the WeightedPINN approximation in terms of the network architecture and the residual errors.
	
	We exploit the universal approximation properties of neural networks to construct numerical approximations of the control. To this end, we first introduce the necessary definitions and assumptions for the neural network framework.
	
	We recall the Heaviside step function
	\begin{align*}
		\mathcal{H}(x) =
		\begin{cases}
			1, & x>0,\\
			0, & x\le 0.
		\end{cases}
	\end{align*}
	
	A function \(\sigma \in W^{m,\infty}_{\mathrm{loc}}(\mathbb{R})\) is called \emph{\(m\)-th order quasi-decay} if there exist constants \(C_1, C_2 > 0\) such that, for every \(x\neq 0\),
	\begin{align*}
		|\sigma(x)-\mathcal{H}(x)|
		\le \min\left\{\frac{C_1}{|x|},\, C_2\right\},
		\quad
		|\sigma^{(k)}(x)|
		\le \min\left\{\frac{C_1}{|x|^{k+1}},\, C_2\right\},
		\quad 1 \le k \le m.
	\end{align*}
	
	Furthermore, we say that \(\sigma\) \emph{exhibits nonlinear behavior} if there exist \(a \in \mathbb{R}\) and \(\delta > 0\) such that
	\begin{align*}
		\sigma''(a) \neq 0,
		\quad \text{and} \quad
		\sigma \in C^\infty( (a-\delta, a+\delta) ).
	\end{align*}
	
	We impose the following assumption on the activation function $\sigma$:
	
	\begin{enumerate}[label=(A\arabic*),series=con]
		\item \label{con:sigma}
		$\sigma \in W^{m,\infty}_{\mathrm{loc}}(\mathbb{R})$ exhibits nonlinear behavior, and either $\sigma(x)$ or $\sigma(x)/x$ is $m$-th order quasi-decaying.
	\end{enumerate}
	
	In \cite{yang2025deepneuralnetworksgeneral}, several examples satisfying Assumption~\ref{con:sigma} are discussed. These include sigmoid, HardSigmoid, tanh, HardTanh, Softsign, Soft-root-sign (SRS), ELU, SELU, CELU, GELU, SiLU, dSiLU, Softplus, and Mish.
	
	The weight neural networks $\phi_{\theta}$ may have flexible architectures in the sense that we can use many different activation functions to construct \(\phi_{\theta}\). We require them and their derivatives to be uniformly bounded. Specifically, we assume
	\begin{enumerate}[label=(A\arabic*),resume=con]
		\item \label{con:weight}
			There exists \(M>0\) such that
			\[
				\sup_{(t,x)\in\overline{Q_T}} \bigl\{|\phi_\theta(t,x)|+|\partial_t\phi_\theta(t,x)|+\sum_{i=1}^{d}|\partial_{x_i}\phi_\theta(t,x)| \bigr\}\le M
			\]
				for any parameter $\theta$ in a given parameter set \(\Theta\).
	\end{enumerate}
	As examples, we consider the weight networks with \(3\) hidden layers. On each hidden layer, we take a \(C^1\) activation function, which can be sigmoid, tanh, or \(\text{ReLU}^{m+1},m\ge 1\). The parameter set \(\Theta\) is assumed to be compact subset of
	\[\mathbb R^{1\times \mathcal W_3}\times\mathbb R^{1}\times \dots\times\mathbb R^{\mathcal W_1\times d}\times\mathbb R^{\mathcal W_1},\]
	which ensures the boundedness condition in \ref{con:weight}.

	\begin{Remark}\label{rmk:uniformboundweight}
		In numerical experiments, it suffices to assume that \(\sup_{(t,x)\in \overline{Q_T}}|\phi_{\theta}(t,x)|\le M\) since several higher-order regularity norms reduced to the \(L^2\)-norm.
	\end{Remark}
	
	\begin{Remark}
		For brevity, we omit the explicit dependence on the parameter set \(\Theta\) when taking supremum over the weight networks in the min--max formulations.
	\end{Remark}

	\begin{Theorem}
		Let \(u \in W^{2,\infty}(Q_T)\) be the solution corresponding to an exact control \(f \in L^{\infty}(q_T)\) of equation \eqref{Heat1}, with initial and terminal conditions \(u_0, z_0 \in L^\infty(\Omega)\).
		The constants \(\lambda\) and \(\rho\) are given such that
		\begin{align*}
		\rho&\ge\max\{3\lambda\|u_0\|^2_{L^\infty(\Omega)},3\lambda\|z_0\|^2_{L^\infty(\Omega)},\\
	&\qquad\quad(d+4)\|u\|^2_{W^{2,\infty}(Q_T)},(d+4)\|v(u)\|^2_{L^\infty(Q_T)},(d+4)\|f\|^2_{L^\infty(q_T)}\}.\end{align*}
		Suppose that the activation function \(\sigma\) satisfies Assumption \ref{con:sigma}, or that \(\sigma = \mathrm{ReLU}^{m+1}\) with \(m \ge 2\). Assume further that the weight networks \(\phi_{\theta^i},\phi_{\bar \theta^i}\) satisfy Assumption \ref{con:weight}.
		
		Then, for any \(\delta > 0\), and sufficiently large \(\mathcal{W}, \mathcal{L} \in \mathbb{N}\) such that \(\log_2(\mathcal{W}) \le \mathcal{L}\), there exist neural networks \(u_{\theta}\) and \(f_{\bar{\theta}}\), constructed using \(\sigma\), with width \(C_1 \mathcal{W}\log \mathcal{W}\) and depth \(C_2 \mathcal{L}\log \mathcal{L}\), respectively, such that
		\begin{align*}
			\sup_{(\theta^i)_{i=1}^{8},(\bar{\theta}^i)_{i=1}^{d}}\|\bar{\mathcal{R}}_Q^{1}\|^2_{Q}+\lambda\|\bar{\mathcal{R}}_\Gamma^{1}\|^2_{\Gamma}-\rho\|\mathcal R_{NN}^{1}\|_{NN}^2
			\le C_3 (\mathcal{W}^{-\frac{1}{d+1}}\mathcal{L}^{-\frac{1}{d+1}} + \delta)^2.
		\end{align*}
		Here, \(C_i\) for \(i=1,2,3\) are constants independent of \(\mathcal{W}\), \(\mathcal{L}\), and \(\delta\). The residuals \(\bar{\mathcal{R}}_Q^{1}\), \(\bar{\mathcal{R}}_\Gamma^{1}\), and \(\mathcal{R}_{NN}^{1}\) are defined in \eqref{eq:barRQ1}, \eqref{eq:barRG1}, and \eqref{eq:RNN1}, respectively.
		\label{thm:errorHeat}
	\end{Theorem}
	
	\begin{Remark}
		In general, the initial and terminal data satisfy \(u_0,z_0\in L^2(\Omega)\). A weak solution of heat equation \eqref{Heat1} satisfies the regularity
		\[u \in C([0,T];L^2(\Omega))\cap L^2(0,T;H^1(\Omega))\]
		and the control \(f\in L^2(q_T)\) (see \cite{FernandezZuazua2000,Cazenave1980}).
		
		Higher regularity is typically required to ensure the well-posedness and stability of the approximation scheme. In this work, we assume that \(u_0,z_0\in L^\infty(\Omega)\) and \(u \in W^{2,\infty}(Q_T),f\in L^{\infty}(q_T)\).
	\end{Remark}
	
	\begin{proof}
		
		We recall a standard mollifier
		\[
		\eta_\varepsilon(t,x) := \varepsilon^{-(d+1)}\, \eta(\varepsilon^{-1}t,\varepsilon^{-1}x),
		\]
		where
		\begin{align}\label{eq:mollifier}
			\eta(t,x) :=
			\begin{cases}
				I \exp\!\left(\frac{1}{|(t,x)|^2-1}\right), & \text{if } |(t,x)|^2<1,\\
				0, & \text{otherwise},
			\end{cases}
		\end{align}
		and \(I\) is chosen such that \(\int_{\mathbb{R}^{d+1}} \eta(t,x)\,dt\,dx = 1\).
		
		Let \[Q_T':= \{(t,x)\in \mathbb R^{d+1}\mid \text{there is }(s,y)\in Q_T: \|(t-s,x-y)\|\le1\}\] and \[q_T':= \{(t,x)\in \mathbb R^{d+1}\mid \text{there is }(s,y)\in q_T: \|(t-s,x-y)\|\le1\}.\]
		By \cite[Theorem 1, Section 5.4]{Evans:1998:PDE}, there exist extension operators \(E_1:H^2(Q_T)\to H^2(\mathbb R^{d+1}), E_2:L^2(q_T)\to L^2(\mathbb R^{d+1})\) and constants \(C_{E_1},C_{E_2}\) such that
		\begin{align*}
			E_1 u = u\text{ almost everywhere in }Q_T,\quad E_2 f = f\text{ almost everywhere in }q_T,
		\end{align*}
		and
		\begin{align*}
			\|E_1 u\|_{H^2(\mathbb R^{d+1})}\le C_{E_1}\|u\|_{H^2(Q_T)},\quad \|E_2 f\|_{L^2(\mathbb R^{d+1})}\le C_{E_2}\|f\|_{L^2(q_T)},
		\end{align*}
		and \(E_1 u\) has support on \(Q_T'\), \(E_2 f\) has support on \(q_T'\).

		We define
		\begin{align*}
		u_\varepsilon := (E_1 u) * \eta_\varepsilon,\quad f_\varepsilon := (E_2 f) * \eta_\varepsilon.\end{align*}

		We denote the differences
		\[\hat{u} := u - u_\varepsilon,\quad \hat{f} := f - f_\varepsilon.\]
		By \cite[Theorem 1, Section 5.3]{Evans:1998:PDE}, we have $u_\varepsilon \to u$ in $H^2(Q_T)$ and \(f_\varepsilon\to f\) in \(L^2(q_T)\) as $\varepsilon \to 0$. Hence, for any $\delta > 0$, there exists $\varepsilon \in (0,1)$ such that
		\begin{align*}
			\|\hat{u}\|_{H^2(Q_T)} + \|\hat{f}\|_{L^2(q_T)} \le \delta.
		\end{align*}
		
		By Young's inequality,
		\begin{align*}
			\|(E_1u) * \eta_\varepsilon\|_{W^{2,\infty}(\mathbb R^{d+1})}
			\le \|\eta_\varepsilon\|_{L^2(\mathbb R^{d+1})} \|E_1 u\|_{H^2(\mathbb R^{d+1})}
			\le \varepsilon^{-(d+1)/2}C_{E_1}\|\eta\|_{L^2(\mathbb R^{d+1})}\|u\|_{H^2(Q_T)}.
		\end{align*}

		Let us recall the Sobolev semi-norm \[|u|^p_{W^{k,p}}:=\sum_{\alpha_0+\dots+\alpha_d=k}\|\partial_t^{\alpha_0}\partial_{x_1}^{\alpha_1}\dots\partial_{x_d}^{\alpha_d}u\|^p_{L^p}.\] Also by Young's inequality, for the third-order Sobolev semi-norm, we get the estimate
		\begin{align*}
			|u_\varepsilon|_{W^{3,\infty}(\mathbb R^{d+1})}\le |\eta_\varepsilon|_{H^1(\mathbb R^{d+1})} |E_1 u|_{H^2(\mathbb R^{d+1})}\le \varepsilon^{-(d+1)/2-1}C_{E_1}|\eta|_{H^1(\mathbb R^{d+1})}\|u\|_{H^2(Q_T)}.
		\end{align*}
		
		Hence, we obtain
		\begin{align}
			\|u_\varepsilon\|_{W^{3,\infty}(\mathbb R^{d+1})}
			\le \varepsilon^{-(d+1)/2-1}C_{E_1}\left(\varepsilon\|\eta\|_{L^2(\mathbb R^{d+1})}+|\eta|_{H^1(\mathbb R^{d+1})}\right)\|u\|_{H^2(Q_T)}.
			\label{eq:YoungW1}
		\end{align}
		
		We apply these estimates to \(f\), which yields
		\begin{align}
			\|f_\varepsilon\|_{W^{2,\infty}(\mathbb R^{d+1})}
			\le \varepsilon^{-(d+1)/2-2}C_{E_2}\left(\varepsilon^{2}\|\eta\|_{L^2(\mathbb R^{d+1})}+\varepsilon|\eta|_{H^1(\mathbb R^{d+1})}+|\eta|_{H^2(\mathbb R^{d+1})}\right)\|f\|_{L^2(q_T)}.
			\label{eq:YoungW2}
		\end{align}
		
		Let \(\mathcal{L}_v\) denote the Lipschitz constant of \(v\) and let \(M\) be the constant from Assumption \ref{con:weight}.	
		
		We estimate the residual as follows:
		\begin{align*}
			\|\bar{\mathcal{R}}_Q^1(u_\varepsilon,f_\varepsilon)\|_{Q}&\le \Big\|\phi_{\theta^1}\partial_t \hat u - \sum_{i=1}^{d}\phi_{\bar{\theta}^i}\partial_{x_i x_i}\hat u+\phi_{\theta^2}(v(u)-v(u_\varepsilon))-\phi_{\theta^3}\hat f \mathbf 1_{\omega}\Big\|_Q\nonumber\\
			&\quad+\|\phi_{\theta^1}-1\|_{Q}\|\partial_t u\|_{L^\infty(Q_T)}+\sum_{i=1}^{d}\|\phi_{\bar\theta^i}-1\|_{Q}\|\partial_{x_i x_i} u\|_{L^\infty(Q_T)}\nonumber\\
			&\quad+\|\phi_{\theta^2}-1\|_{Q}\|v(u)\|_{L^\infty(Q_T)}+\|\phi_{\theta^3}-1\|_{Q}\|f\|_{L^\infty(q_T)}.\nonumber
		\end{align*}

		By the uniform bound of the weight networks, we have
		\begin{align}
			\label{eq:est_res}
			&\Big\|\phi_{\theta^1}\partial_t \hat u - \sum_{i=1}^{d}\phi_{\bar{\theta}^i}\partial_{x_i x_i}\hat u+\phi_{\theta^2}(v(u)-v(u_\varepsilon))-\phi_{\theta^3}\hat f \mathbf 1_{\omega}\Big\|_Q\\
			&\quad\quad\le \frac{M}{\sqrt{|Q_T|}}\Big(
			\|\partial_{t}\hat{u}\|_{L^2(Q_T)}
			+\|\Delta \hat{u}\|_{L^2(Q_T)}
			+\|v(u)-v(u_\varepsilon)\|_{L^2(Q_T)}
			+\|\hat{f}\|_{L^2(q_T)}
			\Big)\nonumber\\
			&\quad\quad\le \frac{M}{\sqrt{|Q_T|}}\Big(
			\|\hat{u}\|_{H^2(Q_T)} 
			+ \mathcal{L}_v \|\hat{u}\|_{L^2(Q_T)} 
			+ \|\hat{f}\|_{L^2(q_T)}
			\Big)\nonumber\\
			&\quad\quad\le \frac{M(1+\mathcal{L}_v)}{\sqrt{|Q_T|}}\,\delta.\nonumber
		\end{align}

		By \cite[Theorem 25]{yang2025deepneuralnetworksgeneral}, there exist neural networks \(u_\theta\) and \(f_{\bar{\theta}}\) with width \(C_1 \mathcal{W}\log \mathcal{W}\) and depth \(C_2 \mathcal{L}\log \mathcal{L}\), where \(C_1,C_2\) are independent of \(\mathcal{W},\mathcal{L}\) and \(\delta\), such that
		\begin{align}
			\|u_\varepsilon-u_\theta\|_{W^{2,\infty}(Q_T)}
			&\le C_{Q_T}(d)\|u_\varepsilon\|_{W^{3,\infty}(Q_T)}\mathcal{W}^{-\frac{2}{d+1}}\mathcal{L}^{-\frac{2}{d+1}},\label{eq:udiff}\\
			\|f_\varepsilon-f_{\bar{\theta}}\|_{L^\infty(q_T)}
			&\le C_{q_T}(d)\|f_\varepsilon\|_{W^{2,\infty}(q_T)}\mathcal{W}^{-\frac{4}{d+1}}\mathcal{L}^{-\frac{4}{d+1}}.\label{eq:fdiff}
		\end{align}
		
		Choosing \(\mathcal{W},\mathcal{L}\) sufficiently large such that \(\mathcal{W}^{\frac{1}{d+1}}\mathcal{L}^{\frac{1}{d+1}} \ge \varepsilon^{-1-\frac{d+1}{2}},\) by \eqref{eq:YoungW1}, \eqref{eq:YoungW2}, \eqref{eq:udiff}, and \eqref{eq:fdiff}, we obtain
		\begin{align*}
			\|u_\varepsilon-u_\theta\|_{H^{2}(Q_T)} + \|f_\varepsilon-f_{\bar{\theta}}\|_{L^2(q_T)}
			\le \tilde{C}_0\mathcal{W}^{-\frac{1}{d+1}}\mathcal{L}^{-\frac{1}{d+1}},
		\end{align*}
		where \(\tilde C_0>0\) is independent of \(\mathcal W,\mathcal L\) and \(\delta\). 

		Similar to \eqref{eq:est_res}, for \(\tilde{u}:=u_\theta - u_\varepsilon,\tilde f:=f_{\bar{\theta}}-f_\varepsilon\), we have
		\begin{align*}
			\Big\|\phi_{\theta^1}\partial_t \tilde u - \sum_{i=1}^{d}\phi_{\bar{\theta}^i}\partial_{x_i x_i}\tilde u+\phi_{\theta^2}(v(u_\theta)-v(u_\varepsilon))-\phi_{\theta^3}\tilde f \mathbf 1_{\omega}\Big\|_Q \le \frac{M(1+\mathcal L_v)}{\sqrt{|Q_T|}} \tilde{C}_0\mathcal{W}^{-\frac{1}{d+1}}\mathcal{L}^{-\frac{1}{d+1}}.
		\end{align*}
		
		Combining the above estimates yields
		\begin{align*}
			\|\bar{\mathcal{R}}_Q^1(u_\theta,f_{\bar{\theta}})\|_{Q}
			&\le \Big\|\phi_{\theta^1}\partial_t \tilde u - \sum_{i=1}^{d}\phi_{\bar{\theta}^i}\partial_{x_i x_i}\tilde u+\phi_{\theta^2}(v(u_\theta)-v(u_\varepsilon))-\phi_{\theta^3}\tilde f \mathbf 1_{\omega}\Big\|_Q+\|\bar{\mathcal{R}}_Q^1(u_\varepsilon,f_\varepsilon)\|_{Q}\\
			&\le \frac{M(1+\mathcal{L}_v)}{\sqrt{|Q_T|}}(\tilde{C}_0\mathcal{W}^{-\frac{1}{d+1}}\mathcal{L}^{-\frac{1}{d+1}}+\delta)\nonumber\\
			&\quad+\|\phi_{\theta^1}-1\|_{Q}\|\partial_t u\|_{L^\infty(Q_T)}+\sum_{i=1}^{d}\|\phi_{\bar\theta^i}-1\|_{Q}\|\partial_{x_i x_i} u\|_{L^\infty(Q_T)}\nonumber\\
			&\quad+\|\phi_{\theta^2}-1\|_{Q}\|v(u)\|_{L^\infty(Q_T)}+\|\phi_{\theta^3}-1\|_{Q}\|f\|_{L^\infty(q_T)}.\nonumber\\
			&\le \tilde C_1 (\mathcal{W}^{-\frac{1}{d+1}}\mathcal{L}^{-\frac{1}{d+1}} + \delta),\nonumber\\
			&\quad+\|\phi_{\theta^1}-1\|_{Q}\|\partial_t u\|_{L^\infty(Q_T)}+\sum_{i=1}^{d}\|\phi_{\bar\theta^i}-1\|_{Q}\|\partial_{x_i x_i} u\|_{L^\infty(Q_T)}\nonumber\\
			&\quad+\|\phi_{\theta^2}-1\|_{Q}\|v(u)\|_{L^\infty(Q_T)}+\|\phi_{\theta^3}-1\|_{Q}\|f\|_{L^\infty(q_T)}.\nonumber
		\end{align*}
		where \(\tilde C_1>0\) is independent of \(\mathcal W,\mathcal L\) and \(\delta\).

		Since \[\rho\ge(d+4)\max\{\|u\|^2_{W^{2,\infty}(Q_T)},\|v(u)\|^2_{L^\infty(Q_T)},\|f\|^2_{L^\infty(q_T)}\},\]
		by the inequality \( (a_1+\dots+a_{d+4})^2 \le (d+4)(a_1^2+\dots+a_{d+4}^2)\), we obtain
			\begin{align}\label{eq:est_R1Q}
				\|\bar{\mathcal{R}}_Q^1(u_\theta,f_{\bar{\theta}})\|^2_{Q}&\le  (d+4)\tilde C_1^2 (\mathcal{W}^{-\frac{1}{d+1}}\mathcal{L}^{-\frac{1}{d+1}} + \delta)^2\\
				&\quad +\rho\Big(\|\phi_{\theta^1}-1\|_{Q}^2+\sum_{i=1}^{d}\|\phi_{\bar\theta^i}-1\|^2_{Q}+\|\phi_{\theta^2}-1\|_{Q}^2+\|\phi_{\theta^3}-1\|_{Q}^2\Big)\nonumber
			\end{align}
		
		We now estimate the boundary term by
		\begin{align*}
			\|\phi_{\theta^4}u_\theta\|_{H^{1/2}(\Sigma_T)}^2
			\le M^2 \|u_\theta\|^2_{H^{1}(\Sigma_T)}.
		\end{align*}
		For the initial condition, by triangle inequality, we estimate
		\begin{align*}
			\|\phi_{\theta^5}(0,\cdot)u_\theta(0,\cdot)-\phi_{\theta^6}(0,\cdot)u_0\|_{L^2(\Omega)}
			&\le \|\phi_{\theta^5}(0,\cdot)(u_\theta(0,\cdot)-u_0)\|_{L^2(\Omega)}+\|(\phi_{\theta^5}(0,\cdot)-\phi_{\theta^6}(0,\cdot))u_0\|_{L^2(\Omega)}\\
			&\le M\|u_\theta (0,\cdot)-u_0\|_{L^2(\Omega)}\\ 
			&\quad + \|u_0\|_{L^\infty(\Omega)}(\|\phi_{\theta^5}(0,\cdot)-1\|_{L^2(\Omega)}+\|\phi_{\theta^6}(0,\cdot)-1\|_{L^2(\Omega)}).
		\end{align*}

		Since \[\rho\ge 3\lambda\max\{\|u_0\|^2_{L^\infty(\Omega)},\|z_0\|^2_{L^\infty(\Omega)}\},\] by the inequality \( (a+b+c)^2\le 3a^2 + 3b^2 + 3c^2\), we obtain
		\begin{align}\label{eq:est_init}
			\|\phi_{\theta^5}(0,\cdot)u_\theta(0,\cdot)-\phi_{\theta^6}(0,\cdot)u_0\|^2_{L^2(\Omega)}
			- \frac{\rho}{\lambda}(\|\phi_{\theta^5}(0,\cdot)-1\|^2_{L^2(\Omega)}+\|\phi_{\theta^6}(0,\cdot)-1\|^2_{L^2(\Omega)})\nonumber\\
			\le 3M^2\|u_\theta (0,\cdot)-u_0\|^2_{L^2(\Omega)}.
		\end{align}
Similarly, we get
\begin{align*}
			\|\phi_{\theta^7}(T,\cdot)u_\theta(T,\cdot)-\phi_{\theta^8}(T,\cdot)z_0\|^2_{L^2(\Omega)}
			- \frac{\rho}{\lambda}(\|\phi_{\theta^7}(T,\cdot)-1\|^2_{L^2(\Omega)}+\|\phi_{\theta^8}(T,\cdot)-1\|^2_{L^2(\Omega)})\\
			\le 3M^2\|u_\theta (T,\cdot)-z_0\|^2_{L^2(\Omega)}.
		\end{align*}

		Thus, we deduce that
		\begin{align*}
			\lambda\|\bar{\mathcal R}_{\Gamma}^1(u_\theta)\|^2_{\Gamma}&\le \frac{3M^2}{|\Gamma|}\|u_\theta-u\|^2_{H^1(\Gamma)}\\
			&\quad+\rho\Big(\|\phi_{\theta^5}(0,\cdot)-1\|^2_{L^2(\Omega)}+\|\phi_{\theta^6}(0,\cdot)-1\|^2_{L^2(\Omega)}\\
			&\quad\quad+\|\phi_{\theta^7}(T,\cdot)-1\|^2_{L^2(\Omega)}+\|\phi_{\theta^8}(T,\cdot)-1\|^2_{L^2(\Omega)}\Big).
		\end{align*}

		By the Trace Theorem \cite[Theorem 1, Section 5.5]{Evans:1998:PDE}, there exists \(\tilde C_2>0\) such that
		\[
			\|u_\theta-u\|_{H^1(\Gamma)}^2
			\le \tilde C_2\|u_\theta-u\|^2_{H^2(Q_T)}\le 2\tilde C_2(\|u_\theta-u_\varepsilon\|^2_{H^2(Q_T)}+\|\hat u\|^2_{H^2(Q_T)}),
		\]
		which implies
		\begin{align}\label{eq:est_R1G_NN}
			\lambda\|\bar{\mathcal R}_{\Gamma}^1(u_\theta)\|^2_{\Gamma}
			&\le \tilde C_3 (\mathcal{W}^{-\frac{1}{d+1}}\mathcal{L}^{-\frac{1}{d+1}} + \delta)^2\\
			&\quad+\rho\Big(\|\phi_{\theta^5}(0,\cdot)-1\|^2_{L^2(\Omega)}+\|\phi_{\theta^6}(0,\cdot)-1\|^2_{L^2(\Omega)}\nonumber\\
			&\quad\quad+\|\phi_{\theta^7}(T,\cdot)-1\|^2_{L^2(\Omega)}+\|\phi_{\theta^8}(T,\cdot)-1\|^2_{L^2(\Omega)}\Big),\nonumber
		\end{align}
		where \(\tilde C_3>0\) is independent of \(\mathcal{W},\mathcal L\) and \(\delta\).

		By \eqref{eq:est_R1Q}, \eqref{eq:est_R1G_NN}, there exists a constant \(C_3>0\) independent of \(\mathcal W,\mathcal L\) and \(\delta\) such that
		\begin{align*}
			\|\bar{\mathcal{R}}_Q^{1}\|^2_{Q}+\lambda\|\bar{\mathcal{R}}_\Gamma^{1}\|^2_{\Gamma}-\rho\|\mathcal R_{NN}^{1}\|_{NN}^2
			\le C_3 (\mathcal{W}^{-\frac{1}{d+1}}\mathcal{L}^{-\frac{1}{d+1}} + \delta)^2
		\end{align*}
		for any given parameters \((\theta^i)_{i=1}^{8},(\bar{\theta}^i)_{i=1}^{d}\).	
	\end{proof}
	
	We now extend the previous result to the wave equation.
	
	\begin{Theorem}
		Let \(u \in W^{2,\infty}(Q_T)\) be the solution corresponding to an exact control \(f \in L^{\infty}(q_T)\) of equation \eqref{Wave1}, with initial and terminal conditions
		\((u_0,u_1), (z_0,z_1) \in W^{1,\infty}(\Omega)\times L^\infty(\Omega)\). 
		The constants \(\lambda\) and \(\rho\) are given such that
		\begin{align*}
		\rho&\ge\max\{3\lambda\|u_0\|^2_{W^{1,\infty}(\Omega)},3\lambda\|u_1\|^2_{L^\infty(\Omega)},3\lambda\|z_0\|^2_{W^{1,\infty}(\Omega)},3\lambda\|z_1\|^2_{L^\infty(\Omega)}\\
&\qquad\quad(d+4)\|u\|^2_{W^{2,\infty}(Q_T)},(d+4)\|v(u)\|^2_{L^\infty(Q_T)},(d+4)\|f\|^2_{L^\infty(q_T)}\}.\end{align*}
		Suppose that the activation function \(\sigma\) satisfies Assumption \ref{con:sigma}, or that \(\sigma = \mathrm{ReLU}^{m+1}\) with \(m \ge 2\). Assume further that the weight networks \(\phi_{\theta^i},\phi_{\bar \theta^i},\phi_{\tilde \theta^i},\phi_{\hat \theta^i},\phi_{\check \theta^i}\) satisfy Assumption \ref{con:weight}.
		
		Then, for any \(\delta > 0\), and sufficiently large \(\mathcal{W}, \mathcal{L} \in \mathbb{N}\) such that \(\log_2(\mathcal{W}) \le \mathcal{L}\), there exist neural networks \(u_{\theta}\) and \(f_{\bar{\theta}}\), constructed using \(\sigma\), with width \(C_1 \mathcal{W}\log \mathcal{W}\) and depth \(C_2 \mathcal{L}\log \mathcal{L}\), respectively, such that
		\begin{align*}
			\sup_{(\theta^i)_{i=1}^{12},(\bar{\theta}^i)_{i=1}^d,(\tilde{\theta}^i)_{i=1}^{d},(\hat{\theta}^i)_{i=1}^{2d},(\check{\theta}^i)_{i=1}^{2d}}\|\bar{\mathcal{R}}_Q^{2}\|^2_{Q}+\lambda\|\bar{\mathcal{R}}_\Gamma^{2}\|^2_{\Gamma}-\rho\|\mathcal R_{NN}^{2}\|_{NN}^2
			\le C_3 (\mathcal{W}^{-\frac{1}{d+1}}\mathcal{L}^{-\frac{1}{d+1}} + \delta)^2.
		\end{align*}
		Here, \(C_i\) for \(i=1,2,3\) are constants independent of \(\mathcal{W}\), \(\mathcal{L}\), and \(\delta\). The residuals \(\bar{\mathcal{R}}_Q^{2}\), \(\bar{\mathcal{R}}_\Gamma^{2}\), and \(\mathcal{R}_{NN}^{2}\) are defined in \eqref{eq:barRQ2}, \eqref{eq:barRG2}, and \eqref{eq:RNN2}, respectively.
		\label{thm:errorWave}
	\end{Theorem}
	
	\begin{proof}
		The proof of Theorem \ref{thm:errorWave} is analogous to that of Theorem \ref{thm:errorHeat}. For brevity, we only list the modified parts of the proof:
		\begin{enumerate}
			\item The time derivative \(\partial_t\) is modified into \(\partial_{tt}\).
			\item Repeat the estimate \eqref{eq:est_init} for \(u_1,z_1\) and all directional derivatives of \(u_0,z_0\).
			\item The boundary is estimated by
				\begin{align*}
					\|\phi_{\theta^4}u_\theta\|_{L^2(\Sigma_T)}^2+\|\phi_{\tilde\theta^1}\partial_tu_\theta\|^2_{H^{1/2}(\Sigma_T)}+\sum_{i=2}^{d}\bigl\|\phi_{\tilde\theta^i}\frac{\partial u_\theta}{\partial \sigma_{i-1}}\bigr\|^2_{H^{1/2}(\Sigma_T)}
					\le CM^2\|u_\theta\|^2_{H^{3/2}(\Sigma_T)}.
				\end{align*}
				\item We use the Trace Theorem for fractional Sobolev space \cite[Theorem 3.37]{McLean2000}, which yields
					\begin{align*}
						\|u_\theta-u\|^2_{H^{3/2}(\Gamma)}\le \tilde C\|u_\theta-u\|_{H^2(Q_T)}^2.
					\end{align*}
		\end{enumerate}
	\end{proof}
	
	\begin{Remark}
		Similar to Theorem \ref{thm:errorHeat}, we assume \( (u_0,u_1),(z_0,z_1)\in W^{1,\infty}(\Omega)\times L^{\infty}(\Omega)\) and \(u\in W^{2,\infty}(Q_T),f\in L^{\infty}(q_T)\) for Theorem \ref{thm:errorWave}.
		
		In general, a weak solution of wave equation \eqref{Wave1} satisfies the regularity
		\[u \in C([0,T];H^1(\Omega))\cap C^1([0,T];L^2(\Omega))\]
		and the control \(f\in L^2(q_T)\) (see \cite{fu_yong_zhang_2007,Cazenave1980}).
	\end{Remark}
	
	\begin{Remark}
		The weight \(\phi \equiv 1\) satisfies Assumption \ref{con:weight}. In this case, the WeightedPINN loss reduces to the standard PINN loss. Therefore, Theorem \ref{thm:errorHeat} and Theorem \ref{thm:errorWave} include the standard PINN as a special case.
	\end{Remark}

	Theorem \ref{thm:errorHeat} and Theorem \ref{thm:errorWave} establish the existence of deep neural networks that approximate both the state and the control given sufficient regularity. In particular, the results guarantee that the total WeightedPINN loss can be made arbitrarily small, provided the network width and depth are chosen sufficiently large.

	Now, let us further assume a multiplicative structure on \eqref{Heat1}. For the heat equation, we assume
	\begin{enumerate}[label=(V\arabic*),series=V]
		\item\label{con:Vheat} There exists \(V\in L^{\infty}(Q_T)\) such that the control problem on the heat equation under consideration is
			\begin{align*}
				\partial_t u - \Delta u + V u &= f \mathbf 1_{\omega} \quad \text{in }Q_T,\\
				u &= 0 \quad \text{on }\Sigma_T,\\
				u(0) &= u_0 \quad \text{in } \Omega,\\
				u(T) &= z_0 \quad \text{in } \Omega.
			\end{align*}
	\end{enumerate}
	In other words, we simply replace \(v(u)\) in Equation \eqref{Heat1} and the Residual \eqref{eq:barRQ1} by \(V u\).

	In addition to Theorem \ref{thm:errorHeat}, we show in Theorem \ref{thm:error_estHeat} that, under the multiplicative structure, if we can find the neural networks \(u_\theta\) and \(f_{\bar{\theta}}\) with sufficiently small WeightedPINN loss, then these networks approximate an admissible pair of solution and control of the PDE.

	For brevity, we denote the space of admissible pairs of solution and control associated with heat equation with given initial and terminal conditions \( u_0,z_0\in L^2(\Omega)\) by \(\mathcal A_{u_0}^{z_0}\).

	As discussed in Section \ref{Sec:internalsettings}, \(\mathcal A_{u_0}^{z_0}\) is not guaranteed to be non-empty. We assume \(\mathcal A_{u_0}^{z_0}\) to be non-empty. This also implies that \(z_0\) is reachable with some control \(f\mathbf 1_\omega\). If \(f_{\bar\theta}\) approximates \(f\), then we can expect to have an approximate control from \(u_0\) to \(z_0\). We only keep the terminal \(z_0\) to be exact. To make it precise, in this work, for \(V\) in \ref{con:Vheat}, we assume:
\begin{enumerate}[label=(A\arabic*),resume=con]
\item \label{con:heat}
	For given neural network \(f_{\bar \theta}\) and \(\delta>0\), there exist
\(F_\delta\in L^2(q_T),u_\delta \in L^2(\Omega)\) and
\[
\bar u\in C([0,T];L^2(\Omega))\cap L^2(0,T;H^1_0(\Omega))
\]
such that
\[
	\|F_\delta\|_{L^2(q_T)}^2 + \|u_\delta\|_{L^2(\Omega)}^2\le \delta,
\]
and
\begin{align*}
\partial_t\bar u-\Delta\bar u+V\bar u
&=
(f_{\bar\theta}+F_\delta)\mathbf 1_\omega
\quad\text{in }Q_T,\\
\bar u &= 0 \quad\text{on }\Sigma_T,\\
\bar u(0)&=u_0+u_\delta \quad\text{in }\Omega,\\
\bar u(T)&=z_0\quad\text{in }\Omega.
\end{align*}
\end{enumerate}

We note that \(\mathcal A_{u_0}^{z_0}\neq\emptyset\) and the fact that \(f_{\bar\theta}\) approximates a control \(f\) do not directly implies \ref{con:heat}. Here, it stays as a structural hypothesis. We do not claim \ref{con:heat} is a necessary condition.

Now, let \(\bar u\in C([0,T];L^2(\Omega))\cap L^2(0,T;H^1(\Omega))\) and \(f\in L^2(q_T)\), the distance of \( (\bar u,\bar f)\) to \(\mathcal A_{u_0}^{z_0}\) is given by
	\begin{align*}
	{\rm dist}_h( (\bar u,\bar f),\mathcal A_{u_0}^{z_0}):= \inf_{(u,f)\in \mathcal{A}_{u_0}^{z_0}}\Big(\max_{t\in [0,T]}\| u(t)-\bar u(t)\|^2_{L^2(\Omega)}
			+\|u-\bar u\|^2_{L^2(0,T;H^1(\Omega))}
	+\|f-\bar f\|^2_{L^2(q_T)}\Big)^{1/2}.
	\end{align*}

	\begin{Theorem}
		Consider the control problem on the heat equation in \ref{con:Vheat} and let \(\mathcal A_{u_0}^{z_0}\) be the set of admissible pair for the given initial and terminal conditions \( u_0,z_0\in L^2(\Omega)\). Suppose that \(\mathcal A_{u_0}^{z_0}\) is non-empty.
		
		Assume that there exist \(u_\theta\) and \(f_{\bar{\theta}}\), constructed using activation function \(\sigma\in W^{2,\infty}_{\rm loc}(\mathbb{R})\), such that \(f_{\bar \theta}\) and a given constant \(\delta>0\) satisfy Assumption \ref{con:heat} and
		\begin{align*}
			\sup_{(\theta^i)_{i=1}^8,(\bar\theta^i)_{i=1}^d}
			\|\bar{\mathcal{R}}_Q^{1}\|_{Q}^2
			+\lambda\|\bar{\mathcal{R}}_{\Gamma}^{1}\|_{\Gamma}^2
			-\rho\|\mathcal{R}_{NN}^{1}\|_{NN}^2
			\le \delta ,
		\end{align*}
		where \(\bar{\mathcal{R}}_{Q}^{\,1},\bar{\mathcal{R}}_{\Gamma}^{\,1}\) and
		\(\mathcal{R}_{NN}^{\,1}\) are defined in \eqref{eq:barRQ1},
		\eqref{eq:barRG1} and \eqref{eq:RNN1}, respectively.  
		The constants \(\lambda,\rho>0\) are given, and the weight networks \(\phi_{\theta^{i}},\phi_{\bar\theta^{i}}\) belong to a function class that contains the constant function \(1\).
		
		Then, there exists a constant \(C>0\), depending only on \(Q_T\), \(\|V\|_{L^\infty(Q_T)}\), and \(\lambda\) such that
		\begin{align*}	
			{\rm dist}_h( (u_\theta,f_{\bar\theta}),\mathcal{A}_{u_0}^{z_0})\le C \delta^{1/2}.
		\end{align*}
		\label{thm:error_estHeat}
	\end{Theorem}
	
	\begin{Remark}
		In Theorem \ref{thm:error_estHeat}, we assume the existence of neural networks \(u_\theta\) and \(f_{\bar{\theta}}\) that drive the loss arbitrarily small. Therefore, in this case, we do not require the higher regularity assumptions \(u_0,z_0\in L^\infty(\Omega)\) or  \(u \in W^{2,\infty}(Q_T),f\in L^\infty(q_T)\) as in Theorem \ref{thm:errorHeat}.
	\end{Remark}
	\begin{Remark}
		We manage to derive an error estimate of the neural networks to the space of admissible pairs for multiplicative structure. The general case, where \(v\) is non-linear, is a challenging problem and remains open. 
	\end{Remark}

	\begin{proof}
		Since the class of weight networks contains the constant function \(1\), we may set \(\phi_{\theta^i},\phi_{\bar{\theta}^i}\equiv 1\). The hypothesis therefore reduces to
		\begin{align*}
			\|\mathcal{R}_Q^1(u_\theta,f_{\bar{\theta}})\|_{Q}^2
			+\lambda\|\mathcal{R}_\Gamma^1(u_\theta)\|_\Gamma^2
			\le\delta.
		\end{align*}
		
		We have function \(V\) such that \(v(u)=V u\) as in Assumption \ref{con:Vheat}.

		By Assumption \ref{con:heat}, let \(\bar u\in C\bigl([0,T];L^{2}(\Omega)\bigr)\cap L^{2}\bigl(0,T;H^{1}(\Omega)\bigr)\) be the weak solution of  
\begin{align*}
\partial_t\bar u-\Delta\bar u+V\bar u
&=
(f_{\bar\theta}+F_\delta)\mathbf 1_\omega
\quad\text{in }Q_T,\\
\bar u &= 0 \quad\text{on }\Sigma_T,\\
\bar u(0)&=u_0+u_\delta \quad\text{in }\Omega,\\
\bar u(T)&=z_0\quad\text{in }\Omega.
\end{align*}
for some \(F_\delta,u_\delta\) satisfying \[\|F_\delta\|_{L^2(q_T)}^2+\|u_\delta\|_{L^2(\Omega)}^2\le\delta.\]
		
		By \cite[Chapter 4, Equation (15.38)]{LionsMagenesVol2}, there exist constant \(\tilde C(Q_T)\) and \(w\in L^2(0,T;H^1(\Omega))\cap H^{1}(0,T;H^{-1}(\Omega))\) such that 
		\[w(t)|_{\partial \Omega}=u_\theta(t)|_{\partial \Omega}\]
		and
		\[\|w\|^2_{L^2(0,T;H^1(\Omega))}+\|\partial_t w\|^2_{L^{2}(0,T;H^{-1}(\Omega))}\le \tilde C(Q_T)(\|u_\theta\|^2_{L^2(0,T;H^{1/2}(\partial\Omega))}+\|u_\theta\|^2_{H^{1/4}(0,T;L^{2}(\partial\Omega))}).\]

		We note that, for compatibility, we do not prescribe the initial or terminal condition for \(w\).

We have
		\begin{align}
			\|w\|^2_{L^2(0,T;H^1(\Omega))}+\|\partial_t w\|^2_{L^{2}(0,T;H^{-1}(\Omega))}
			&\le \tilde C(Q_T)\|u_\theta\|^2_{H^{1/2}(\Sigma_T)}.\label{eq:H-1/2trace}
	\end{align}
		
			We denote \(\hat u := u_{\theta}-\bar u-w\). Using the previous identity, \(\hat u\) is the weak solution of the following equation 
		\begin{align*}
			\partial_{t}\hat u-\Delta\hat u+V\hat u
			&= R-F_\delta\mathbf 1_\omega-(\partial_t w-\Delta w+V w)
			\quad\text{in }Q_{T},\\
			\hat u &= 0 \quad\text{on }\Sigma_{T},\\
			\hat u(0)
			&= u_\theta(0)-u_0-u_\delta-w(0)
			\quad\text{in }\Omega,
		\end{align*}
		where
		\[
			R:=\partial_{t}u_{\theta}-\Delta u_{\theta}+V u_{\theta}-f_{\bar\theta}\mathbf{1}_{\omega}
			\quad\text{in }Q_T.
		\]

		Taking test function \(\hat u\) and integrating by parts yield
		\begin{align*}
		\frac12\partial_t\|\hat u\|_{L^{2}(\Omega)}^{2}
		+\|\nabla\hat u\|_{L^{2}(\Omega)}^{2}
		&=\int_{\Omega}(R-F_\delta\mathbf 1_\omega-V\hat u)\hat u\, d x\\
		&\quad-\int_{\Omega}\partial_t w\hat u\, d x-\int_{\Omega} \nabla w\cdot \nabla\hat u\, d x-\int_{\Omega}V w \hat u\, d x.
	\end{align*}

		By Cauchy–Schwarz’s inequalities,
		\[
			\int_{\Omega}(R-F_\delta\mathbf 1_\omega-V\hat u)\hat u\, d x
			\le\|R\|_{L^{2}(\Omega)}^{2}+\delta+\frac{1}{2}\|\hat u\|_{L^{2}(\Omega)}^{2}
		+\|V\|_{L^{\infty}(Q_T)}\|\hat u\|_{L^{2}(\Omega)}^{2}\bigr).
		\]

		We also have
		\begin{align*}
		\Big|\int_{\Omega}\partial_t w\hat u\, d x+\int_{\Omega} \nabla w\cdot \nabla\hat u\, d x+\int_{\Omega}V w \hat u\, d x\Big|
		&\le \|\partial_t w\|^2_{H^{-1}(\Omega)}+\|\nabla w\|^2_{L^2(\Omega)}\\
		&\quad\quad+\|V\|^2_{L^\infty(Q_T)}\|w\|^2_{L^2(\Omega)}+\frac{1}{2}\|\hat u\|^2_{H^1(\Omega)}
	\end{align*}

		Let \(\mathcal L_v:= \|V\|_{L^\infty(Q_T)}\), there exists a constant \(C_{0}=C_{0}(Q_{T},\mathcal L_{v})>0\) such that
		\begin{equation}\label{eq:energy}
			\partial_t\|\hat u\|_{L^{2}(\Omega)}^{2}
			+\|\hat u\|_{H^{1}(\Omega)}^{2}
			\le C_{0}\bigl(\|\hat u\|_{L^{2}(\Omega)}^{2}+\|R\|_{L^{2}(\Omega)}^{2}
			+\|\partial_t w\|^2_{H^{-1}(\Omega)}+\|w\|^2_{H^1(\Omega)}+\delta\bigr).
		\end{equation}
		
		Applying Gr\"onwall’s inequality to \eqref{eq:energy} gives, for any \(t\in[0,T]\),
		\begin{align*}
			\|\hat u\|_{L^{2}(\Omega)}^{2}
			&\le e^{C_{0}t}\Bigl(
				\|\hat u(0,\cdot)\|_{L^{2}(\Omega)}^{2}
				-\int_{0}^{t}\|\hat{u}\|^2_{H^1(\Omega)}ds\nonumber\\
			&\quad+C_{0}\!\int_{0}^{t}\!\bigl(\|R\|_{L^{2}(\Omega)}^{2}
			+\|\partial_t w\|^2_{H^{-1}(\Omega)}+\|w\|^2_{H^1(\Omega)}+\delta\bigr)\, ds\Bigr)\\
			&\le e^{C_{0}T}\Bigl(
			C_{0}\|R\|_{L^{2}(Q_{T})}^{2}
			+C_{0}\|w\|^2_{L^2(0,T;H^1(\Omega))}
			+C_0\|\partial_t w\|^2_{L^2(0,T;H^{-1}(\Omega))}+(C_0T+3)\delta\\
		&\qquad\quad+3\|u_\theta(0)-u_0\|^2_{L^2(\Omega)}+3\|w(0)\|^2_{L^2(\Omega)}\Bigr)-e^{C_{0}t}\int_{0}^{t}\|\hat{u}\|^2_{H^1(\Omega)}ds.
		\end{align*}
		
		Consequently, there exists a constant \(C_{1}=C_{1}(Q_{T},\mathcal L_{v})>0\) such that
		\begin{align}\label{eq:hat_u_est1}
			&\max_{t\in[0,T]}\|u_\theta(t)-\bar u(t)\|_{L^2(\Omega)}^2
			+\|u_\theta-\bar u\|_{L^2(0,T;H^1(\Omega))}^2\\
			&\le 2\Big(\sup_{t\in[0,T]}\|\hat u(t)\|_{L^{2}(\Omega)}^{2}
			+\|\hat u\|_{L^{2}(0,T;H^{1}(\Omega))}^{2}
		+\sup_{t\in[0,T]}\| w(t)\|_{L^{2}(\Omega)}^{2}
			+\|w\|_{L^{2}(0,T;H^{1}(\Omega))}^{2}\Big)\nonumber\\
			&\le C_{1}\Bigl(
			\|R\|_{L^{2}(Q_{T})}^{2}
			+\|u_{\theta}(0,\cdot)-u_{0}\|_{L^2(\Omega)}^2+\delta\nonumber\\
			&\quad+\sup_{t\in[0,T]}\|w(t)\|_{L^{2}(\Omega)}^{2}
			+\|w\|^2_{L^2(0,T;H^1(\Omega))}
		+\|\partial_t w\|^2_{L^2(0,T;H^{-1}(\Omega))}\Bigr).\nonumber
		\end{align}

		We note that \[ |\partial_t \|w\|_{L^2(\Omega)}^2| = \Big|\int_{\Omega}w \partial_t w\, d x\Big|\le \frac{1}{2}\big(\|w\|_{H^1(\Omega)}^2+\|\partial_t w\|_{H^{-1}(\Omega)}^2\big).\]

		Hence,
		\[
			\|w(t)\|_{L^2(\Omega)}^2\le \|w(s)\|_{L^2(\Omega)}^2
			+ \frac{1}{2}(\|w\|^2_{L^2(0,T;H^1(\Omega))}+\|\partial_t w\|^2_{L^2(0,T;H^{-1}(\Omega))}).
		\]
		for all \(s,t\). For any \(\varepsilon>0\), we take \(s_\varepsilon\) such that \(\|w(s_\varepsilon)\|_{H^1(\Omega)}^2<\inf_{s\in[0,T]}\|w(s)\|_{H^1(\Omega)}^2+\varepsilon\). At time \(s_\varepsilon\), we have
		\[\|w(s_\varepsilon)\|_{L^2(\Omega)}^2\le\|w(s_\varepsilon)\|_{H^1(\Omega)}^2\le \frac{1}{T}\|w\|^2_{L^2(0,T;H^1(\Omega))}+\varepsilon.\]
		Let \(\varepsilon\to 0\) and taking supremum over \(t\), there exists constant \(C_2=C_2(Q_T)>0\) such that
		\begin{align}\label{eq:embedding}
			\sup_{t\in[0,T]}\|w(t)\|^2_{L^2(\Omega)}\le C_2(\|w\|^2_{L^2(0,T;H^1(\Omega))}+\|\partial_t w\|^2_{L^2(0,T;H^{-1}(\Omega))}).
		\end{align}

		Combining with \eqref{eq:H-1/2trace}, \eqref{eq:hat_u_est1} and \eqref{eq:embedding}, there exists \(C_3 = C_3(Q_T,\mathcal L_v)>0\) such that
		\begin{align}\label{eq:hat_u_est2}
		\max_{t\in[0,T]}\|u_\theta(t)-\bar u(t)\|_{L^2(\Omega)}^2
			+\|u_\theta-\bar u\|_{L^2(0,T;H^1(\Omega))}^2
			&\le C_3 \delta.
	\end{align}
		
By the null controllability of the linear heat equation established through global Carleman estimates in \cite[Theorem~1.6 and its proof]{fernandez-cara_guerrero_2006}, there exist $g\in L^2(q_T)$ and
$\xi\in C([0,T];L^2(\Omega))\cap L^2(0,T;H^1_0(\Omega))$ such that
\begin{equation}\label{eq:xi}
	\partial_t\xi-\Delta\xi+V\xi=(-F_\delta+g)\mathbf 1_\omega
\text{ in }Q_T,\quad
\xi|_{\Sigma_T}=0,\quad
\xi(0)=-u_\delta,\quad
\xi(T)=0,
\end{equation}
together with a constant $C_4=C_4(Q_T,\mathcal L_v)>0$ satisfying
\begin{equation}\label{eq:g_est}
	\|g\|_{L^2(q_T)}\le \|F_\delta\|_{L^2(q_T)}+\|-F_\delta+g\|_{L^2(q_T)}\le \sqrt\delta + C_4\|u_\delta\|_{L^2(\Omega)}\le (1+C_4)\sqrt\delta.
\end{equation}
The standard energy estimate applied to \eqref{eq:xi} yields
\begin{equation}\label{eq:xi_est}
\max_{t\in[0,T]}\|\xi(t)\|_{L^2(\Omega)}^{2}
+\|\xi\|_{L^2(0,T;H^1(\Omega))}^{2}
\le C_5\bigl(\|-F_\delta+g\|_{L^2(q_T)}^{2}+\|u_\delta\|^2_{L^2(\Omega)}\bigr)\le C_6\delta
\end{equation}
for $C_5,C_6>0$ depending only on $Q_T$ and $\mathcal L_v$.

Define
\[
u:=\bar u+\xi,\qquad
f:=f_{\bar\theta}+g.
\]
Adding the equations for $\bar u$ and $\xi$,
\[
\partial_t u-\Delta u+Vu
=(f_{\bar\theta}+F_\delta-F_\delta+g)\mathbf 1_\omega
=f\mathbf 1_\omega\text{ in }Q_T,
\]
together with $u|_{\Sigma_T}=0$, $u(0)=\bar u(0)+\xi(0)=u_0$, and
$u(T)=\bar u(T)+\xi(T)=z_0$. Hence $(u,f)\in\mathcal A_{u_0}^{z_0}$.
		
Combining \eqref{eq:hat_u_est2} and \eqref{eq:xi_est},
\begin{align*}
&\max_{t\in[0,T]}\|u_\theta(t)-u(t)\|_{L^2(\Omega)}^{2}
+\|u_\theta-u\|_{L^2(0,T;H^1(\Omega))}^{2}\\
&\qquad\le 2\bigl(\max_{t\in[0,T]}\|u_\theta(t)-\bar u(t)\|_{L^2}^{2}
+\max_{t\in[0,T]}\|\xi(t)\|_{L^2}^{2}\bigr)
+2\bigl(\|u_\theta-\bar u\|_{L^2(H^1)}^{2}+\|\xi\|_{L^2(H^1)}^{2}\bigr)\le C_7\delta,
\end{align*}
for some $C_7=C_7(Q_T,\mathcal L_v,\lambda)>0$. Therefore, there exists
$C=C(Q_T,\mathcal L_v,\lambda)>0$ such that
\[
\max_{t\in[0,T]}\|u_\theta(t)-u(t)\|_{L^2(\Omega)}^{2}
+\|u_\theta-u\|_{L^2(0,T;H^1(\Omega))}^{2}
+\|f_{\bar\theta}-f\|_{L^2(q_T)}^{2}\le C^2\delta,
\]
which yields
\[
\operatorname{dist}_h\!\bigl((u_\theta,f_{\bar\theta}),\mathcal A_{u_0}^{z_0}\bigr)
\le C\sqrt\delta.
\qedhere
\]
	\end{proof}

	We extend the result to the wave equation \eqref{Wave1}. 
	The multiplicative structure on the wave equation is written as follows:
\begin{enumerate}[label=(V\arabic*),resume=V]
		\item\label{con:Vwave} There exists \(V\in L^{\infty}(Q_T)\) such that the control problem on the wave equation is
			\begin{align*}
				\partial_{tt} u - \Delta u + V u &= f \mathbf 1_{\omega} \quad \text{in }Q_T,\\
				u &= 0 \quad \text{on }\Sigma_T,\\
				(u(0),\partial_t u(0)) &= (u_0,u_1) \quad \text{in } \Omega,\\
				(u(T),\partial_t u(T)) &= (z_0,z_1) \quad \text{in } \Omega.
			\end{align*}
	\end{enumerate}
	We replace \(v(u)\) in Equation \eqref{Wave1} and the Residual \eqref{eq:barRQ2} by \(V u\).

	The space of admissible pairs of solution and control associated with \eqref{Wave1} with given initial and terminal conditions \( (u_0,u_1),(z_0,z_1) \in H^1(\Omega)\times L^2(\Omega)\) is denoted by \(\mathcal A_{(u_0,u_1)}^{(z_0,z_1)}\).

	Let \(\bar u \in C([0,T];H^1(\Omega))\cap C^1([0,T];L^2(\Omega))\) and \(\bar f\in L^2(q_T)\), the distance of \( (\bar u,\bar f)\) to \(\mathcal A_{(u_0,u_1)}^{(z_0,z_1)}\) is given by
	\begin{align*}
		&{\rm dist}_w( (\bar u,\bar f),\mathcal A_{(u_0,u_1)}^{(z_0,z_1)})\\
		&\quad\quad:= \inf_{(u,f)\in \mathcal{A}_{(u_0,u_1)}^{(z_0,z_1)}}\Big(\max_{t\in [0,T]}\{\| u(t)-\bar u(t)\|^2_{H^1(\Omega)}
			+\|\partial_t u(t)-\partial_t \bar u(t)\|^2_{L^2(\Omega)}\}
	+\|f-\bar f\|^2_{L^2(q_T)}\Big)^{1/2}.
	\end{align*}

	We consider the following assumptions:
	\begin{enumerate}[label=(A\arabic*),resume=con]
		\item \label{con:wave} The domain \(\omega\subset\Omega\) contains a neighborhood of the boundary \(\partial \Omega\).
		\item \label{con:waveT} \(T > \text{diameter} (\Omega)\).
	\end{enumerate}
	
	\begin{Remark}\label{rmk:con_wave}
		Assumptions \ref{con:wave} and \ref{con:waveT} are sufficient for controllability and stability (see \cite[Theorem~4.1]{zuazua2024exactcontrollabilitystabilizationwave}). A sharp sufficient condition for the controllability of the wave equation, known as the \emph{Geometric Control Condition} (GCC), was introduced in \cite{BardosLebeauRauch1992}. We refer to \cite{LeRousseau2017,DehmanErvedozaZuazua2025} for more recent developments on the GCC. For readability, we adopt these simpler assumptions.
	\end{Remark}
	
	\begin{Theorem}
		Consider the control problem on the wave equation in \ref{con:Vwave} and let \(\mathcal A_{(u_0,u_1)}^{(z_0,z_1)}\) be the set of admissible pair for the given initial and terminal condition \( (u_0,u_1),(z_0,z_1)\in H^1(\Omega)\times L^2(\Omega)\).
			We assume that the conditions \ref{con:wave} and \ref{con:waveT} hold.

		Suppose there exist \(u_\theta\) and \(f_{\bar{\theta}}\), constructed using activation function \(\sigma\in W^{2,\infty}_{\rm loc}(\mathbb{R})\), such that
		\begin{align*}
			\sup_{(\theta^i)_{i=1}^{12},(\bar\theta^i)_{i=1}^d,(\tilde\theta^i)_{i=1}^d,(\hat\theta^i)_{i=1}^{2d},(\check\theta^i)_{i=1}^{2d}}
			\|\bar{\mathcal{R}}_Q^{2}\|_{Q}^2+\lambda\|\bar{\mathcal{R}}_{\Gamma}^{2}\|_{\Gamma}^2-\rho\|\mathcal{R}_{NN}^{2}\|_{NN}^2\le \delta,
		\end{align*}
		where \(\bar{\mathcal{R}}_Q^{2},\bar{\mathcal{R}}_{\Gamma}^{2}\), and \(\mathcal{R}_{NN}^{2}\) are defined in \eqref{eq:barRQ2}, \eqref{eq:barRG2}, and \eqref{eq:RNN2}, respectively. The constants $\delta,\lambda,\rho>0$ are given, and the weight networks \(\phi_{\theta^i},\phi_{\bar{\theta}^i},\phi_{\tilde{\theta}^i},\phi_{\hat{\theta}^i},\phi_{\check{\theta}^i}\) belong to a function class that includes constant function \(1\).
		
		Then, there exists a constant \(C>0\), depending only on \(Q_T\), \(\|V\|_{L^\infty(Q_T)}\), and \(\lambda\) such that
		\begin{align*}
			{\rm dist}_w( (u_\theta,f_{\bar{\theta}}),\mathcal A_{(u_0,u_1)}^{(z_0,z_1)})
			\le C \delta^{1/2}.
		\end{align*}
		\label{thm:error_estWave}
	\end{Theorem}
	\begin{Remark}
		In Theorem \ref{thm:error_estWave}, we assume the existence of neural networks \(u_\theta\) and \(f_{\bar{\theta}}\) that drive the loss arbitrarily small. Therefore, in this case, we do not require the higher regularity assumptions \( (u_0,u_1),(z_0,z_1)\in W^{1,\infty}(\Omega)\times L^\infty(\Omega)\) or \(u \in W^{2,\infty}(Q_T),f\in L^\infty(\Omega)\) as in Theorem \ref{thm:errorWave}.
	\end{Remark}

	\begin{Remark}
		The situation with general \(v\) also remains open in Theorem \ref{thm:error_estWave}.
	\end{Remark}
	
	\begin{proof}
		The proof proceeds similarly to that of Theorem \ref{thm:error_estHeat}.
		
		Since the class of weight networks contains constant function \(1\), we may choose \(\phi_{\theta^i},\phi_{\bar{\theta}^i},\phi_{\tilde{\theta}^i},\phi_{\hat{\theta}^i},\phi_{\check{\theta}^i}\equiv 1\).
		Under this choice, the hypothesis reduces to
		\begin{align*}
			\|\mathcal{R}_Q^2(u_\theta,f_{\bar{\theta}})\|_{Q}^2+\lambda\|\mathcal{R}_\Gamma^2(u_\theta)\|_\Gamma^2\le\delta.
		\end{align*}
		
		We have function \(V\) such that \(v(u) = Vu\) as in Assumption \ref{con:Vwave}.

		Let \(\bar{u}\in C([0,T];H^1(\Omega))\cap C^1([0,T];L^2(\Omega))\) be the weak solution of
		\begin{align*}
			\partial_{tt} \bar{u} -\Delta \bar{u} + V(T-t,\cdot)\bar{u} &= f_{\bar{\theta}}(T-t,\cdot)\mathbf 1_{\omega} \quad\text{in } Q_T,\\
			\bar{u} &= 0 \quad\text{on } \Sigma_T,\\
			(\bar{u}(0),\partial_t \bar{u}(0)) &= (z_0,-z_1) \quad\text{in }\Omega.
		\end{align*}
		It follows from the approach developed in \cite{Cazenave1980} (see also \cite[Equation~(1.5)]{fu_yong_zhang_2007}) that the above  wave equation admits a unique solution. 

		By \cite[Chapter 4, Theorem 2.1 and Theorem 2.3]{LionsMagenesVol2}, the trace operator is a continuous linear operator from \(L^{2}(0,T;H^2(\Omega))\cap H^2(0,T;L^2(\Omega))\) to \( L^2(0,T;H^{3/2}(\partial\Omega))\cap H^{3/2}(0,T;L^2(\partial\Omega))\). We do not prescribe initial nor terminal condition for compatibility.

		There exist constant \(\tilde C(Q_T)\) and \(w\in L^{2}(0,T;H^2(\Omega))\cap H^2(0,T;L^2(\Omega))\) such that
		\[w(t)|_{\partial \Omega}=u_\theta(t)|_{\partial \Omega}\]
		and
		\[\|w\|_{L^2(0,T;H^2(\Omega))}^2+\|w\|_{H^2(0,T;L^2(\Omega))}^2\le \tilde C(Q_T)(\|u_\theta(t)\|^2_{L^2(0,T;H^{3/2}(\partial\Omega))}+\|u_\theta(t)\|^2_{H^{3/2}(0,T;L^{2}(\partial\Omega))}).\]

		Therefore, we observe that
\begin{align}
	\|w\|_{L^2(0,T;H^2(\Omega))}^2+\|w\|_{H^2(0,T;L^2(\Omega))}^2\le \tilde C(Q_T)\|u_\theta\|^2_{H^{3/2}(\Sigma_T)}.\label{eq:H-3/2trace}
	\end{align}
	The estimate holds when we replace \(w\) by \(\tilde w(t,x) := w(T-t,x)\).

	We denote \(\hat u(t,x) := u_{\theta}(T-t,x)-\bar u(t,x)-\tilde w(t,x)\). Using the previous identity, \(\hat u\) is the weak solution of the following equation
		\begin{align*}
			\partial_{tt}\hat u-\Delta\hat u+V(T-t)\hat u
			&= R(T-t) - (\partial_{tt}\tilde w-\Delta\tilde w+V(T-t)\tilde w)
			\quad\text{in }Q_{T},\\
			\hat u &= 0 \quad\text{on }\Sigma_{T},\\
			(\hat u(0),\partial_t \hat u(0))
			&= (u_\theta(T)-z_0-\tilde w(0),-\partial_t u_\theta(T)+z_1-\partial_t \tilde w(0))
			\quad\text{in }\Omega,
		\end{align*}
		where
		\[
			R:=\partial_{tt}u_{\theta}-\Delta u_{\theta}+V u_{\theta}-f_{\bar\theta}\mathbf{1}_{\omega}
		\quad\text{in }Q_T.
		\]
		
		Taking test function \(\partial_t \hat{u}\) and integrating by parts yield
		\begin{align*}
			\frac{1}{2}\partial_t\big( \|\partial_t \hat{u}\|^2_{L^2(\Omega)}+\|\nabla \hat{u}\|^2_{L^2(\Omega)}\big)
			&= \int_{\Omega} (R(T-t)-V(T-t) \hat{u}) \partial_t \hat{u} d x\\
			&\quad-\int_{\Omega}(\partial_{tt}\tilde w-\Delta\tilde w+V(T-t)\tilde w)\partial_t \hat u\, dx.
		\end{align*}
		
		By Cauchy-Schwarz's inequalities,
		\begin{align*}
			\int_{\Omega} (R(T-t)-V(T-t) \hat{u}) \partial_t \hat{u} d x
			&\le \frac{1}{2}\bigg(\|R(T-t)\|^2_{L^2(\Omega)}
			+\|\partial_t \hat{u}\|^2_{L^2(\Omega)}\\
			&\quad+ \|V\|_{L^\infty(Q_T)}(\|\hat{u}\|^2_{L^2(\Omega)}
		+\|\partial_t \hat{u}\|^2_{L^2(\Omega)})\bigg).
		\end{align*}
		We also have
		\begin{align*}
			\Big|\int_{\Omega}(\partial_{tt}\tilde w-\Delta\tilde w+V(T-t)\tilde w)\partial_t \hat u\, d x\Big|
			&\le \frac{3}{4}\|\partial_t \hat u\|^2_{L^2(\Omega)}
		+\|\partial_{tt}\tilde w\|_{L^2(\Omega)}^2\\
		&\quad+\|\Delta\tilde w\|_{L^2(\Omega)}^2+\|V\|^2_{L^\infty(Q_T)}\|\tilde w\|_{L^2(\Omega)}^{2}.
		\end{align*}
		
		By the Poincar\'e inequality, there exists a constant \(C(\Omega)>0\) such that
		\begin{align*}
			\|\hat{u}\|^2_{L^2(\Omega)}\le C(\Omega)\|\nabla\hat{u}\|^2_{L^2(\Omega)}.
		\end{align*}
		Let \(\mathcal L_v:=\|V\|_{L^\infty(Q_T)}\), there exists a constant \(C_{0}=C_{0}(Q_{T},\mathcal L_{v})>0\) such that
		\begin{align*}
			\partial_t E(t)
			\le C_0 \bigl( E(t)+\|R(T-t)\|^2_{L^2(\Omega)}
				+\|\tilde w\|_{H^2(\Omega)}^2
			+\|\partial_{tt}\tilde w\|_{L^2(\Omega)}^2\bigr)
		\end{align*}
		where \(E(t) = \|\partial_t \hat{u}\|^2_{L^2(\Omega)}+\|\nabla \hat{u}\|^2_{L^2(\Omega)}\).
		
		Applying the Gronwall's inequality yields, for any \(t\in [0,T]\),
		\begin{align*}
			E &\le e^{C_0t}\bigg( E(0)+C_0\int_{0}^{t}(\|R(T-t)\|^2_{L^2(\Omega)}+\|\tilde w\|^2_{H^2(\Omega)}+\|\partial_{tt}\tilde w\|^2_{L^2(\Omega)})ds\bigg)\\
			&\le e^{C_0T}\bigg( E(0)+C_0\bigl(\|R\|^2_{L^2(Q_T)}+\|w\|_{L^2(0,T;H^2(\Omega))}^2+\|w\|_{H^2(0,T;L^2(\Omega))}^2\bigr)\bigg).
		\end{align*}

		Consequently, there exists a constant \(C_1=C_1(Q_{T},\mathcal L_{v})>0\) such that
		\begin{align*}
			\max_{t\in[0,T]}\bigl\{\|\partial_t\hat{u}(t)\|^2_{L^2(\Omega)}+\|\nabla \hat{u}\|^2_{L^2(\Omega)}\bigr\}
			&\le C_1 \bigl(\|R\|^2_{L^2(Q_T)}+\|w\|_{L^2(0,T;H^2(\Omega))}^2+\|w\|_{H^2(0,T;L^2(\Omega))}\\
				&\quad+\|\nabla w(T)\|_{L^2(\Omega)}^2
				+\|\partial_t w(T)\|_{L^2(\Omega)}^2\\
			&\quad+\|\nabla(u_{\theta}(T)- z_0)\|^2_{L^2(\Omega)}+\|\partial_t u_{\theta}(T)-z_1\|^2_{L^2(\Omega)}\bigr).\nonumber
		\end{align*}
		
		By fundamental theorem of calculus and \(\|a+b+c\|^2\le 3(\|a\|^2+\|b\|^2+\|c\|^2)\), we have
		\begin{align*}
			\|\hat{u}(t)\|^2_{L^2(\Omega)}
			\le 3\|u_\theta(T)-z_0\|^2_{L^2(\Omega)}
			+3\|w(T)\|_{L^2(\Omega)}^2
			+ 3\bigg\|\int_{0}^{t}\partial_t \hat{u}(s)ds\bigg\|^2_{L^2(\Omega)}.
		\end{align*}
		We further estimate
		\begin{align*}
			\bigg\|\int_{0}^{t}\partial_t \hat{u}(s)ds\bigg\|^2_{L^2(\Omega)}
			\le t\int_{\Omega}\int_{0}^{t}|\partial_t \hat{u}(s,x)|^2dsdx
			\le T^2\max_{s\in[0,T]}\|\partial_t \hat{u}(s)\|^2_{L^2(\Omega)}.
		\end{align*}
		
		Therefore, there exists a constant \(C_2 = C_2(Q_{T},\mathcal L_{v})>0\) such that
		\begin{align*}
			&\max_{t\in[0,T]}\bigl\{\|u_\theta(t)-\bar u(T-t)\|^2_{H^1(\Omega)}+\|\partial_tu_\theta(t)+\partial_t\bar u(T-t)\|^2_{L^2(\Omega)}\bigr\}\\
			&\quad
			\le C_2 \bigl(\|R\|^2_{L^2(Q_T)}+\|u_\theta\|^2_{H^{3/2}(\Sigma_T)}
				+\max_{s\in [0,T]}\{\| w(s)\|_{H^1(\Omega)}^2
				+\|\partial_t w(s)\|_{L^2(\Omega)}^2\}\nonumber\\
			&\qquad\quad+\|u_{\theta}(T)- z_0\|^2_{H^1(\Omega)}+\|\partial_t u_{\theta}(T)-z_1\|^2_{L^2(\Omega)}\bigr).\nonumber
		\end{align*}

		By the interpolation property \cite[Chapter 4, Proposition 2.1]{LionsMagenesVol2}, \(L^2(0,T;H^2(\Omega))\cap H^2(0,T;L^2)\) can be embedded into \(H^{2\theta}(0,T;H^{2-2\theta}(\Omega))\) for \(\theta\in[0,1]\).

		Then, we can embed \(H^{2\theta}(0,T;H^{2-2\theta}(\Omega))\)  into \( C([0,T];H^{2-2\theta}(\Omega))\) if \(\theta>1/4\) and into \( C^1([0,T];H^{2-2\theta}(\Omega))\) if \(\theta>3/4\).

		Consequently, \(L^2(0,T;H^2(\Omega))\cap H^2(0,T;L^2(\Omega))\) can be embedded into \(C([0,T],H^{1}(\Omega))\cap C^1([0,T],L^2(\Omega))\).

		This implies that there exists constant \(C_3=C_3(Q_T)>0\) such that
		\begin{align}\label{eq:embedding2}
			\max_{s\in [0,T]}\{\| w(s)\|_{H^1(\Omega)}^2
				+\|\partial_t w(s)\|_{L^2(\Omega)}^2\}
				\le C_3(\|w\|^2_{L^2(0,T;H^2(\Omega))}+\|w\|^2_{H^2(0,T;L^2(\Omega))}).
		\end{align}

		Hence, there exists a constant \(C_4=C_4(Q_T,\mathcal L_v)>0\) such that
		\begin{align}\label{eq:liftfinish}
			&\max_{t\in[0,T]}\bigl\{\|u_\theta(t)-\bar u(T-t)\|^2_{H^1(\Omega)}+\|\partial_tu_\theta(t)+\partial_t\bar u(T-t)\|^2_{L^2(\Omega)}\bigr\}\\
			&\quad
			\le C_4\bigl(\|R\|^2_{L^2(Q_T)}+\|u_\theta\|^2_{H^{3/2}(\Sigma_T)}
			+\|u_{\theta}(T)- z_0\|^2_{H^1(\Omega)}+\|\partial_t u_{\theta}(T)-z_1\|^2_{L^2(\Omega)}\bigr)\nonumber\\
			&\quad \le C_4\delta.\nonumber
		\end{align}
		
By the controllability of the linear wave equation established in \cite[Theorem 4.1]{zuazua2024exactcontrollabilitystabilizationwave}, there exist $g\in L^2(q_T)$ and
$\xi\in C([0,T];H^1_0(\Omega))\cap C^1([0,T];L^2(\Omega))$ such that
\begin{equation}\label{eq:xi_wave}
	\begin{split}
	\partial_{tt}\xi-\Delta\xi+V\xi&=g\mathbf 1_\omega
\text{ in }Q_T,\quad
\xi|_{\Sigma_T}=0,\quad\\
(\xi(0),\partial_t\xi(0))&=(u_0-\bar u(T),u_1+\partial_t\bar u(T)),\quad
(\xi(T),\partial_t\xi(T))=(0,0),
\end{split}
\end{equation}
together with a constant $C_5=C_5(Q_T,\mathcal L_v)>0$ satisfying
\begin{align}\label{eq:g_est_wave}
	\|g\|_{L^2(q_T)}&\le C_5\|(u_0-\bar u(T),u_1+\partial_t\bar u(T))\|_{H^1(\Omega)\times L^2(\Omega)}\\
	&\le C_5(\|(u_0- u_\theta(0),u_1-\partial_t u_\theta(0))\|_{H^1(\Omega)\times L^2(\Omega)}+\|(u_\theta(0)-\bar u(T),\partial_t u_\theta(0)+\partial_t \bar u(T))\|_{H^1(\Omega)\times L^2(\Omega)})\nonumber\\
	&\le C_5(1+\sqrt{C_4})\sqrt\delta\nonumber.
\end{align}
The standard energy estimate applied to \eqref{eq:xi_wave} yields
\begin{equation}\label{eq:xi_est_wave}
\max_{t\in[0,T]}\{\|\xi(t)\|_{H^1(\Omega)}^{2}
+\|\partial_t \xi(t)\|_{L^2(\Omega)}^{2}\}
\le C_6\bigl(\|g\|_{L^2(q_T)}^2+\|(u_0-\bar u(T),u_1+\partial_t\bar u(T))\|_{H^1(\Omega)\times L^2(\Omega)}^2\bigr)\le C_7\delta
\end{equation}
for $C_6,C_7>0$ depending only on $Q_T$ and $\mathcal L_v$.

Define
\[
	u(t,x):=\bar u(T-t,x)+\xi(t,x),\qquad
	f:=f_{\bar\theta}+g.
\]
Adding the equations for $\bar u(T-t)$ and $\xi$,
\[
	\partial_{tt} u-\Delta u+Vu
=f\mathbf 1_\omega\text{ in }Q_T,
\]
together with $u|_{\Sigma_T}=0$, $(u(0),\partial_t u(0))=(\bar u(T)+\xi(0),-\partial_t \bar u(T)+\partial_t\xi(0))=(u_0,u_1)$, and
$(u(T),\partial_t u(T))=(\bar u(0)+\xi(T),-\partial_t\bar u(0)+\partial_t\xi(T))=(z_0,z_1)$. Hence $(u,f)\in\mathcal A_{(u_0,u_1)}^{(z_0,z_1)}$.
		
Combining \eqref{eq:liftfinish} and \eqref{eq:xi_est_wave},
\begin{align*}
&\max_{t\in[0,T]}\{\|u_\theta(t)-u(t)\|_{H^1(\Omega)}^{2}
+\|\partial_t u_\theta(t)-\partial_t u(t)\|_{L^2(\Omega)}^{2}\}\\
&\qquad\le 2\bigl(\max_{t\in[0,T]}\{\|u_\theta(t)-\bar u(T-t)\|_{H^1(\Omega)}^{2}+\|\partial_t u_\theta(t)+\partial_t\bar u(T-t)\|_{L^2(\Omega)}^2\}\\
&\qquad\quad+\max_{t\in[0,T]}\{\|\xi(t)\|_{H^1(\Omega)}^{2}+\|\partial_t \xi(t)\|_{L^2(\Omega)}^{2}\}\bigr)\\
&\qquad\le C_8\delta,
\end{align*}
for some $C_8=C_8(Q_T,\mathcal L_v,\lambda)>0$. Therefore, there exists
$C=C(Q_T,\mathcal L_v,\lambda)>0$ such that
\[
\max_{t\in[0,T]}\{\|u_\theta(t)-u(t)\|_{H^1(\Omega)}^{2}
+\|\partial_tu_\theta(t)-\partial_t u(t)\|_{L^2(\Omega)}^{2}\}
+\|f_{\bar\theta}-f\|_{L^2(q_T)}^{2}\le C^2\delta,
\]
which yields
\[
	\operatorname{dist}_w\!\bigl( (u_\theta,f_{\bar{\theta}}),\mathcal A_{(u_0,u_1)}^{(z_0,z_1)}\bigr)
\le C\sqrt\delta.
\qedhere
\]	
	\end{proof}
	
	\subsection{PINNs and WeightedPINNs for bilinear controls}~~
	
	We extend the results from Theorem \ref{thm:errorHeat} and Theorem \ref{thm:errorWave} to bilinear control settings.
	
	\begin{Theorem}
		Assume that there exists a bilinear control \(f\in W^{1,\infty}(Q_T)\)  and a corresponding solution \(u\in W^{2,\infty}(Q_T)\) governed by \eqref{Heat1b}  with initial and terminal conditions \( u_0, z_0\in L^\infty(\Omega)\). The constants \(\lambda\) and \(\rho\) are given such that
		\begin{align*}
		\rho&\ge\max\{3\lambda\|u_0\|^2_{L^\infty(\Omega)},3\lambda\|z_0\|^2_{L^\infty(\Omega)},\\
	&\qquad\quad(d+4)\|u\|^2_{W^{2,\infty}(Q_T)},(d+4)\|v(u)\|^2_{L^\infty(Q_T)},(d+4)\|f\|^2_{L^\infty(Q_T)}\|u\|^2_{L^\infty(Q_T)}\}.\end{align*}
		Suppose that the activation function \(\sigma\) satisfies Assumption \ref{con:sigma}, or that \(\sigma = \mathrm{ReLU}^{m+1}\) with \(m \ge 2\), and suppose that the weight networks \(\phi_{\theta^i}, \phi_{\bar \theta^i}\) satisfy Assumption \ref{con:weight}.
		
		Then, for any \(\delta > 0\), and sufficiently large \(\mathcal{W}, \mathcal{L} \in \mathbb{N}\) such that \(\log_2(\mathcal{W}) \le \mathcal{L}\), there exist neural networks \(u_{\theta}\) and \(f_{\bar{\theta}}\), constructed using \(\sigma\), with width \(C_1 \mathcal{W}\log \mathcal{W}\) and depth \(C_2 \mathcal{L}\log \mathcal{L}\), respectively, such that
		\begin{align*}
			\sup_{(\theta^i)_{i=1}^{8},(\bar{\theta}^i)_{i=1}^d}\|\bar{\mathcal{R}}_Q^{3}\|^2_{Q}+\lambda\|\bar{\mathcal{R}}_\Gamma^{3}\|^2_{\Gamma}-\rho\|\mathcal R_{NN}^{3}\|_{NN}^2
			\le C_3 (\mathcal{W}^{-\frac{1}{d+1}}\mathcal{L}^{-\frac{1}{d+1}} + \delta)^2.
		\end{align*}
		Here, \(C_i\) for \(i=1,2,3\) are constants independent of \(\mathcal{W}\), \(\mathcal{L}\), and \(\delta\). 
		The residuals \(\bar{\mathcal{R}}_Q^{3},\bar{\mathcal{R}}_\Gamma^{3}\), and \(\mathcal{R}_{NN}^3\)  are defined in \eqref{eq:barRQ3}, \eqref{eq:barRG3}, and \eqref{eq:RNN3}.
		\label{thm:errorHeatb}
	\end{Theorem}
	
	\begin{Remark}
		Similar to Theorem \ref{thm:errorHeat}, to ensure the well-posedness and stability of the approximation scheme, we assume that \(u_0,z_0\in L^\infty(\Omega)\) and \(u \in W^{2,\infty}(Q_T)\). The control \(f\) enters the equation multiplicatively, which requires higher regularity for analysis. Thus, we further assume that \(f \in W^{1,\infty}(Q_T)\).
	\end{Remark}
	
	\begin{proof}
		
		Let us recall the standard mollifier
		\[
		\eta_\varepsilon(t,x):=\varepsilon^{-(d+1)}\,
		\eta\bigl(\varepsilon^{-1}t,\varepsilon^{-1}x\bigr),
		\]
		where \(\eta\) is defined in \eqref{eq:mollifier}.
		
		Let \[Q_T':= \{(t,x)\in \mathbb R^{d+1}\mid \text{there is }(s,y)\in Q_T: \|(t-s,x-y)\|\le1\}.\]
By \cite[Theorem 1, Section 5.4]{Evans:1998:PDE}, there exist extension operators \(E_1:H^2(Q_T)\to H^2(\mathbb R^{d+1}), E_2:W^{1,\infty}(Q_T)\to W^{1,\infty}(\mathbb R^{d+1})\) and constants \(C_{E_1},C_{E_2}\) such that
		\begin{align*}
			E_1 u|_{Q_T} = u\text{ almost everywhere in }Q_T,\quad E_2 f = f\text{ almost everywhere in }Q_T,
		\end{align*}
		and
		\begin{align*}
			\|E_1 u\|_{H^2(\mathbb R^{d+1})}\le C_{E_1}\|u\|_{H^2(Q_T)},\quad \|E_2 f\|_{W^{1,\infty}(\mathbb R^{d+1})}\le C_{E_2}\|f\|_{W^{1,\infty}(Q_T)},
		\end{align*}
		and \(E_1 u,E_2 f\) has support on \(Q_T'\).

		We define
		\begin{align*}
		u_\varepsilon := (E_1 u) * \eta_\varepsilon,\quad f_\varepsilon := (E_2 f) * \eta_\varepsilon.\end{align*}
		We denote the differences
		\[
		\hat u := u-u_\varepsilon,\qquad
		\hat f := f-f_\varepsilon .
		\]
		
		For any \(\delta>0\) there exists a sufficiently small \(\varepsilon>0\) such that
		\begin{align*}
			\|\hat u\|_{H^{2}(Q_T)} \le \delta .
		\end{align*}
		Since \(f\in W^{1,\infty}(Q_T)\),
		\begin{align*}
			\|\hat f\|_{L^{\infty}(Q_T)}
			&= \esssup_{(t,x)\in Q_T}
			\Bigg|\int_{\mathbb{R}^{d+1}}\eta_{\varepsilon}(s,y)
			\bigl(E_2f(t,x)-E_2f(t-s,x-y)\bigr)\,ds\,dy\Bigg| \\
			&\le \|\nabla_{(t,x)}E_2f\|_{L^{\infty}(\mathbb R^{d+1})}
			\int_{\mathbb{R}^{d+1}}\eta_{\varepsilon}(s,y)\,\|(s,y)\|\,ds\,dy  \\
			&\le \varepsilon\,C_{E_2}\|f\|_{W^{1,\infty}(q_T)} .
		\end{align*}
		Choosing \(\varepsilon\) so that \(\varepsilon C_{E_2}\|f\|_{W^{1,\infty}(Q_T)}<\delta\) yields
		\(\|\hat f\|_{L^{\infty}(Q_T)}\le\delta\).

		We repeat the estimate in \eqref{eq:YoungW1} and get that
		\begin{align}
			\|u_\varepsilon\|_{W^{3,\infty}(\mathbb R^{d+1})}
			\le \varepsilon^{-(d+1)/2-1}C_{E_1}
			\Bigl(\varepsilon\|\eta\|_{L^{2}(\mathbb R^{d+1})}+|\eta|_{H^1(\mathbb R^{d+1})}\Bigr)\,
			\|u\|_{H^{2}(Q_T)} . \label{eq:uEps}
		\end{align}
		Similarly,
		\begin{align}
			\|f_\varepsilon\|_{W^{2,\infty}(\mathbb R^{d+1})}
			\le \varepsilon^{-(d+1)/2-1}C_{E_2}
			\Bigl(\varepsilon\|\eta\|_{L^{1}(\mathbb R^{d+1})}+|\eta|_{W^{1,1}(\mathbb R^{d+1})}\Bigr)\,
			\|f\|_{W^{1,\infty}(q_T)} .
		\label{eq:fEps}
		\end{align}

		From \cite[Theorem~25]{yang2025deepneuralnetworksgeneral} there exist neural networks
		\(u_{\theta}\) and \(f_{\bar\theta}\) with width \(C_{1}\mathcal{W}\log\mathcal{W}\) and depth
		\(C_{2}\mathcal{L}\log\mathcal{L}\), where constants \(C_{1},C_{2}\) are independent of \(\mathcal{W},\mathcal{L}\) and \(\delta\), such that
		\begin{align*}
			\|u_\varepsilon-u_{\theta}\|_{W^{2,\infty}(Q_T)}
			&\le C_{Q_T}(d)\,\|u_\varepsilon\|_{W^{3,\infty}(Q_T)}\,
			\mathcal{W}^{-\frac{2}{d+1}}\mathcal{L}^{-\frac{2}{d+1}}, \\
			\|f_\varepsilon-f_{\bar\theta}\|_{L^{\infty}(Q_T)}
			&\le C_{Q_T}(d)\,\|f_\varepsilon\|_{W^{2,\infty}(Q_T)}\,
			\mathcal{W}^{-\frac{4}{d+1}}\mathcal{L}^{-\frac{4}{d+1}}. 
		\end{align*}
		Choosing \(\mathcal{W},\mathcal{L}\) sufficiently large that \(\mathcal{W}^{\frac{1}{d+1}}\mathcal{L}^{\frac{1}{d+1}}\ge \,\varepsilon^{-1-(d+1)/2}\) and using \eqref{eq:uEps}, \eqref{eq:fEps} we obtain
		\[
		\|u_\varepsilon-u_{\theta}\|_{H^{2}(Q_T)},
		\;\|f_\varepsilon-f_{\bar\theta}\|_{L^{\infty}(q_T)}
		\le \tilde C_0\,\mathcal{W}^{-\frac{1}{d+1}}\mathcal{L}^{-\frac{1}{d+1}},
		\]
		where \(\tilde C_0>0\) is a constant independent of \(\mathcal W,\mathcal L\) and \(\delta\).
		
		Let \(\mathcal L_v\) be the Lipschitz constant of \(v\) and let \(M\) be the constant from Assumption \ref{con:weight}.
		Using the bound on \(\hat{u}\) and the bound on \(\hat f\), we obtain
		\begin{align*}
			\|\bar{\mathcal{R}}_Q^{3}(u_\varepsilon,f_\varepsilon)\|_{Q}
			&\le \frac{M}{\sqrt{|Q_T|}}
			\Big(
			\|\partial_{t}\hat u\|_{L^{2}(Q_T)}
			+ \|\Delta \hat u\|_{L^{2}(Q_T)}+\|v(u)-v(u_\varepsilon)\|_{L^2(Q_T)}\\
			&\quad
			+ \|\hat f\|_{L^{\infty}(Q_T)}\|u\|_{L^{2}(Q_T)}
			+ \|f_\varepsilon\|_{L^{\infty}(Q_T)}\|\hat u\|_{L^{2}(Q_T)}
			\Big)\\
			&\quad
			+\|\phi_{\theta^1}-1\|_{Q}\|\partial_t u\|_{L^\infty(Q_T)}+\sum_{i=1}^{d}\|\phi_{\bar\theta^i}-1\|_{Q}\|\partial_{x_i x_i} u\|_{L^\infty(Q_T)}\nonumber\\
			&\quad
			+\|\phi_{\theta^2}-1\|_{Q}\|v(u)\|_{L^\infty(Q_T)}+\|\phi_{\theta^3}-1\|_{Q}\|f\|_{L^\infty(Q_T)}\|u\|_{L^\infty(Q_T)}.\nonumber\\
			&\le \frac{M}{\sqrt{|Q_T|}}
			\bigl(1+\mathcal L_v+\delta+\|f\|_{L^{\infty}(Q_T)}+\|u\|_{L^{2}(Q_T)}\bigr)\,\delta\\
			&\quad
			+\|\phi_{\theta^1}-1\|_{Q}\|\partial_t u\|_{L^\infty(Q_T)}+\sum_{i=1}^{d}\|\phi_{\bar\theta^i}-1\|_{Q}\|\partial_{x_i x_i} u\|_{L^\infty(Q_T)}\nonumber\\
			&\quad+\|\phi_{\theta^2}-1\|_{Q}\|v(u)\|_{L^\infty(Q_T)}+\|\phi_{\theta^3}-1\|_{Q}\|f\|_{L^\infty(Q_T)}\|u\|_{L^\infty(Q_T)}.\nonumber
		\end{align*}
		Now, we estimate
		\begin{align*}
			\|\bar{\mathcal{R}}_Q^{3}(u_{\theta},f_{\bar\theta})\|_{Q}
			&\le \|\bar{\mathcal{R}}_Q^{3}(u_{\theta},f_{\bar\theta})
			-\bar{\mathcal{R}}_Q^{3}(u_\varepsilon,f_\varepsilon)\|_{Q}
			+ \|\bar{\mathcal{R}}_Q^{3}(u_\varepsilon,f_\varepsilon)\|_{Q} \\
			&\le \frac{M}{\sqrt{|Q_T|}}(1+\tilde C_0)
			\bigl(1+\mathcal L_v+\delta+\|f\|_{L^{\infty}(Q_T)}+\|u\|_{L^{2}(Q_T)}\bigr)
			\bigl(\mathcal{W}^{-\frac{1}{d+1}}\mathcal{L}^{-\frac{1}{d+1}}+\delta\bigr)\\
			&\quad
			+\|\phi_{\theta^1}-1\|_{Q}\|\partial_t u\|_{L^\infty(Q_T)}+\sum_{i=1}^{d}\|\phi_{\bar\theta^i}-1\|_{Q}\|\partial_{x_i x_i} u\|_{L^\infty(Q_T)}\nonumber\\
			&\quad+\|\phi_{\theta^2}-1\|_{Q}\|v(u)\|_{L^\infty(Q_T)}+\|\phi_{\theta^3}-1\|_{Q}\|f\|_{L^\infty(Q_T)}\|u\|_{L^\infty(Q_T)}.\nonumber
		\end{align*}

		Since \[\rho\ge(d+4)\max\{\|u\|^2_{W^{2,\infty}(Q_T)},\|v(u)\|^2_{L^\infty(Q_T)},\|f\|^2_{L^\infty(Q_T)}\|u\|^2_{L^\infty(Q_T)}\},\]
		by the inequality \( (a_1+\dots+a_{d+4})^2 \le (d+4)(a_1^2+\dots+a_{d+4}^2)\), it follows that
			\begin{align*}
				\|\bar{\mathcal{R}}_Q^3(u_\theta,f_{\bar{\theta}})\|^2_{Q}&\le  \tilde C_1 (\mathcal{W}^{-\frac{1}{d+1}}\mathcal{L}^{-\frac{1}{d+1}} + \delta)^2\\
				&\quad +\rho\Big(\|\phi_{\theta^1}-1\|_{Q}^2+\sum_{i=1}^{d}\|\phi_{\bar\theta^i}-1\|^2_{Q}+\|\phi_{\theta^2}-1\|_{Q}^2+\|\phi_{\theta^3}-1\|_{Q}^2\Big)\nonumber
			\end{align*}
		where \(\tilde C_1 > 0\) is independent of \(\mathcal{W}\), \(\mathcal{L}\), and \(\delta\).
		
		Repeat the estimate in \eqref{eq:est_R1G_NN}, we obtain
		\begin{align*}
			\lambda\|\bar{\mathcal R}_{\Gamma}^3(u_\theta)\|^2_{\Gamma}
			&\le \tilde C_2 (\mathcal{W}^{-\frac{1}{d+1}}\mathcal{L}^{-\frac{1}{d+1}} + \delta)^2\\
			&\quad+\rho\Big(\|\phi_{\theta^5}(0,\cdot)-1\|^2_{L^2(\Omega)}+\|\phi_{\theta^6}(0,\cdot)-1\|^2_{L^2(\Omega)}\nonumber\\
			&\quad\quad+\|\phi_{\theta^7}(T,\cdot)-1\|^2_{L^2(\Omega)}+\|\phi_{\theta^8}(T,\cdot)-1\|^2_{L^2(\Omega)}\Big),\nonumber
		\end{align*}
		where \(\tilde C_2>0\) is independent of \(\mathcal{W},\mathcal L\) and \(\delta\).

		Therefore, there exists a constant \(C_3>0\) independent of \(\mathcal W,\mathcal L\) and \(\delta\) such that
		\begin{align*}
			\|\bar{\mathcal{R}}_Q^{3}\|^2_{Q}+\lambda\|\bar{\mathcal{R}}_\Gamma^{3}\|^2_{\Gamma}-\rho\|\mathcal R_{NN}^{3}\|_{NN}^2
			\le C_3 (\mathcal{W}^{-\frac{1}{d+1}}\mathcal{L}^{-\frac{1}{d+1}} + \delta)^2
		\end{align*}
		for any given parameters \((\theta^i)_{i=1}^{8},(\bar{\theta}^i)_{i=1}^{d}\).		
	\end{proof}

	We now state the corresponding results for bilinear control of the wave equation.
	\begin{Theorem}
		Assume that there exists a bilinear control \(f\in W^{1,\infty}(Q_T)\)  and a corresponding solution \(u\in W^{2,\infty}(Q_T)\) governed by \eqref{Wave1b} with initial and terminal conditions \( (u_0,u_1),(z_0,z_1)\in W^{1,\infty}(\Omega)\times L^\infty(\Omega)\).
		The constants \(\lambda\) and \(\rho\) are given such that 
		\begin{align*}
		\rho&\ge\max\{3\lambda\|u_0\|^2_{W^{1,\infty}(\Omega)},3\lambda\|u_1\|^2_{L^\infty(\Omega)},3\lambda\|z_0\|^2_{W^{1,\infty}(\Omega)},3\lambda\|z_1\|^2_{L^\infty(\Omega)}\\
&\qquad\quad(d+4)\|u\|^2_{W^{2,\infty}(Q_T)},(d+4)\|v(u)\|^2_{L^\infty(Q_T)},(d+4)\|f\|^2_{L^\infty(Q_T)}\|u\|^2_{L^\infty(Q_T)}\}.\end{align*}
		
		Suppose that the activation function \(\sigma\) satisfies Assumption \ref{con:sigma}, or that \(\sigma = \mathrm{ReLU}^{m+1}\) with \(m \ge 2\), and suppose that the weight networks \(\phi_{\theta^i},\phi_{\bar \theta^i},\phi_{\tilde \theta^i},\phi_{\hat \theta^i},\phi_{\check \theta^i}\) satisfy Assumption \ref{con:weight}.
		
		Then, for any \(\delta > 0\), and sufficiently large \(\mathcal{W}, \mathcal{L} \in \mathbb{N}\) such that \(\log_2(\mathcal{W}) \le \mathcal{L}\), there exist neural networks \(u_{\theta}\) and \(f_{\bar{\theta}}\), constructed using \(\sigma\), with width \(C_1 \mathcal{W}\log \mathcal{W}\) and depth \(C_2 \mathcal{L}\log \mathcal{L}\), respectively, such that
		\begin{align*}
			\sup_{(\theta^i)_{i=1}^{12},(\bar{\theta}^i)_{i=1}^d,(\tilde{\theta}^i)_{i=1}^{d},(\hat{\theta}^i)_{i=1}^{2d},(\check{\theta}^i)_{i=1}^{2d}}\|\bar{\mathcal{R}}_Q^{4}\|^2_{Q}+\lambda\|\bar{\mathcal{R}}_\Gamma^{4}\|^2_{\Gamma}-\rho\|\mathcal R_{NN}^{4}\|_{NN}^2
			\le C_3 (\mathcal{W}^{-\frac{1}{d+1}}\mathcal{L}^{-\frac{1}{d+1}} + \delta)^2.
		\end{align*}
		Here, \(\lambda,\rho>0\) are given constants and \(C_i\) for \(i=1,2,3\) are constants independent of \(\mathcal{W}\), \(\mathcal{L}\), and \(\delta\). 
		The residuals \(\bar{\mathcal{R}}_Q^{4},\bar{\mathcal{R}}_\Gamma^{4}\), and \(\mathcal{R}_{NN}^4\) are defined in \eqref{eq:barRQ4}, \eqref{eq:barRG4}, and \eqref{eq:RNN4}.
		\label{thm:errorWaveb}
	\end{Theorem}
	
	\begin{Remark}
		Similar to Theorem \ref{thm:errorWave}, to ensure the well-posedness and stability of the approximation scheme, we assume that \(u \in W^{2,\infty}(Q_T)\) and \(f \in W^{1,\infty}(Q_T)\).
	\end{Remark}
	
	\begin{proof}
		The proof of Theorem \ref{thm:errorWaveb} is analogous to that of Theorem \ref{thm:errorHeatb}. The differences comparing to the proof of Theorem \ref{thm:errorHeatb} are:
		\begin{enumerate}
			\item The time derivative \(\partial_t\) is modified into \(\partial_{tt}\).
			\item Repeat the estimate \eqref{eq:est_init} for \(u_1,z_1\) and all directional derivatives of \(u_0,z_0\).
			\item The boundary is estimated by
				\begin{align*}
					\|\phi_{\theta^4}u_\theta\|_{L^2(\Sigma_T)}^2+\|\phi_{\tilde\theta^1}\partial_tu_\theta\|^2_{H^{1/2}(\Sigma_T)}+\sum_{i=2}^{d}\bigl\|\phi_{\tilde\theta^i}\frac{\partial u_\theta}{\partial\sigma_{i-1}}\bigr\|^2_{H^{1/2}(\Sigma_T)}
					\le CM^2\|u_\theta\|^2_{H^{3/2}(\Sigma_T)}.
				\end{align*}
				\item We use the Trace Theorem for fractional Sobolev space \cite[Theorem 3.37]{McLean2000}, which yields
					\[
						\|u_\theta-u\|^2_{H^{3/2}(\Gamma)}\le \tilde C\|u_\theta-u\|_{H^2(Q_T)}^2.
\qedhere
					\]
		\end{enumerate}
	\end{proof}

	\section{Numerical results: High dimensional Problems}\label{Sec:Numerical}
	\subsection{Set up}

	We fix the dimension \(d = 10\). We consider several configurations of the domain \(\Omega\) and the controlled region \(\omega\).

	We set the constant \(\lambda = \rho = 1\) in the formulation
	\[\|\bar{\mathcal{R}}_Q\|_{Q}^2+\lambda\|\bar{\mathcal{R}}_{\Gamma}\|_{\Gamma}^2-\rho\|\mathcal{R}_{NN}\|_{NN}^2.\]

	In the numerical experiments, derivatives and the Laplace operator are approximated using Finite Difference (FD) schemes with step size $h = 0.001$.
	We denotes the time derivative, second-order derivative and the Laplace in the FD schemes by \(\bar\partial_t,\bar\partial_{tt},\bar{\Delta}_x\), respectively.

	Now, we set
	\[
	v(u) := -u + u^2.
	\]
The numerical experiments are conducted in a locally Lipschitz regime. The theoretical assumptions are not directly satisfied by this choice of \(v\), but the experiments remain stable because the learned states remain uniformly bounded.

	\begin{itemize}
		\item \textbf{Internal control:}
		\begin{itemize}
			\item[$\circ$] \textit{Heat equation:} We set $T = 1$ and define
			\begin{equation}\label{eq:heat_disc}
				\mathscr{P}_h u := \bar{\partial}_t u - \bar{\Delta}_x u + v(u).
			\end{equation}
			The discrete equation is
			\[
			\mathscr{P}_h u = f \mathbf{1}_\omega.
			\]
			
			\item[$\circ$] \textit{Wave equation:} We set $T = 3$ and define
			\begin{equation}\label{eq:wave_disc}
				\mathscr{P}_w u := \bar{\partial}_{tt} u - \bar{\Delta}_x u + v(u).
			\end{equation}
			The discrete equation is
			\[
			\mathscr{P}_w u = f \mathbf{1}_\omega.
			\]
		\end{itemize}
		
		\item \textbf{Bilinear control:}
		\begin{itemize}
			\item[$\circ$] \textit{Heat equation:} We set $T = 1$ and use $\mathscr{P}_h$ as in \eqref{eq:heat_disc}. The discrete equation is
			\[
			\mathscr{P}_h u = f u.
			\]
			
			\item[$\circ$] \textit{Wave equation:} We set $T = 3$ and use $\mathscr{P}_w$ as in \eqref{eq:wave_disc}. The discrete equation is
			\[
			\mathscr{P}_w u = f u.
			\]
		\end{itemize}
	\end{itemize}

	We next describe the architectures of the neural networks.
	
	We recall the ReLU activation function $\sigma:\mathbb{R}\to[0,\infty)$, defined by
	\begin{align}\label{ReLU}
		\mathrm{ReLU}(x) := \max\{0,x\}.
	\end{align}
	
	To ensure twice differentiability, we use the activation function $\mathrm{ReLU}^3$. This function, denoted by $\sigma_3$, is defined by
	\begin{align}\label{ReLU3}
		\sigma_3(x) := \mathrm{ReLU}(x)^3.
	\end{align}
	
	We also recall the sigmoid activation function $\sigma_s:\mathbb{R}\to(0,1)$, defined by
	\begin{align}\label{sigmoid}
		\sigma_s(x) := \frac{1}{1+e^{-x}}.
	\end{align}

	We consider two types of neural networks with different activation functions: one for $u_\theta$ and $f_{\bar{\theta}}$, and the other for the weight networks $\phi_{\theta^i}$.
	
	\begin{itemize}
		\item \textbf{Solution and control networks:}
		These networks have width $\mathcal{W} = 100$ and depth $\mathcal{L} = 3$, and use $\sigma_3$ as the activation function.
		
		\item \textbf{Weight networks:}
			These networks have width $\mathcal{W} = 40$ and depth $\mathcal{L} = 3$. We use $\sigma_3$ as the activation function in the first two hidden layers and $\sigma_s$ in the final hidden layer. This structure allows us to control the range of the weight networks. In the numerical experiments, the output of the sigmoid layer is rescaled affinely to \([0.2,5]\) via \(\phi=0.2+4.8\sigma_s(\cdot)\). The aforementioned structure satisfies the condition mentioned in Remark \ref{rmk:uniformboundweight}.
	\end{itemize}

	Finally, we describe the training and testing procedures.
	
	{\bf Training procedures:} We consider eight different configurations. For each configuration, we perform $10000$ training iterations. At each iteration, the following steps are carried out:
	\begin{enumerate}
		\item We randomly generate $N_1 = 1000$ training points $(t_i,x_i)$ in the interior of the domain, and $N_1$ training points $(\hat{t}_i,\hat{x}_i)$ on the boundary. We also generate $N_2 = 100$ training points $(0,\bar{x}_i)$ for the initial condition and $N_2$ training points $(T,\tilde{x}_i)$ for the terminal condition.
		
		\item These training points are used to evaluate the loss functions in \eqref{eq:lossheat} and \eqref{eq:losswave}, which correspond to discrete approximations of
		\[
		\|\bar{\mathcal{R}}_Q\|_{Q}^2
		+ \|\bar{\mathcal{R}}_\Gamma\|_{\Gamma}^2
		- \|\mathcal{R}_{NN}\|_{NN}^2,
		\]
		where $\bar{\mathcal{R}}_Q$, $\bar{\mathcal{R}}_\Gamma$, and $\mathcal{R}_{NN}$  are defined in Sections~\ref{Sec:internalsettings} and~\ref{Sec:bilinearsettings}. For simplicity, we use the $L^2$-norm formulation of the loss.

	\item We then employ the ADAM optimizer implemented in the \texttt{torch} package to solve the resulting min--max optimization problem. We set the learning rate at \(0.001\).
	\end{enumerate}

	{\bf Testing errors:} We randomly generate $N_3 = 10000$ test points in the interior and on the boundary, and $N_4 = 1000$ test points for the initial and terminal conditions. These test points are used to evaluate the approximation error in the numerical experiments.

	{\bf Presenting results:} We display the training loss as a function of the number of iterations
and report the testing error for each numerical experiment.  In the
high-dimensional tests of Sections~\ref{Sec:Internal} and
\ref{Sec:Bilinear}, the learned control and the corresponding state
are computed but are not displayed, due to the high-dimensional nature
of the problem.  The reported testing errors in these sections are
therefore  {\it residual-based quantities, measuring the PDE residual together
with the boundary, initial, and terminal constraint errors}.  In the
lower-dimensional tests of Section~\ref{Sec:NumericalOtherMethodFull},
we instead  {\it evaluate the reconstruction errors directly against known
manufactured solutions.}
	
	\subsection{Internal control problems}\label{Sec:Internal}
	
	\subsubsection{Heat equation}\label{Sec:Heateq}
	
	For the heat equation, the loss \eqref{eq:lossheat} is numerically evaluated using training points and the \(H^{1/2}\)-norm is reduced to \(L^2\)-norm.

	For the standard PINN formulation, we set $\phi_{\theta^j},\phi_{\bar\theta^j} \equiv 1$ for all $j$.

	Using the test points, we numerically evaluate the error as
	\begin{align*}
		H_h
		&= \|\mathscr{P}_h u - f \mathbf{1}_\omega\|_{L^2([0,1]\times\Omega)}
		+ \|u\|_{L^2([0,1]\times\partial\Omega)}
		+ \|u(0,\cdot) - u_0\|_{L^2(\Omega)}
		+ \|u(T,\cdot) - z_0\|_{L^2(\Omega)}.
	\end{align*}
	
	For convenience, we refer to $\|\mathscr{P}_h u - f \mathbf{1}_\omega\|_{L^2([0,1]\times\Omega)}$ as the \emph{equation error}, and the remaining terms as the \emph{boundary error}.
	
	\begin{enumerate}
		\item \textbf{Situation 1:} The domain $\Omega$ is the unit ball in $\mathbb{R}^d$, and the control region is given by
		\[
		\omega = \Big(-\sqrt{\frac{1}{d}-\frac{1}{2d^2}},\sqrt{\frac{1}{d}-\frac{1}{2d^2}}\Big)^{d}.
		\]
		
		We consider the initial condition $u_0 = 0$ and the terminal condition
		\[
		z_0 = (e-1)\sin\!\left(\frac{\pi}{2}(1-|x|)^{2.5}\right).
		\]

		In Figure~\ref{fig:heat1}, we plot the total training loss as a function of the number of iterations for both methods (blue: PINN, red: WeightedPINN).
		
		\begin{figure}[H]
			\centering
			\includegraphics[width=.6\textwidth]{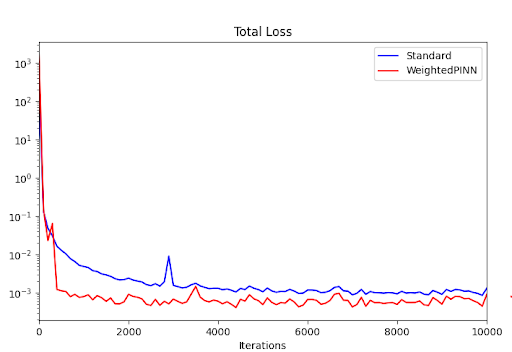}
			\caption{Training loss over iterations in Situation 1}
			\label{fig:heat1}
		\end{figure}
		
		We observe in Figure~\ref{fig:heat1} that the initial loss is at \(10^3\) and decreases to approximately \(10^{-3}\) by the \(10000\)-th iteration, demonstrating the effectiveness of both the standard PINN and WeightedPINN.
		
		We further observe that the WeightedPINN achieves a lower total training loss compared to the standard PINN.
		
		Table~\ref{tab:heat1} reports the test errors at the final iteration.
		
		\begin{table}[H]
			\centering
			\begin{tabular}{c|c|c|c}
				& Total & Eqn. & Bnd. \\
				\hline
				Standard & \(6.082\times10^{-2}\) & \(2.004\times10^{-2}\) & \(4.078\times10^{-2}\) \\
				\hline
				WeightedPINN & \(4.411\times10^{-2}\) & \(3.425\times10^{-3}\) & \(4.068\times10^{-2}\)
			\end{tabular}
			\caption{Test errors at the final iteration in Situation 1}
			\label{tab:heat1}
		\end{table}
		
		From Table~\ref{tab:heat1}, both methods achieve relatively small overall errors. However, the WeightedPINN significantly reduces the equation error, while the boundary error remains comparable between the two approaches.

		\item \textbf{Situation 2:} The domain is $\Omega = (-1,1)^d$, and the control region $\omega$ is the unit ball in $\mathbb{R}^d$. We consider the initial condition
		\[
		u_0 = \prod_{i=1}^{d} (1 - |x_i|),
		\]
		and the terminal condition
		\[
		z_0 = e \prod_{i=1}^{d} (1 - |x_i|).
		\]
		
		Similar to Situation 1, WeightedPINN achieves improved numerical performance in this setting.
		
		In Figure~\ref{fig:heat2}, we plot the total training loss as a function of the number of iterations for both methods (blue: PINN, red: WeightedPINN).
		
		\begin{figure}[H]
			\centering
			\includegraphics[width=0.6\textwidth]{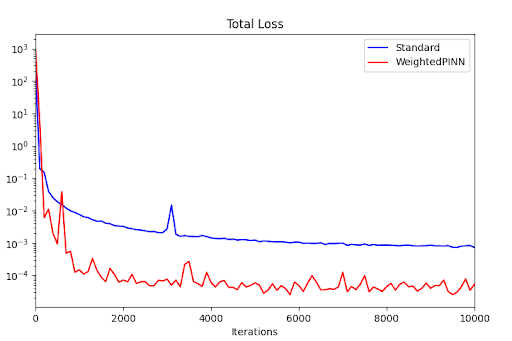}
			\caption{Training loss over iterations in Situation 2}
			\label{fig:heat2}
		\end{figure}
		
		We observe in Figure~\ref{fig:heat2} that the initial loss is at \(10^3\). The loss of the standard PINN method decreases to approximately \(10^{-3}\) by the \(10000\)-th iteration and the loss of the WeightedPINN method reaches approximately \(10^{-4}\) by the \(10000\)-th iteration. This again shows the effectiveness of both methods, especially the WeightedPINN.

		Table~\ref{tab:heat2} reports the test errors at the final iteration.
		
		\begin{table}[H]
			\centering
			\begin{tabular}{c|c|c|c}
				&Total&Eqn.&Bnd.\\
				\hline
				Standard&\(5.114\times10^{-2}\)&\(2.733\times10^{-2}\)&\(2.382\times10^{-2}\)\\
				\hline
				WeightedPINN&\(2.299\times10^{-2}\)&\(3.171\times10^{-3}\)&\(1.982\times10^{-2}\)
			\end{tabular}
			\caption{Test errors at the final iteration in Situation 2}
			\label{tab:heat2}
		\end{table}

	\end{enumerate}
	
	\subsubsection{Wave equation}\label{Sec:Waveeq}
	
	For the wave equation, the loss \eqref{eq:losswave} is numerically evaluated using the training points and the \(H^{3/2}\)-norm on \(\Sigma_T\), the \(H^1\)-norms on \(u_0,z_0\) are reduced to \(L^2\)-norms.

	Similarly to Section~\ref{Sec:Heateq}, for the wave equation we numerically evaluate the error as
	\begin{align*}
		H_w
		&= \|\mathscr{P}_w u - f \mathbf{1}_{\omega}\|_{L^2([0,T]\times\Omega)}
		+ \|u\|_{L^2([0,T]\times\partial\Omega)} \\
		&\quad
		+ \|(u(0,\cdot),\bar{\partial}_t u(0,\cdot))  - (u_0,u_1)\|_{L^2(\Omega)\times L^2(\Omega)}
		+ \|(u(T,\cdot),\bar{\partial}_t u(T,\cdot)) - (z_0,z_1)\|_{L^2(\Omega)\times L^2(\Omega)}
	\end{align*}
	We refer to $\|\mathscr{P}_w u - f \mathbf{1}_{\omega}\|_{L^2([0,T]\times\Omega)}$ as the equation error, and to the remaining terms collectively as the boundary error.
	
	\begin{enumerate}
		\item \textbf{Situation 3:} The domain $\Omega$ is the unit ball in $\mathbb{R}^d$, and the control region $\omega$ is defined such that $\Omega \setminus \omega$ is the closed ball centered at $(1/d,\dots,1/d)$ with radius $3/4$. We choose $\omega$ so that Assumption~\ref{con:wave} is satisfied. For simplicity, we further assume that $\omega$ contains the entire boundary $\partial\Omega$.
		
		We consider the initial and terminal conditions
		\begin{align*}
			u_0 &= u_1 = 0, \\
			z_0 &= (e-1)\sin\!\left(\frac{\pi}{2}(1-|x|)^{2.5}\right), \\
			z_1 &= 2e\,\sin\!\left(\frac{\pi}{2}(1-|x|)^{2.5}\right).
		\end{align*}
		
		Figure~\ref{fig:wave1} shows the total training loss in Situation 3 (blue: PINN, red: WeightedPINN).
		
		\begin{figure}[H]
			\centering
			\includegraphics[width=0.6\textwidth]{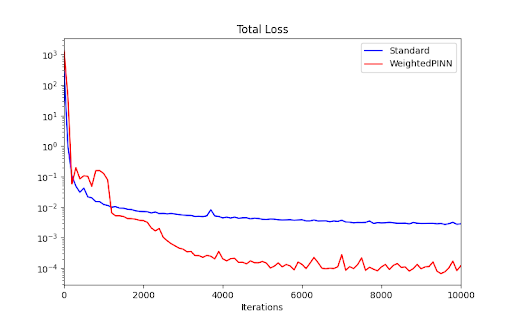}
			\caption{Training loss over iterations in Situation 3}
			\label{fig:wave1}
		\end{figure}
		
		In this situation, the WeightedPINN requires more iterations to stabilize; however, it eventually achieves a lower training loss compared to the standard PINN.
		
		We observe in Figure~\ref{fig:wave1} that the loss of the standard PINN method can only reach approximately \(10^{-2}\) at \(10000\)-th iteration. On the other hand, the loss of WeightedPINN can reach approximately \(10^{-4}\) at \(10000\)-th iteration.
		
		Table~\ref{tab:wave1} reports the test errors at the final iteration.
		
		\begin{table}[H]
			\centering
			\begin{tabular}{c|c|c|c}
				& Total & Eqn. & Bnd. \\
				\hline
				Standard & \(1.745\times10^{-1}\) & \(3.611\times10^{-2}\) & \(1.384\times10^{-1}\) \\
				\hline
				WeightedPINN & \(4.833\times10^{-2}\) & \(3.505\times10^{-3}\) & \(4.482\times10^{-2}\)
			\end{tabular}
			\caption{Test errors at the final iteration in Situation 3}
			\label{tab:wave1}
		\end{table}
		
		In Situation 3, the WeightedPINN demonstrates significant improvements in both the equation error and the boundary error.

		\item \textbf{Situation 4:} The domain is $\Omega = (-1,1)^d$, and the control region $\omega$ is defined such that
		\[
		\Omega \setminus \omega
		=
		\Big[-\sqrt{\frac{1}{d}-\frac{1}{2d^2}},\sqrt{\frac{1}{d}-\frac{1}{2d^2}}\Big]^d.
		\]
		
		We consider the initial and terminal conditions:
		\begin{align*}
			u_0 &= u_1 = \prod_{i=1}^{d}(1 - |x_i|), \\
			z_0 &= z_1 = e \prod_{i=1}^{d}(1 - |x_i|).
		\end{align*}

		Figure~\ref{fig:wave2} shows the total training loss in Situation 4 (blue: PINN, red: WeightedPINN).
		
		\begin{figure}[H]
			\centering
			\includegraphics[width=0.6\textwidth]{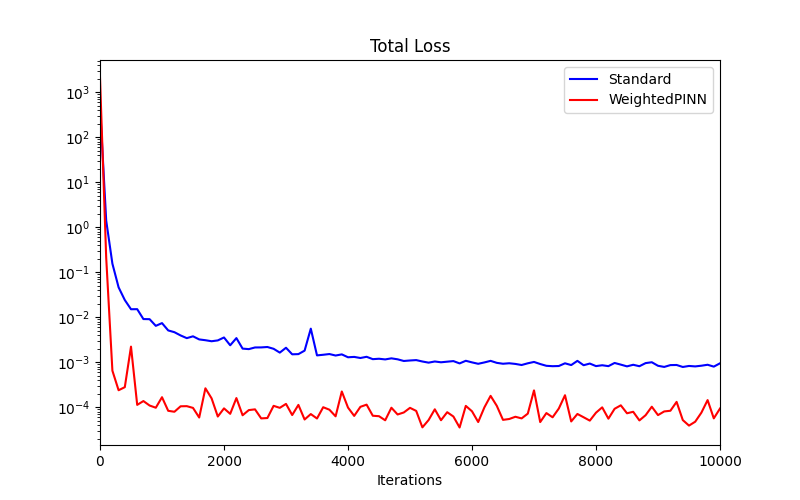}
			\caption{Training loss over iterations in Situation 4}
			\label{fig:wave2}
		\end{figure}
		
		In figure~\ref{fig:wave2}, after \(10000\) iterations, the loss of the standard PINN method reaches approximately \(10^{-3}\) and the loss of the WeightedPINN method reaches approximately \(10^{-4}\).
		
		Table~\ref{tab:wave2} reports the test errors at the final iteration.
		
		\begin{table}[H]
			\centering
			\begin{tabular}{c|c|c|c}
				& Total & Eqn. & Bnd. \\
				\hline
				Standard & \(7.561\times10^{-2}\) & \(3.047\times10^{-2}\) & \(4.514\times10^{-2}\) \\
				\hline
				WeightedPINN & \(4.225\times10^{-2}\) & \(2.566\times10^{-3}\) & \(3.968\times10^{-2}\)
			\end{tabular}
			\caption{Test errors at the final iteration in Situation 4}
			\label{tab:wave2}
		\end{table}
		
		The WeightedPINN consistently achieves better numerical performance in this setting.

	\end{enumerate}
	
	\subsection{Bilinear control problems}\label{Sec:Bilinear}

	The loss functions and error metrics for the bilinear control problems are identical to those in Section \ref{Sec:Internal}, with the only difference being that \(f_{\bar{\theta}}(t_i,x_i)\mathbf 1_\omega(x_i)\) is replaced by \(f_{\bar{\theta}}(t_i,x_i)u_\theta(t_i,x_i)\) throughout. We test the same four situations with the same domain and initial/terminal conditions. For brevity, we directly report the numerical results.
	
	\subsubsection{Heat equation}
	
	\begin{enumerate}
		\item \textbf{Situation 5:} We consider the same domain $\Omega$, initial condition $u_0$, and terminal condition $z_0$ as in Situation 1 of Section~\ref{Sec:Heateq}.
		
		Figure~\ref{fig:bheat1} shows the total training loss in Situation 5 (blue: PINN, red: WeightedPINN).
		
		\begin{figure}[H]
			\centering
			\includegraphics[width=.6\textwidth]{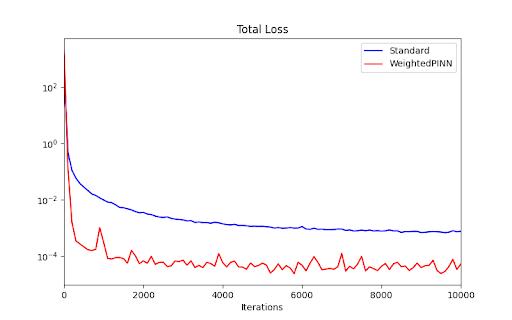}
			\caption{Training loss over iterations in Situation 5}
			\label{fig:bheat1}
		\end{figure}
		
		Table~\ref{tab:bheat1} reports the test errors at the final iteration.
		
		\begin{table}[H]
			\centering
			\begin{tabular}{c|c|c|c}
				& Total & Eqn. & Bnd. \\
				\hline
				Standard & \(4.801\times10^{-2}\) & \(2.572\times10^{-2}\) & \(2.228\times10^{-2}\) \\
				\hline
				WeightedPINN & \(2.263\times10^{-2}\) & \(2.780\times10^{-3}\) & \(1.985\times10^{-2}\)
			\end{tabular}
			\caption{Test errors at the final iteration in Situation 5}
			\label{tab:bheat1}
		\end{table}

		\item \textbf{Situation 6:} We consider the same domain $\Omega$, initial condition $u_0$, and terminal condition $z_0$ as in Situation 2 of Section~\ref{Sec:Heateq}.
		
		Figure~\ref{fig:bheat2} shows the total training loss in Situation 6 (blue: PINN, red: WeightedPINN).
		
		\begin{figure}[H]
			\centering
			\includegraphics[width=.6\textwidth]{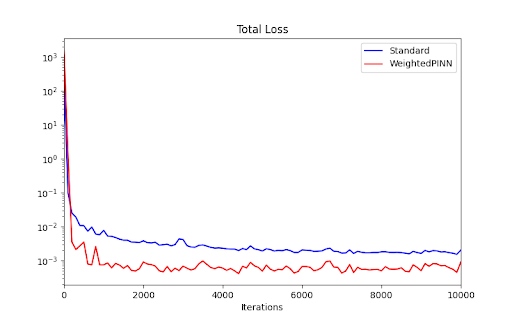}
			\caption{Training loss over iterations in Situation 6}
			\label{fig:bheat2}
		\end{figure}
		
		Table~\ref{tab:bheat2} reports the test errors at the final iteration.
		
		\begin{table}[H]
			\centering
			\begin{tabular}{c|c|c|c}
				& Total & Eqn. & Bnd. \\
				\hline
				Standard & \(8.053\times10^{-2}\) & \(3.106\times10^{-2}\) & \(4.948\times10^{-2}\) \\
				\hline
				WeightedPINN & \(4.358\times10^{-2}\) & \(2.950\times10^{-3}\) & \(4.063\times10^{-2}\)
			\end{tabular}
			\caption{Test errors at the final iteration in Situation 6}
			\label{tab:bheat2}
		\end{table}

	\end{enumerate}
	
	\subsubsection{Wave equation}
	
	\begin{enumerate}
		\item \textbf{Situation 7:} We consider the same domain $\Omega$, initial condition, and terminal condition as in Situation 3 of Section~\ref{Sec:Waveeq}.
		
		Figure~\ref{fig:bwave1} shows the total training loss in Situation 7 (blue: PINN, red: WeightedPINN).
		
		\begin{figure}[H]
			\centering
			\includegraphics[width=.6\textwidth]{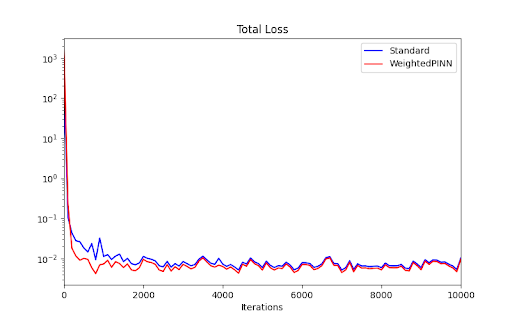}
			\caption{Training loss over iterations in Situation 7}
			\label{fig:bwave1}
		\end{figure}
		
		Table~\ref{tab:bwave1} reports the test errors at the final iteration.
		
		\begin{table}[H]
			\centering
			\begin{tabular}{c|c|c|c}
				& Total & Eqn. & Bnd. \\
				\hline
				Standard & \(1.986\times10^{-1}\) & \(2.350\times10^{-2}\) & \(1.751\times10^{-1}\) \\
				\hline
				WeightedPINN & \(1.755\times10^{-1}\) & \(1.005\times10^{-2}\) & \(1.655\times10^{-1}\)
			\end{tabular}
			\caption{Test errors at the final iteration in Situation 7}
			\label{tab:bwave1}
		\end{table}

		\item \textbf{Situation 8:} We consider the same domain $\Omega$, initial condition, and terminal condition as in Situation 4 of Section~\ref{Sec:Waveeq}.
		
		Figure~\ref{fig:bwave2} shows the total training loss in Situation 8 (blue: PINN, red: WeightedPINN).
		
		\begin{figure}[H]
			\centering
			\includegraphics[width=.6\textwidth]{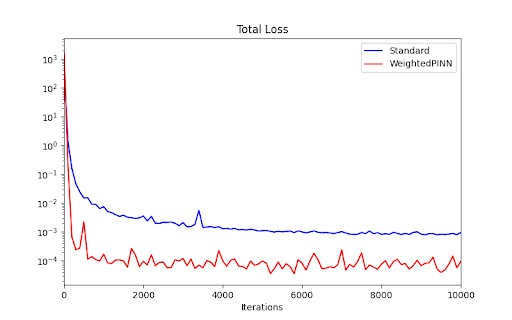}
			\caption{Training loss over iterations in Situation 8}
			\label{fig:bwave2}
		\end{figure}
		
		Table~\ref{tab:bwave2} reports the test errors at the final iteration.
		
		\begin{table}[H]
			\centering
			\begin{tabular}{c|c|c|c}
				& Total & Eqn. & Bnd. \\
				\hline
				Standard & \(7.538\times10^{-2}\) & \(3.024\times10^{-2}\) & \(4.514\times10^{-2}\) \\
				\hline
				WeightedPINN & \(4.258\times10^{-2}\) & \(2.824\times10^{-3}\) & \(3.975\times10^{-2}\)
			\end{tabular}
			\caption{Test errors at the final iteration in Situation 8}
			\label{tab:bwave2}
		\end{table}

	\end{enumerate}
	\textbf{Conclusion:} Under the settings of our numerical experiments, both the standard PINN and the WeightedPINN achieve small total losses, demonstrating the effectiveness of the proposed approach for solving the considered control problems.
	
	Moreover, across the high-dimensional situations considered here for
both internal and bilinear control problems, WeightedPINN yields lower
residual-based testing errors than the standard PINN baseline, with
the most pronounced improvement usually appearing in the equation
residual.

\section{Numerical Results: Comparison with other PINN-Type Methods}\label{Sec:NumericalOtherMethodFull}

  \subsection{Comparison with  Loss-balanced PINN, RAR-D PINN and Causal PINN Methods}\label{Sec:NumericalOtherMethod}

The high-dimensional experiments of Sections~\ref{Sec:Internal} and
\ref{Sec:Bilinear} indicate that WeightedPINN yields smaller
residual-based testing errors than the standard PINN baseline for the
high-dimensional control problems considered. 
To probe the operator-level adaptive mechanism of WeightedPINN in a
controlled one-dimensional setting, we now design a benchmark in
which the residual operators \(\partial_t u\), \(\partial_{xx} u\),
\(v(u)\), and \(f\,u\) have markedly different magnitudes by
construction. We then compare WeightedPINN against four reference
methods: the standard PINN of \cite{Raissi2019}, the loss-balanced
PINN of \cite{WangTengPerdikaris2021}, the residual-adaptive (RAR-D)
sampler of \cite{WuZhuTan2023_AdaptiveSamplingPINN}, and the causal
PINN of \cite{WangSankaranPerdikaris2024_CausalPINN}. The
self-adaptive PINN of \cite{McClennyBragaNeto2023}, which acts at
the level of individual collocation points rather than at the
operator level, is treated separately. {\it In contrast to Sections~\ref{Sec:Internal} and
\ref{Sec:Bilinear}, where the learned control and state are not
displayed due to the high-dimensional architecture, in the present
section both the state and the control are presented, and their
relative reconstruction errors are computed.}

\subsubsection{A singularly perturbed bilinear-control heat problem.}
We consider the bilinear-control heat equation in
\((0,T)\times\Omega\) with \(T=1\) and \(\Omega=(0,1)\), in which
the diffusion coefficient is a small parameter \(\varepsilon\):
\begin{equation}
\label{eq:stiff-pde}
\begin{aligned}
\partial_t u - \varepsilon\,\partial_{xx} u + v(u) &= f\,u
&& \text{in }(0,T)\times\Omega,\\
u(t,0) = u(t,1) &= 0 && \text{for }t\in(0,T),\\
u(0,x) &= u_0(x) && \text{in }\Omega,\\
u(T,x) &= z_0(x) && \text{in }\Omega,
\end{aligned}
\end{equation}
with \(v(u) = -u + u^2\), as in the rest of the paper, and
\(\varepsilon = 10^{-2}\). We use the  manufactured solution
\[
u^\star(t,x) = \sin(\pi x)\exp\!\bigl(\sin(2\pi t)+t\bigr),
\]
which yields the reference control
\[
f^\star(t,x) = 2\pi\cos(2\pi t) + \varepsilon\,\pi^2 - 1 + u^\star(t,x),
\]
together with \(u_0(x) = \sin(\pi x)\) and \(z_0(x) = e\sin(\pi x)\).

\subsubsection{Operator imbalance.} The motivation for this choice of
\(\varepsilon\) is the following. Along the truth,
\begin{equation*}
\frac{\partial_t u^\star}{u^\star}
= 2\pi\cos(2\pi t)+1,\qquad
\frac{\varepsilon\,\partial_{xx}u^\star}{u^\star}
= -\varepsilon\,\pi^2,\qquad
\frac{v(u^\star)}{u^\star}
= -1+u^\star,\qquad
\frac{f^\star u^\star}{u^\star}
= f^\star,
\end{equation*}
so that the magnitude of the diffusion contribution
\(|\varepsilon\,\partial_{xx} u^\star|\) is of order
\(\varepsilon\,\pi^2 \approx 10^{-1}\), while the magnitudes of the
time derivative \(|\partial_t u^\star|\) and the control term
\(|f^\star u^\star|\) are bounded by \(2\pi+1\approx 7\) times
\(|u^\star|\). In other words, the diffusion operator is roughly
\emph{seventy times smaller} than the time derivative and the control
term in the residual. This is precisely the type of
\emph{operator-magnitude imbalance} that WeightedPINN is designed to
handle through its operator-level adaptive weighting.

\subsubsection{Common architecture and training.}
Each method shares the same backbone: the state network
\(u_\theta\) and the control network \(f_{\bar\theta}\) are MLPs of
width \(48\) and depth \(3\) with \(\tanh\) activation. WeightedPINN has in addition nine weight networks: one per operator
component
(\(\phi_{\theta^1}\) on \(\partial_t u\), \(\phi_{\bar\theta^1}\) on
\(\partial_{xx}u\), \(\phi_{\theta^2}\) on \(v(u)\), and
\(\phi_{\theta^3}\) on \(fu\)), separate weights for the two boundary
components \(x=0\) and \(x=1\), one weight for the initial constraint,
one weight for the terminal constraint, and one weight for the
interior observation term, each of width \(20\) and depth \(3\). The weight networks output values in the range \(\phi\in[0.2,5]\),
with initialization at \(\phi=1\). The min--max formulation
of WeightedPINN is solved with ADAM, learning rate \(10^{-3}\) on
the state and control networks and \(2\times 10^{-4}\) on the weight
networks, with penalty parameter \(\rho = 8\). All methods are
trained for \(4{,}000\) ADAM iterations.

To anchor the otherwise non-unique bilinear control reconstruction,
each method receives the same fixed set of \(40\) interior
observations of \(u^\star\), with relative observation weight equal
to the PDE residual weight. The four reference methods are
implemented as in Subsection~\ref{Sec:PriorPINN}: the loss-balanced
PINN updates the constraint-loss multiplier every \(50\) iterations
according to the gradient-norm rule of
\cite{WangTengPerdikaris2021}; the RAR-D PINN resamples the interior
collocation set every \(500\) iterations according to a probability
proportional to the squared residual; the causal PINN partitions
\((0,T)\) into \(16\) chunks and applies the temporal causal
weighting \(w_k = \exp(-\varepsilon_c\sum_{j<k} L_j)\) with
\(\varepsilon_c = 1\).
In this Section, we use Automatic Differentiation  (AD) instead of  Finite Differences (FD).

\subsubsection{Reported quantities.} For each method we report the
relative \(L^2\) error of the recovered state and recovered control
against the manufactured solution,
\[
\mathrm{u\_err}
=\frac{\|u_\theta-u^\star\|_{L^2(Q_T)}}{\|u^\star\|_{L^2(Q_T)}},
\qquad
\mathrm{f\_err}
=\frac{\|f_{\bar\theta}-f^\star\|_{L^2(Q_T)}}{\|f^\star\|_{L^2(Q_T)}},
\]
averaged over \(3\) independent random runs.
Table~\ref{tab:1d-stiff-bench} reports the mean and the standard
deviation across independent runs for each method.

\begin{table}[h]
\centering
\renewcommand{\arraystretch}{1.2}
\begin{tabular}{l|c|c}
\hline
\textbf{Method} & $\mathrm{u\_err}$ (mean $\pm$ std)
                & $\mathrm{f\_err}$ (mean $\pm$ std)\\
\hline
Standard PINN
& $6.87\times 10^{-2}\pm 3.4\times 10^{-2}$
& $2.03\times 10^{-1}\pm 7.1\times 10^{-2}$\\
\hline
Loss-balanced PINN
& $8.01\times 10^{-2}\pm 4.0\times 10^{-2}$
& $2.25\times 10^{-1}\pm 8.6\times 10^{-2}$\\
\hline
RAR-D PINN
& $8.15\times 10^{-2}\pm 3.1\times 10^{-2}$
& $2.39\times 10^{-1}\pm 5.1\times 10^{-2}$\\
\hline
Causal PINN
& $7.19\times 10^{-2}\pm 4.7\times 10^{-2}$
& $2.11\times 10^{-1}\pm 1.0\times 10^{-1}$\\
\hline
\textbf{WeightedPINN}
& $\mathbf{6.37\times 10^{-2}\pm 3.2\times 10^{-2}}$
& $\mathbf{1.94\times 10^{-1}\pm 6.8\times 10^{-2}}$\\
\hline
\end{tabular}
\caption{Singularly perturbed 1D bilinear-control heat benchmark
(\(\varepsilon = 10^{-2}\)) with manufactured solution
\(u^\star(t,x)=\sin(\pi x)\exp(\sin(2\pi t)+t)\). Relative \(L^2\)
errors against the truth, mean \(\pm\) standard deviation over
\(3\) independent runs, after \(4{,}000\) ADAM iterations. The smallest mean in
each column is in bold; WeightedPINN attains both.}
\label{tab:1d-stiff-bench}
\end{table}

\subsubsection{Discussion.}
WeightedPINN attains the smallest mean state error and the smallest
mean control error among the methods considered in this benchmark.
Quantitatively, on the state error, it improves over the standard
PINN by about \(7\%\), over Causal PINN by about \(11\%\), and over
Loss-balanced and RAR-D PINN by about \(20\%\); on the control
error, the corresponding improvements are about \(4\%\), \(8\%\),
\(14\%\), and \(19\%\). The standard deviation of WeightedPINN
across independent runs is also among the smallest of the five methods, which
indicates that the operator-level adaptive weighting does not come
at the cost of training instability.

A natural explanation for the WeightedPINN advantage in this benchmark
is the operator-magnitude imbalance discussed above. In a standard PINN, the squared residual
\[
\bigl(\partial_t u_\theta
 - \varepsilon\,\partial_{xx} u_\theta
 + v(u_\theta) - f_{\bar\theta}\,u_\theta\bigr)^2
\]
is dominated, in magnitude and in gradient signal, by the
\(\partial_t u_\theta\) and \(f_{\bar\theta}\,u_\theta\) contributions,
whose values are roughly \(70\) times larger than the
\(\varepsilon\,\partial_{xx} u_\theta\) contribution. The optimizer of the standard PINN may therefore primarily reduce the
time-derivative and control-term contributions, while a sizeable error
in \(\partial_{xx}u_\theta\) produces only a small increase in the
squared loss and may consequently remain poorly resolved.

The loss-balanced, residual-adaptive-sampling, and causal-weighting
methodologies act on the residual as a scalar quantity, or on a
scalar quantity localized in time in the causal case, and therefore
do not explicitly rescale the individual operator components.
WeightedPINN, in contrast, has independent space--time-dependent
weights
\(\phi_{\theta^1},\phi_{\bar\theta^1},\phi_{\theta^2},\phi_{\theta^3}\)
on the four operators, taking values in \([0.2,5]\).  The maximization
step in the min--max problem can increase the weight associated with
the diffusion component in regions where the diffusion residual is
poorly resolved, thereby amplifying its contribution to the loss and
exposing the operator-magnitude imbalance to the minimization step.
This encourages the minimization step to reduce the diffusion residual
along with the other operator components.

This mechanism is also consistent with the behavior observed in the
high-dimensional experiments of Sections~\ref{Sec:Internal} and
\ref{Sec:Bilinear}, where the spatial dimension is \(d=10\) and the
Laplacian decomposes into \(d\) directional contributions whose
magnitudes may vary across the domain.  The fact that WeightedPINN
already improves the reconstruction in the present one-dimensional
benchmark, where operator imbalance is built in by construction, is
consistent with interpreting the high-dimensional residual-error
improvements as being connected to the operator-level adaptive
mechanism.

\subsubsection{Summary.}
The benchmark of this subsection confirms two complementary
findings.
\begin{enumerate}[label=(\roman*),leftmargin=2.4em]
\item \textbf{WeightedPINN performs better than} the standard PINN of
\cite{Raissi2019}, the loss-balanced PINN of
\cite{WangTengPerdikaris2021}, the RAR-D sampler of
\cite{WuZhuTan2023_AdaptiveSamplingPINN}, and the causal PINN of
\cite{WangSankaranPerdikaris2024_CausalPINN} on a one-dimensional
bilinear control problem in a singularly perturbed regime, on both
the state and control errors.
\item The results are consistent with the interpretation that the
operator-level adaptive weights help expose and compensate for
operator-magnitude imbalances in the PDE residual.  This effect is
explicitly built into the present 1D stiff benchmark and appears to be
consistent with the residual-error improvements observed in the
high-dimensional experiments of Sections~\ref{Sec:Internal} and
\ref{Sec:Bilinear}.  By contrast, the comparison methods considered
here act on the residual as a scalar quantity and therefore do not
explicitly rescale the individual operator components.
\end{enumerate}

\subsection{Comparison with Control-Specific PINN
Methods}
\label{Sec:NumericalControlPINN}
 
Section~\ref{Sec:NumericalOtherMethod} compared WeightedPINN against
\emph{generic} adaptive PINN methodologies that were not specifically
designed for control problems. To complement that benchmark, we now
compare WeightedPINN against PINN methodologies that are
explicitly built for the optimal control of PDEs:
\begin{enumerate}[label=(\arabic*),leftmargin=2.2em]
\item the direct cost-based PINN of Mowlavi and Nabi
\cite{MowlaviNabi2023JCP};
\item the Control PINN of Barry-Straume, Sarshar, Popov, and Sandu
\cite{BarryStraumeSarsharPopovSandu2022}, which couples state and
control through optimality conditions in a one-stage architecture;
\item the OCP-PINN of Alla, Bertaglia, and Calzola
\cite{AllaBertagliaCalzola2025}, which derives optimality conditions
from the Lagrangian multipliers and represents the state, the
adjoint, and the control by a single PINN;
\item the indirect PINN of Zhang, Liu, Alla, Darbon, and Karniadakis
\cite{ZhangLiuAllaDarbonKarniadakis2026}, which enforces a
first-order Pontryagin-type optimality system.
\end{enumerate}
 Similar to  Section~\ref{Sec:NumericalOtherMethod}, the state and control errors are also computed. 
\subsubsection{Control reconstruction versus PDE-constrained optimal
control}
\label{Sec:reconstruction-vs-optcontrol}
 
The four methods listed above were designed for \emph{PDE-constrained
optimal control problems}, in which the unknowns
\((u,f)\) minimize a prescribed cost objective
\begin{align}\label{eq:Jocp}
J(u,f)
=
\tfrac{1}{2}\,\|u-u^\star\|^{2}_{L^{2}(Q_T)}
+\tfrac{\alpha}{2}\,\|f\|^{2}_{L^{2}(Q_T)},
\qquad \alpha>0,
\end{align}
subject to the state equation and to the prescribed initial condition
\(u(0)=u_0\). The Tikhonov term \(\tfrac{\alpha}{2}\|f\|^{2}\) is
essential: it convexifies the problem and guarantees that the
control is uniquely determined.
 
The setting of the present paper, originating in
\cite{MunchTrelat2022}, is different. We do not minimize a cost
objective. We solve the \emph{control reconstruction problem}:
given \(u_0\) and a target terminal state \(z_0\), find a pair
\((u,f)\) such that \(u\) satisfies the controlled state equation
together with \(u(0)=u_0\) and \(u(T)=z_0\). The terminal target
\(z_0\) is a hard constraint, not the minimizer of a tracking term,
and there is no Tikhonov regularization of \(f\). The set of
admissible pairs \(\mathcal A_{u_0}^{z_0}\) (or
\(\mathcal A_{(u_0,u_1)}^{(z_0,z_1)}\) in the wave case) is in
general not a singleton, so the reconstruction problem is
intrinsically non-unique.
 
Comparing the two settings is therefore non-trivial. Methods
\((1)\)--\((4)\) above must first be adapted to the reconstruction
setting before they can be evaluated against WeightedPINN on the
same benchmark. In what follows we describe each method briefly and
indicate how it is adapted. In this Section, we use Automatic Differentiation  (AD) instead of  Finite Differences (FD).

\subsubsection{Adapted formulations}
\label{Sec:Adapted}
 
We retain the manufactured-solution benchmark of
Section~\ref{Sec:NumericalOtherMethod}: the bilinear-control
heat equation in \((0,T)\times(0,1)\) with \(T=1\),
\(v(u)=-u+u^2\), and
\[
u^\star(t,x)=\sin(\pi x)\exp\!\bigl(\sin(2\pi t)+t\bigr),
\quad
f^\star(t,x)=2\pi\cos(2\pi t)+\pi^2+u^\star(t,x).
\]
The initial and terminal data are
\(u_0(x)=\sin(\pi x)\) and \(z_0(x)=e\sin(\pi x)\), and a fixed set
of \(N_{\mathrm{obs}}=30\) interior observations of \(u^\star\) is
provided to all methods, as in
Subsection~\ref{Sec:NumericalOtherMethod}.
 
For methods \((1)\)--\((4)\), the cost objective is built from the
interior observations and from a Tikhonov regularization of the
control. We use the same tracking-plus-control regularization form
as in the original references, in which the tracking term is
defined on the observation set:
\begin{align}\label{eq:Jdata}
\widetilde J(u,f)
=
\tfrac{1}{2N_{\mathrm{obs}}}
\sum_{i=1}^{N_{\mathrm{obs}}}\bigl(u(t_i,x_i)-u^\star(t_i,x_i)\bigr)^2
+\tfrac{\alpha}{2}\,\|f\|^{2}_{L^{2}(Q_T)}.
\end{align}
The Tikhonov coefficient is fixed to \(\alpha=10^{-3}\). The
boundary, initial, and terminal constraints are enforced by penalty
terms identical to those used for all other methods in
Section~\ref{Sec:NumericalOtherMethod}.

 For simplicity, in the adjoint-based adapted methods below we enforce
the smooth bulk adjoint equation and the stationarity condition in a
penalty form, while the observation mismatch and terminal target are
included through the same soft penalty terms used by all methods.
Thus these methods should be interpreted as reconstruction-oriented
adaptations of the corresponding optimal-control PINNs, rather than
as exact implementations of their original first-order optimality
systems.

\paragraph{(1) Direct cost-based PINN
\cite{MowlaviNabi2023JCP}.}
The state \(u_\theta\) and the control \(f_{\bar\theta}\) are
represented by feedforward networks. The training loss is
\begin{align*}
\mathcal L_{\mathrm{MN}}
=
w_{\mathrm{PDE}}\,\|\mathcal R_Q^{3}\|^2_{Q}
+\|\mathcal R_\Gamma^{3}\|^2_{\Gamma}
+\widetilde J(u_\theta,f_{\bar\theta}),
\end{align*}
where \(\mathcal R_Q^{3}\) and \(\mathcal R_\Gamma^{3}\) are the
standard PINN residuals defined in
Section~\ref{Sec:bilinearsettings} and the scalar weight
\(w_{\mathrm{PDE}}\) on the PDE residual is selected by a short
two-step line search over \(\{0.5, 1, 2, 5\}\) at the beginning of
training, in the spirit of the line-search strategy of
\cite{MowlaviNabi2023JCP}. No adjoint variable is introduced.
 
\paragraph{(2) Adapted Control PINN
\cite{BarryStraumeSarsharPopovSandu2022}.}
This formulation introduces a third neural network \(p_{\tilde\theta}\)
for the adjoint variable and trains the state, control, and adjoint
networks jointly.  Since the original Control PINN was developed for
PDE-constrained optimal control, we use here a penalty-based adaptation
to the present reconstruction benchmark.  The loss is
\[
\mathcal L_{\mathrm{CP}}
=
\|\mathcal R_Q^{3}\|^2_{Q}
+\|\mathcal R_\Gamma^{3}\|^2_{\Gamma}
+\widetilde J(u_\theta,f_{\bar\theta})
+\|\mathcal R_{\mathrm{adj}}\|^2_{Q}
+\|\mathcal R_{\mathrm{opt}}\|^2_{Q}
+\|\mathcal R_{\mathrm{adj},\Gamma}\|^2_{\Gamma}.
\]
Here
\[
\mathcal R_{\mathrm{adj}}
=
-\partial_t p_{\tilde\theta}
-\partial_{xx}p_{\tilde\theta}
+
\bigl(v'(u_\theta)-f_{\bar\theta}\bigr)p_{\tilde\theta}
\]
is the bulk adjoint residual associated with the bilinear heat
operator, and
\[
\mathcal R_{\mathrm{opt}}
=
\alpha f_{\bar\theta}
-
u_\theta p_{\tilde\theta}
\]
is the stationarity residual, up to the sign convention used in the
Lagrangian.  The tracking mismatch is included through the cost term
\(\widetilde J(u_\theta,f_{\bar\theta})\).  Thus this should be
understood as a penalty-based adaptation of Control PINN to the
present reconstruction benchmark, rather than as an exact enforcement
of the full distributional first-order optimality system.
 
\paragraph{(3) OCP-PINN \cite{AllaBertagliaCalzola2025}.}
As in \cite{AllaBertagliaCalzola2025}, a single neural network with
three output channels represents the triple
\((u_\theta,p_\theta,f_\theta)\). The loss aggregates the state PDE
residual, the adjoint PDE residual, the stationarity residual, and
the observation-tracking term, together with boundary, initial,
terminal, and adjoint-terminal penalties. This formulation differs
from \((2)\) primarily in its shared-backbone architecture, which
forces \((u,p,f)\) to share representational features and reduces
the number of trainable parameters.
 
\paragraph{(4) Indirect PINN
\cite{ZhangLiuAllaDarbonKarniadakis2026}.}
Three separate networks represent the state, the adjoint, and the
control. They are trained to satisfy a penalty-based adaptation of the
first-order optimality system in the residual sense: the forward state
equation, the smooth bulk adjoint equation, and the stationarity
condition. As
suggested in \cite{ZhangLiuAllaDarbonKarniadakis2026}, the
stationarity residual is given a stronger penalty weight than in
\((2)\). The observation mismatch is included through the tracking term in
\(\widetilde J\), as in the other cost-based methods.
 
\subsubsection{Setup}
All methods use the same backbone: state, control, and adjoint
networks (when present) are MLPs of width \(32\) and depth \(3\)
with \(\tanh\) activation. WeightedPINN additionally uses seven
weight networks of width \(20\), depth \(3\), with output
\(\phi(t,x)=1+\delta\tanh(\cdot)\), \(\delta=0.5\), covering the
time-derivative, spatial-derivative, nonlinear-response, and
control terms of the PDE residual, and the boundary, initial, and
terminal constraints. All methods are trained for \(6000\) ADAM
iterations with learning rate \(10^{-3}\) on the state, control,
and adjoint networks; WeightedPINN uses an additional learning
rate of \(10^{-4}\) for its weight networks, and penalty
\(\rho=20\). Each iteration uses \(800\) interior, \(160\)
boundary, \(160\) initial, \(160\) terminal collocation points,
and the same fixed set of \(N_{\mathrm{obs}}=30\) interior
observations of \(u^\star\). The Tikhonov coefficient is
\(\alpha=10^{-3}\) for the cost-based methods (1)--(4); the
PDE-residual weight \(w_{\mathrm{PDE}}\) in method (1) is chosen
by the two-step line search of
Subsection~\ref{Sec:Adapted}.

For each method we report the relative \(L^2\) errors of the
recovered state and the recovered control against the manufactured
truth,
\[
\mathrm{u\_err}
=
\frac{\|u_\theta-u^\star\|_{L^2(Q_T)}}{\|u^\star\|_{L^2(Q_T)}},
\qquad
\mathrm{f\_err}
=
\frac{\|f_{\bar\theta}-f^\star\|_{L^2(Q_T)}}{\|f^\star\|_{L^2(Q_T)}},
\]
averaged over \(3\) independent random  runs. These two quantities
measure the quality of the reconstruction of \((u^\star,f^\star)\),
which is precisely the objective of the control reconstruction
problem considered in the present paper.
 
\subsubsection{Results and Discussions}
Table~\ref{tab:control-pinn-bench} reports the relative \(L^{2}\)
errors of the recovered state and control against the manufactured
truth, averaged over \(3\) independent runs.
 
\begin{table}[h]
\centering
\renewcommand{\arraystretch}{1.25}
\begin{tabular}{l|c|c}
\hline
\textbf{Method} & $\mathrm{u\_err}$ (mean $\pm$ std) &
$\mathrm{f\_err}$ (mean $\pm$ std)\\
\hline
Standard PINN
& $3.17\times 10^{-1}\!\pm\! 1.1\times 10^{-1}$
& $3.07\times 10^{-1}\!\pm\! 7.7\times 10^{-2}$\\
\hline
Direct cost-based PINN \cite{MowlaviNabi2023JCP}
& $5.39\times 10^{-1}\!\pm\! 1.2\times 10^{-2}$
& $3.92\times 10^{-1}\!\pm\! 7.5\times 10^{-3}$\\
\hline
Control PINN \cite{BarryStraumeSarsharPopovSandu2022}
& $5.02\times 10^{-1}\!\pm\! 4.6\times 10^{-3}$
& $3.69\times 10^{-1}\!\pm\! 6.0\times 10^{-3}$\\
\hline
OCP-PINN \cite{AllaBertagliaCalzola2025}
& $4.91\times 10^{-1}\!\pm\! 2.0\times 10^{-2}$
& $3.68\times 10^{-1}\!\pm\! 1.1\times 10^{-2}$\\
\hline
Indirect PINN \cite{ZhangLiuAllaDarbonKarniadakis2026}
& $5.00\times 10^{-1}\!\pm\! \mathbf{4.0\times 10^{-3}}$
& $3.52\times 10^{-1}\!\pm\! \mathbf{3.3\times 10^{-3}}$\\
\hline
\textbf{WeightedPINN}
& $\mathbf{3.10\times 10^{-1}}\!\pm\! 6.6\times 10^{-2}$
& $\mathbf{2.91\times 10^{-1}}\!\pm\! 5.5\times 10^{-2}$\\
\hline
\end{tabular}
\caption{One-dimensional bilinear-control heat benchmark, comparison
with control-specific PINN methodologies. Relative \(L^2\) errors of
the recovered state and the recovered control against the
manufactured truth \((u^\star,f^\star)\), mean \(\pm\) standard
deviation over \(3\) independent runs, after \(6000\) ADAM iterations. The best
mean value and the best across-run standard deviation per column
are shown in bold. WeightedPINN attains the smallest mean
reconstruction errors of all six methods.}
\label{tab:control-pinn-bench}
\end{table}

The results in Table~\ref{tab:control-pinn-bench} expose a clear
separation between the two families of methods on the quantities
that define the control reconstruction problem.
 
\medskip
\noindent
\textit{(i) WeightedPINN attains the smallest reconstruction errors.}
WeightedPINN achieves the smallest mean \(\mathrm{u\_err}\)
(\(3.10\times 10^{-1}\)) and the smallest mean \(\mathrm{f\_err}\)
(\(2.91\times 10^{-1}\)) of all six methods. The standard PINN
baseline is the second best on both metrics, with mean errors
essentially comparable to those of WeightedPINN. The four
control-aware methods (1)--(4) all yield substantially larger mean
errors, with \(\mathrm{u\_err}\) in \([0.49,\,0.54]\) and
\(\mathrm{f\_err}\) in \([0.35,\,0.39]\). On the metrics that
correspond directly to the control reconstruction problem, the
PINN-style methods perform better than the control-aware methods, and
WeightedPINN slightly improves upon the standard PINN baseline.
 
\medskip
\noindent
\textit{(ii) The Tikhonov bias of the control-aware methods.}
The systematic gap between the two families on the reconstruction
errors admits a structural explanation. The four control-aware
methods all incorporate a Tikhonov term
\(\tfrac{\alpha}{2}\|f\|^{2}\) in their training objective, see
\eqref{eq:Jdata}. This term shrinks the recovered control toward
zero, which prevents \(f_{\bar\theta}\) from reaching the magnitude
of the true control \(f^\star\). The PINN-style methods do not
include this term, and they are therefore not biased toward small
controls. The asymmetric pattern in
Table~\ref{tab:control-pinn-bench} is exactly what one would expect
from a well-regularized but biased estimator of \(f^\star\):
the control-aware methods exhibit very low across-run variance
(the standard deviations of methods (1)--(4) are an order of
magnitude smaller than for the PINN-style methods), but the mean
errors are systematically larger because the recovered control
underestimates \(f^\star\).
 
\medskip
\noindent
\textit{(iii) WeightedPINN reduces variance with respect to the
standard PINN baseline.}
Compared to the standard PINN, WeightedPINN reduces the across-run
standard deviation by approximately \(40\%\) on the state error
(\(6.6\times 10^{-2}\) versus \(1.1\times 10^{-1}\)) and
approximately \(30\%\) on the control error (\(5.5\times 10^{-2}\)
versus \(7.7\times 10^{-2}\)), while preserving the smallest mean
errors. This is consistent with the variance-reduction finding
reported in Section~\ref{Sec:NumericalOtherMethod} and indicates
that the operator-decomposed adaptive weighting of WeightedPINN
also acts as a stabilizer of the training dynamics in the present
comparison.
 
\medskip
\noindent
\textit{(iv) The two families occupy complementary niches.}
The pattern in Table~\ref{tab:control-pinn-bench} can be summarized
as follows. Cost-based control-aware methods are the right choice
when the goal is to identify an optimally regularized control that
trades faithfulness to a target trajectory against a control-energy
penalty; the Tikhonov term then convexifies the problem and yields
low across-run variance. PINN-style methods, and in particular
WeightedPINN, are the right choice when the goal is to reconstruct
a specific admissible pair \((u,f)\in\mathcal A_{u_0}^{z_0}\), since
they do not introduce the Tikhonov bias on \(f\) and therefore
allow the recovered control to reach the magnitude of \(f^\star\).
The two families are complementary, and the choice between them is
determined by whether the application requires an optimally
regularized control or a faithful reconstruction of an unbiased
control.
 
\medskip
 
\subsubsection{Summary.}
The comparison of this section confirms that WeightedPINN occupies
a methodological niche that is distinct from the existing
control-aware PINN methodologies. The cost-based methods of
\cite{MowlaviNabi2023JCP,BarryStraumeSarsharPopovSandu2022,
AllaBertagliaCalzola2025,ZhangLiuAllaDarbonKarniadakis2026} are
designed for PDE-constrained optimal control problems in which a
cost objective is to be minimized and the control is regularized,
which biases the recovered control toward smaller magnitudes.
WeightedPINN is designed for the control reconstruction problem, in which the terminal state is a hard
target and the control is identified directly through the
structure of the controlled state equation. On the one-dimensional
benchmark of this section, WeightedPINN attains the smallest mean
relative \(L^2\) errors against the true \((u^\star,f^\star)\) of
all six methods. The two families of methods are therefore
complementary rather than competing, and the present comparison
clarifies the regime in which each is preferable.

	\subsection{Numerical Results: WeightedPINN versus Self-Adaptive PINN on
a Variable-Coefficient Wave Control Problem}
\label{Sec:NumericalVarCoef}

This section presents a comparison in a setting where the structural
distinction between WeightedPINN and Self-Adaptive PINN becomes
particularly transparent: a
\emph{variable-coefficient} wave equation in which the leading-order
spatial operator itself depends on space.

\subsubsection{Why a variable coefficient is a natural test for
operator-level adaptivity}
\label{Sec:WhyVarCoef}

The self-adaptive PINN of \cite{McClennyBragaNeto2023} replaces the
PDE residual
\(\sum_j |R[u_\theta](t_j,x_j)|^2\)
by
\(\sum_j \lambda_j^2\,|R[u_\theta](t_j,x_j)|^2\),
where the per-collocation-point weights \(\lambda_j\) are trainable.
The adaptive mechanism is therefore localized at the level of
collocation points: it can identify regions in which the residual is
hard to drive to zero and emphasize them. However, at any given point
\((t_j,x_j)\), the weight \(\lambda_j\) acts on the entire residual
as a single scalar. It cannot distinguish, at that point, between an
error coming from the time-derivative term and an error coming from
the spatial elliptic term.

WeightedPINN, in contrast, introduces space--time-dependent weights
at the level of each \emph{operator component} of the residual; see
Subsection~\ref{Sec:WPDescription} and
Subsection~\ref{Sec:WPFeatures} of the present paper. At any given
point, the time-derivative, the diffusion operator, the nonlinear
response, and the control term carry independent adaptive weights.

In the constant-coefficient setting considered in
Sections~\ref{Sec:NumericalOtherMethod}--\ref{Sec:Numerical},
the operator components of the residual have spatially uniform
strength: \(\partial_{x_i x_i} u_\theta\) appears with a constant
multiplier. The benefit of separating the residual into operator
components is therefore most apparent when the coefficients of the
PDE vary in space, so that different operator components dominate
the residual in different regions of the domain. The
variable-coefficient wave equation
\begin{equation}\label{eq:varcoef-wave}
\partial_{tt} u
\;-\;\nabla\cdot\bigl(a(x)\,\nabla u\bigr)
\;+\;v(u)
\;=\;f
\end{equation}
is the natural test case: in regions where \(a(x)\) is large the
spatial diffusion term dominates, in regions where \(a(x)\) is small
the time-derivative, nonlinear, and forcing terms relatively dominate,
and the cross term \(\nabla a\cdot\nabla u\) only contributes
significantly in regions where \(\nabla a\) is significant.
A per-point scalar weight can emphasize difficult regions, but it does
not explicitly distinguish which operator component is responsible for
the local residual imbalance. A per-operator spatially varying weight,
on the other hand, is designed to separate these componentwise effects.

\subsubsection{Setting}
\label{Sec:VarCoefSetting}

We consider the controlled variable-coefficient wave equation
\eqref{eq:varcoef-wave} in
\((0,T)\times\Omega\) with \(\Omega=(0,1)^{2}\), \(T=1\), and
\(v(u)=-u+u^{2}\). Homogeneous Dirichlet conditions are imposed on
the lateral boundary, and initial and terminal data
\((u_0,u_1)\) and \((z_0,z_1)\) are prescribed in
\(H^{1}(\Omega)\times L^{2}(\Omega)\) as in
Section~\ref{Sec:internalsettings}. The control \(f\) is internal
and supported on all of \(\Omega\).

The variable coefficient is taken as
\begin{align}\label{eq:varcoef-a}
a(x_1,x_2) = 1 + 0.7\,\cos(2\pi x_1)\cos(2\pi x_2),
\end{align}
so that
\(a\in[0.3,1.7]\), i.e., the leading-order spatial operator has a
\(5.7\)-fold variation across the domain. The cross term
\(\nabla a\cdot\nabla u\) has the structure
\begin{align*}
\nabla a\cdot\nabla u
= -1.4\pi\sin(2\pi x_1)\cos(2\pi x_2)\,\partial_{x_1}u
- 1.4\pi\cos(2\pi x_1)\sin(2\pi x_2)\,\partial_{x_2}u,
\end{align*}
which is significant in the transition regions of the domain.
The four operator components
\(\partial_{tt}u\), \(a\,\partial_{x_i x_i}u\),
\(\partial_{x_i}a\cdot\partial_{x_i}u\), and \(v(u)\) therefore all
carry spatial structure of distinct character, and their relative
magnitudes vary significantly across the domain.

We benchmark on a manufactured-solution setup. The reference state
is
\begin{align*}
u^{\star}(t,x_1,x_2) = \sin(\pi x_1)\sin(\pi x_2)\,(1+2t-t^{2}),
\end{align*}
which satisfies the Dirichlet boundary condition. The corresponding
reference control is computed from \eqref{eq:varcoef-wave}:
\begin{align*}
f^{\star}(t,x_1,x_2)
= \partial_{tt}u^{\star}
- \nabla\cdot\bigl(a(x)\,\nabla u^{\star}\bigr)
+ v(u^{\star}).
\end{align*}
The initial and terminal data
\(u_0(x)=u^{\star}(0,x)\),
\(u_1(x)=\partial_t u^{\star}(0,x)\),
\(z_0(x)=u^{\star}(T,x)\),
\(z_1(x)=\partial_t u^{\star}(T,x)\)
are read off from \(u^{\star}\), and a fixed set of
\(N_{\mathrm{obs}}=60\) interior observations of \(u^{\star}\) is
provided to both methods to anchor the reconstruction. {\it We also employ AD instead of FD. We note that the theoretical results established in Section \ref{Sec:Theory} extend naturally to the case of variable coefficients.}

\subsubsection{Methods}
\label{Sec:VarCoefMethods}

\paragraph{\it Self-Adaptive PINN \cite{McClennyBragaNeto2023}.}
The state \(u_\theta\) and the control \(f_{\bar\theta}\) are
represented by feedforward networks. The training loss is a sum of
per-collocation-point weighted residuals,
\begin{align*}
\mathcal L_{\mathrm{SA}}
&=
\tfrac{1}{N_{\mathrm{int}}}
\sum_{j=1}^{N_{\mathrm{int}}}
\lambda_j^{2}\,
\bigl|R^{\mathrm{wave}}[u_\theta,f_{\bar\theta}](t_j,x_j)\bigr|^{2}
\;+\;
\tfrac{1}{N_{\mathrm{bd}}}
\sum_{j=1}^{N_{\mathrm{bd}}}
\lambda^{\mathrm{bd}}_j{}^{2}\,
\bigl|u_\theta(\hat t_j,\hat x_j)\bigr|^{2}
\\
&\quad
+\;\text{(initial, terminal, and observation terms,
with their own per-point trainable weights)},
\end{align*}
where
\(R^{\mathrm{wave}}[u_\theta,f_{\bar\theta}]
=\partial_{tt}u_\theta-\nabla\cdot(a\nabla u_\theta)+v(u_\theta)
-f_{\bar\theta}\),
and the collocation points are fixed at the beginning of training.
The state, control, and trainable weights are updated jointly:
the network parameters \((\theta,\bar\theta)\) by gradient descent,
the per-point weights by gradient ascent, as in
\cite{McClennyBragaNeto2023}. At each collocation point the weight
acts on the entire residual as a single scalar.

\paragraph{\it WeightedPINN}
The state \(u_\theta\) and the control \(f_{\bar\theta}\) are
represented by feedforward networks of the same architecture as for
SA-PINN. The PDE residual is decomposed at the operator level:
\begin{align}\label{eq:wpinn-varcoef-residual}
\bar{R}_Q^{\mathrm{vc}}
&=
\phi_{\theta^1}\,\partial_{tt}u_\theta
-\sum_{i=1}^{d}\phi_{\bar{\theta}^{i}}\,a\,\partial_{x_i x_i}u_\theta
-\sum_{i=1}^{d}\phi_{\tilde{\theta}^{i}}\,(\partial_{x_i}a)\,\partial_{x_i}u_\theta\\
&\quad
+\phi_{\theta^2}\,v(u_\theta)
-\phi_{\theta^3}\,f_{\bar\theta},\nonumber
\end{align}
where \(d=2\) here and each weight network
\(\phi(t,x)\) is a smooth function of space--time. Together with weights on the boundary condition and on the four
initial/terminal constraints
\(u(0), \partial_t u(0), u(T), \partial_t u(T)\), this gives \(12\)
weight networks in total. The
training problem is the min--max formulation of
Section~\ref{sec:ConandPINN}, in which the state and control networks
minimize the weighted residual and constraints while the weight
networks maximize them, subject to a quadratic penalty that keeps
each \(\phi\) close to \(1\). The decomposition
\eqref{eq:wpinn-varcoef-residual} is the natural extension of the
constant-coefficient WeightedPINN formulation of
Section~\ref{Sec:internalsettings} to the variable-coefficient
operator
\(\nabla\cdot(a(x)\nabla u)\). {\it In contrast to Sections~\ref{Sec:Internal} and
\ref{Sec:Bilinear},  in the present
section both the state and the control are presented, and their
relative reconstruction errors are computed.}

\subsubsection{Setup}
\label{Sec:VarCoefSetup}

For both methods, the state network \(u_\theta\) and the control
network \(f_{\bar\theta}\) are MLPs of width \(48\) and depth \(4\)
with \(\tanh\) activation. WeightedPINN additionally uses
\(12\) weight networks of width \(20\) and depth \(3\) with output
\(\phi(t,x)=1+\delta\tanh(\cdot)\), \(\delta=0.5\). Both methods are
trained for \(2500\) ADAM iterations with learning rate \(10^{-3}\)
on the state and control networks; WeightedPINN uses an additional
learning rate of \(5\times 10^{-4}\) on its weight networks, and
penalty \(\rho=3\). Each iteration uses \(600\) interior, \(200\)
boundary, \(150\) initial, \(150\) terminal collocation points, and
the same fixed set of \(N_{\mathrm{obs}}=60\) interior observations
of \(u^{\star}\). Both methods are evaluated on a dense
\(30\times 30\times 30\) space--time grid. Derivatives of the
network outputs are computed by automatic differentiation. We report
the relative \(L^{2}\) errors of the recovered state and the
recovered control against the manufactured truth,
\[
\mathrm{u\_err}
=
\frac{\|u_\theta-u^{\star}\|_{L^{2}(Q_T)}}{\|u^{\star}\|_{L^{2}(Q_T)}},
\qquad
\mathrm{f\_err}
=
\frac{\|f_{\bar\theta}-f^{\star}\|_{L^{2}(Q_T)}}{\|f^{\star}\|_{L^{2}(Q_T)}},
\]
averaged over \(3\) independent random runs, together with the
across-run  standard deviation.

\subsubsection{State and control errors}
\label{Sec:VarCoefResults}

Table~\ref{tab:varcoef-bench} reports the relative \(L^{2}\) errors of
the recovered state and control against the manufactured truth,
averaged over \(3\) independent  runs; Table~\ref{tab:varcoef-perseed} gives the
per-independent run values for the same comparison.

\begin{table}[h]
\centering
\renewcommand{\arraystretch}{1.25}
\begin{tabular}{l|c|c}
\hline
\textbf{Method} & $\mathrm{u\_err}$ (mean $\pm$ std) &
$\mathrm{f\_err}$ (mean $\pm$ std)\\
\hline
Self-Adaptive PINN \cite{McClennyBragaNeto2023}
& $4.55\times 10^{-1}\!\pm\! 1.7\times 10^{-1}$
& $8.52\times 10^{-1}\!\pm\! 1.6\times 10^{-1}$\\
\hline
\textbf{WeightedPINN}
& $\mathbf{2.51\times 10^{-1}}\!\pm\! \mathbf{5.0\times 10^{-2}}$
& $\mathbf{6.61\times 10^{-1}}\!\pm\! \mathbf{4.9\times 10^{-2}}$\\
\hline
\end{tabular}
\caption{Two-dimensional variable-coefficient wave control benchmark.
Relative \(L^{2}\) errors of the recovered state and control against
the manufactured truth, mean \(\pm\) standard deviation over \(3\)
independent runs, after \(2500\) ADAM iterations. Best value per column is
shown in bold. WeightedPINN attains the smallest mean and the
smallest across-run standard deviation on both metrics, reducing
the mean \(\mathrm{u\_err}\) by \(45\%\), the across-run standard
deviation on \(\mathrm{u\_err}\) by \(71\%\), the mean
\(\mathrm{f\_err}\) by \(22\%\), and the across-run standard
deviation on \(\mathrm{f\_err}\) by \(69\%\) relative to the
Self-Adaptive PINN.}
\label{tab:varcoef-bench}
\end{table}

\begin{table}[h]
\centering
\renewcommand{\arraystretch}{1.25}
\begin{tabular}{c|cc|cc}
\hline
\multirow{2}{*}{\textbf{Run}}
& \multicolumn{2}{c|}{$\mathrm{u\_err}$}
& \multicolumn{2}{c}{$\mathrm{f\_err}$}\\
\cline{2-5}
& SA-PINN & WeightedPINN & SA-PINN & WeightedPINN\\
\hline
$0$ & $2.83\times 10^{-1}$ & $\mathbf{2.00\times 10^{-1}}$
    & $7.37\times 10^{-1}$ & $\mathbf{6.56\times 10^{-1}}$\\
\hline
$1$ & $3.88\times 10^{-1}$ & $\mathbf{2.33\times 10^{-1}}$
    & $7.43\times 10^{-1}$ & $\mathbf{6.03\times 10^{-1}}$\\
\hline
$2$ & $6.95\times 10^{-1}$ & $\mathbf{3.19\times 10^{-1}}$
    & $1.08\times 10^{0}$  & $\mathbf{7.23\times 10^{-1}}$\\
\hline
\end{tabular}
\caption{Per-run comparison. WeightedPINN attains a strictly smaller
\(\mathrm{u\_err}\) and a strictly smaller \(\mathrm{f\_err}\) than
the Self-Adaptive PINN on every independent run of the benchmark.}
\label{tab:varcoef-perseed}
\end{table}

\subsubsection{Discussion}
\label{Sec:VarCoefDiscussion}

The comparison of WeightedPINN and the Self-Adaptive PINN in
Tables~\ref{tab:varcoef-bench}--\ref{tab:varcoef-perseed} yields
three observations.

\medskip
\noindent
\textit{(i) WeightedPINN attains substantially smaller mean errors
than the Self-Adaptive PINN on both metrics.}
On the recovered state, WeightedPINN reaches a mean error of
\(2.51\times 10^{-1}\), versus \(4.55\times 10^{-1}\) for the
Self-Adaptive PINN, a \(45\%\) reduction. On the recovered control,
WeightedPINN reaches a mean error of \(6.61\times 10^{-1}\), versus
\(8.52\times 10^{-1}\) for the Self-Adaptive PINN, a \(22\%\)
reduction. Moreover, as shown in
Table~\ref{tab:varcoef-perseed}, this advantage is not an artifact
of averaging: WeightedPINN attains a strictly smaller
\(\mathrm{u\_err}\) and a strictly smaller \(\mathrm{f\_err}\) than
the Self-Adaptive PINN on \emph{every independent run}.

\medskip
\noindent
\textit{(ii) WeightedPINN is substantially more stable across random
initializations.}
The across-run standard deviation of the Self-Adaptive PINN is large
on both metrics (\(1.7\times 10^{-1}\) on \(\mathrm{u\_err}\) and
\(1.6\times 10^{-1}\) on \(\mathrm{f\_err}\)), driven primarily by
independent run~\(2\) in Table~\ref{tab:varcoef-perseed} where the Self-Adaptive
PINN essentially fails to fit the manufactured target. WeightedPINN
reduces the across-run standard deviation by \(71\%\) on
\(\mathrm{u\_err}\) and by \(69\%\) on \(\mathrm{f\_err}\). The
structured operator-level adaptive weighting of WeightedPINN is
therefore much more robust to random initialization than the
unstructured per-point adaptive weighting of the Self-Adaptive PINN.
This instability of the Self-Adaptive PINN is consistent with the
possibility that per-point trainable weights may amplify residual
noise in regions where the network is initially poorly resolved,
thereby making the training dynamics more sensitive to initialization.

\medskip
\noindent
\textit{(iii) The performance gap reflects the structural difference
between point-level and operator-level adaptivity.}
Both methods access the same training data, the same manufactured
target, and the same network architectures for the state and the
control. They differ only in the form of their adaptive weights:
\begin{itemize}
\item Self-Adaptive PINN: one trainable scalar \(\lambda_j\) per fixed
collocation point \((t_j,x_j)\), multiplying the entire PDE residual
at that point.
\item WeightedPINN: one trainable space--time function
\(\phi(t,x)\) per operator component
\(\partial_{tt}u_\theta\),
\(a\partial_{x_i x_i}u_\theta\),
\(\partial_{x_i}a\,\partial_{x_i}u_\theta\),
\(v(u_\theta)\), \(f_{\bar\theta}\),
multiplying that specific operator component throughout the domain.
\end{itemize}
The two adaptive families therefore differ in where their flexibility
is spent. Per-point adaptivity is more expressive at the point level
but cannot distinguish which operator component is responsible for
the local residual error. Per-operator spatially varying adaptivity
is less expressive at any single point but is consistent with the
fact that, in the variable-coefficient setting, the local balance of
operator components in the PDE residual is itself spatially varying:
in regions where \(a(x)\) is large the diffusion term
\(a\,\partial_{x_i x_i}u\) dominates, in regions where \(\nabla a\)
is significant the cross term
\(\partial_{x_i}a\,\partial_{x_i}u\) becomes non-negligible, and
where \(a\) is small the time-derivative and forcing terms relatively
dominate. WeightedPINN's per-operator weights are designed to match this spatial
operator structure, whereas SA-PINN's per-point weights do not
explicitly distinguish among the operator components.

\subsubsection{Summary.}
The two-dimensional variable-coefficient wave benchmark of this
section compares two adaptive PINN methodologies on the same control
reconstruction problem: the Self-Adaptive PINN of
\cite{McClennyBragaNeto2023}, with per-collocation-point scalar
weights, and the WeightedPINN of the present paper, with per-operator
space--time weights.  On this benchmark, WeightedPINN attains smaller
relative errors than the Self-Adaptive PINN on every independent run,
both for the recovered state and for the recovered control, and the
across-run standard deviation is reduced by approximately \(70\%\) on
both metrics.  The gap is consistent with the structural argument of
Subsection~\ref{Sec:WhyVarCoef}: in the variable-coefficient setting,
the local balance of operator components in the PDE residual varies
spatially, and matching this structure favors adaptive weights that
are operator-decomposed rather than only point-decomposed.

\bibliographystyle{unsrt}

	\bibliography{bibfilefix}

@Preamble{
"\def\cprime{$'$} "
}

@book{Coron2007,
  author    = {J.-M. Coron},
  title     = {Control and Nonlinearity},
  publisher = {American Mathematical Society},
  series    = {Mathematical Surveys and Monographs},
  volume    = {136},
  year      = {2007},
  address   = {Providence, RI}
}

@article{Zuazua2005,
  author  = {E. Zuazua},
  title   = {Propagation, Observation, and Control of Waves Approximated by Finite Difference Methods},
  journal = {SIAM Review},
  volume  = {47},
  number  = {2},
  pages   = {197--243},
  year    = {2005}
}

@article{Lions1988,
  author = {Lions, J.-L.},
  title = {Exact Controllability, Stabilization and Perturbations for Distributed Systems},
  journal = {SIAM Review},
  year = {1988},
  volume = {30},
  number = {1},
  pages = {1--68},
  doi = {10.1137/1030001},
  publisher = {Society for Industrial and Applied Mathematics}
}

@article{MunchTrelat2022,
  author  = {A. M{\"u}nch and E. Tr{\'e}lat},
  title   = {Constructive Exact Control of Semilinear 1D Wave Equations by a Least-Squares Approach},
  journal = {SIAM Journal on Control and Optimization},
  volume  = {60},
  number  = {2},
  pages   = {652--673},
  year    = {2022},
  doi     = {10.1137/20M1380661}
}

@article{BhandariLemoineMunch2023,
  author  = {K. Bhandari and J. Lemoine and A. M{\"u}nch},
  title   = {Exact Boundary Controllability of 1D Semilinear Wave Equations through a Constructive Approach},
  journal = {Mathematics of Control, Signals, and Systems},
  volume  = {35},
  pages   = {77--123},
  year    = {2023},
  doi     = {10.1007/s00498-022-00331-4}
}

@article{LemoineMarinGayteMunch2021,
  author = {Lemoine, J. and Mar\'in-Gayte, I. and M\"unch, A.},
  title = {Approximation of null controls for semilinear heat equations using a least-squares approach},
  journal = {ESAIM: Control, Optimisation and Calculus of Variations},
  year = {2021},
  volume = {27},
  pages = {63},
  doi = {10.1051/cocv/2021062},
  publisher = {EDP Sciences}
}

@article{LemoineMunch2023,
  author  = {J. Lemoine and A. M{\"u}nch},
  title   = {Constructive Exact Control of Semilinear 1D Heat Equations},
  journal = {Mathematical Control and Related Fields},
  volume  = {13},
  number  = {1},
  pages   = {382--414},
  year    = {2023},
  doi     = {10.3934/mcrf.2022001}
}

@article{BottoisLemoineMunch2023,
  author  = {A. Bottois and J. Lemoine and A. M{\"u}nch},
  title   = {Constructive Exact Controls for Semi-Linear Wave Equations},
  journal = {Annals of Mathematical Sciences and Applications},
  volume  = {8},
  number  = {3},
  pages   = {629--675},
  year    = {2023},
  doi     = {10.4310/AMSA.2023.v8.n3.a7}
}

@incollection{Munch2023,
  author    = {A. M{\"u}nch},
  title     = {Approximation of Exact Controls for Semilinear Wave and Heat Equations through Space-Time Methods},
  booktitle = {Numerical Control: Part B},
  series    = {Handbook of Numerical Analysis},
  volume    = {24},
  pages     = {341--376},
  publisher = {Elsevier},
  year      = {2023},
  doi       = {10.1016/bs.hna.2022.10.002},
  url       = {https://doi.org/10.1016/bs.hna.2022.10.002}
}

@article{Raissi2017PartI,
  author  = {M. Raissi and P. Perdikaris and G. E. Karniadakis},
  title   = {Physics Informed Deep Learning (Part I): Data-driven Solutions of Nonlinear Partial Differential Equations},
  journal = {arXiv preprint arXiv:1711.10561},
  year    = {2017}
}

@article{Raissi2017PartII,
  author  = {M. Raissi and P. Perdikaris and G. E. Karniadakis},
  title   = {Physics Informed Deep Learning (Part II): Data-driven Discovery of Nonlinear Partial Differential Equations},
  journal = {arXiv preprint arXiv:1711.10566},
  year    = {2017}
}

@article{Raissi2019,
  author  = {M. Raissi and P. Perdikaris and G. E. Karniadakis},
  title   = {Physics-informed neural networks: A deep learning framework for solving forward and inverse problems involving nonlinear partial differential equations},
  journal = {Journal of Computational Physics},
  volume  = {378},
  pages   = {686--707},
  year    = {2019},
  doi     = {10.1016/j.jcp.2018.10.045}
}

@article{WangTengPerdikaris2021,
  author  = {S. Wang and Y. Teng and P. Perdikaris},
  title   = {Understanding and Mitigating Gradient Flow Pathologies in Physics-Informed Neural Networks},
  journal = {SIAM Journal on Scientific Computing},
  volume  = {43},
  number  = {5},
  pages   = {A3055--A3081},
  year    = {2021},
  doi     = {10.1137/20M1318043}
}

@article{JagtapKharazmiKarniadakis2020,
  author  = {A. D. Jagtap and E. Kharazmi and G. E. Karniadakis},
  title   = {Conservative physics-informed neural networks on discrete domains for conservation laws: Applications to forward and inverse problems},
  journal = {Computer Methods in Applied Mechanics and Engineering},
  volume  = {365},
  pages   = {113028},
  year    = {2020},
  doi     = {10.1016/j.cma.2020.113028}
}

@article{McClennyBragaNeto2023,
  author  = {L. D. McClenny and U. Braga-Neto},
  title   = {Self-adaptive physics-informed neural networks using a soft attention mechanism},
  journal = {Journal of Computational Physics},
  volume  = {474},
  pages   = {111722},
  year    = {2023},
  doi     = {10.1016/j.jcp.2022.111722}
}

@article{ShinDarbonKarniadakis2020,
  author  = {Y. Shin and J. Darbon and G. E. Karniadakis},
  title   = {On the convergence of physics informed neural networks for linear second-order elliptic and parabolic type {PDE}s},
  journal = {Communications in Computational Physics},
  volume  = {28},
  number  = {5},
  pages   = {2042--2074},
  year    = {2020},
  doi     = {10.4208/cicp.OA-2020-0193}
}

@article{MishraMolinaro2023,
  author  = {S. Mishra and R. Molinaro},
  title   = {Estimates on the generalization error of physics-informed neural networks for approximating partial differential equations},
  journal = {IMA Journal of Numerical Analysis},
  volume  = {43},
  number  = {1},
  pages   = {1--43},
  year    = {2023},
  doi     = {10.1093/imanum/drab093}
}

@article{RaissiAhmadiPerdikarisKarniadakis2024,
  author  = {M. Raissi and N. Ahmadi and P. Perdikaris and G. E. Karniadakis},
  title   = {Physics-Informed Neural Networks and Extensions},
  journal = {arXiv preprint arXiv:2408.16806},
  year    = {2024}
}

@article{DeRyckMishra2024,
  author    = {T. De Ryck and S. Mishra},
  title     = {Numerical analysis of physics-informed neural networks and related models in physics-informed machine learning},
  journal   = {Acta Numerica},
  volume    = {33},
  pages     = {633--713},
  year      = {2024},
  doi       = {10.1017/S0962492923000089},
  url       = {https://doi.org/10.1017/S0962492923000089}
}

@article{ZhangLiuAllaDarbonKarniadakis2026,
  author  = {Z. Zhang and S. Liu and A. Alla and J. Darbon and G. E. Karniadakis},
  title   = {PINNs in PDE Constrained Optimal Control Problems: Direct vs Indirect Methods},
  journal = {arXiv preprint arXiv:2604.04920},
  year    = {2026}
}

@article{YongLuoSun2024,
  author  = {J. Yong and X. Luo and S. Sun},
  title   = {Deep multi-input and multi-output operator networks method for optimal control of PDEs},
  journal = {Electronic Research Archive},
  volume  = {32},
  number  = {7},
  pages   = {4291--4320},
  year    = {2024},
  doi     = {10.3934/era.2024193}
}

@article{GarciaCerveraKesslerPeriago2023,
  author  = {C. J. Garc{\'i}a-Cervera and M. Kessler and F. Periago},
  title   = {Control of Partial Differential Equations via Physics-Informed Neural Networks},
  journal = {Journal of Optimization Theory and Applications},
  volume  = {196},
  pages   = {391--414},
  year    = {2023},
  doi     = {10.1007/s10957-022-02100-4}
}

@article{AllaBertagliaCalzola2025,
  author  = {A. Alla and G. Bertaglia and E. Calzola},
  title   = {A PINN Approach for the Online Identification and Control of Unknown PDEs},
  journal = {Journal of Optimization Theory and Applications},
  volume  = {206},
  pages   = {8},
  year    = {2025},
  doi     = {10.1007/s10957-025-02686-5}
}

@article{BensoussanNguyenTranTu2026,
  author  = {A. Bensoussan and T. P. B. Nguyen and M.-B. Tran and S. N. T. Tu},
  title   = {Operator Splitting, Policy Iteration, and Machine Learning for Stochastic Optimal Control},
  journal = {arXiv preprint arXiv:2603.12167},
  year    = {2026},
  url     = {https://arxiv.org/abs/2603.12167}
}

@incollection{BensoussanLiNguyenTranYamZhou2022,
  author    = {A. Bensoussan and Y. Li and D. P. C. Nguyen and M.-B. Tran and S. C. P. Yam and X. Zhou},
  title     = {Machine Learning and Control Theory},
  booktitle = {Numerical Control: Part A},
  series    = {Handbook of Numerical Analysis},
  volume    = {23},
  pages     = {531--558},
  year      = {2022},
  publisher = {Elsevier},
  doi       = {10.1016/bs.hna.2021.12.016}
}

@article{walton2022deep,
title = {A Deep Learning Approximation of Non-Stationary Solutions to Wave Kinetic Equations},
author = {Walton, S. and Tran, M.-B. and Bensoussan, A.},
journal = {Applied Numerical Mathematics},
year = {2024},
volume = {199},
pages = {213--226},
doi = {10.1016/j.apnum.2023.12.010},
}

@article{FernandezZuazua2000,
title = {Null and approximate controllability for weakly blowing up semilinear heat equations},
journal = {Annales de l'Institut Henri Poincaré C, Analyse non linéaire},
volume = {17},
number = {5},
pages = {583-616},
year = {2000},
issn = {0294-1449},
doi = {https://doi.org/10.1016/S0294-1449(00)00117-7},
author = {Fernández-Cara, E. and Zuazua, E.},
}

@article{fernandez-cara_guerrero_2006,
author = {Fern{\'a}ndez-Cara, E. and Guerrero, S.},
title = {Global Carleman inequalities for parabolic systems and application to controllability},
journal = {SIAM Journal on Control and Optimization},
volume = {45},
number = {4},
pages = {1395--1446},
year = {2006},
doi = {10.1137/05064078X}
}

@book{zuazua2024exactcontrollabilitystabilizationwave,
  author       = {E. Zuazua},
  title        = {Exact Controllability and Stabilization of the Wave Equation},
  translator   = {Darlis Bracho Tudares},
  series       = {UNITEXT — Texts in Mathematics},
  volume       = {162},
  edition      = {1st},
  publisher    = {Springer Nature Switzerland AG},
  address      = {Cham},
  year         = {2024},
  pages        = {xviii+133},
  isbn         = {978-3-031-58856-3},
  isbn-electronic = {978-3-031-58857-0},
  doi          = {10.1007/978-3-031-58857-0},
  url          = {https://doi.org/10.1007/978-3-031-58857-0}
}

@article{fu_yong_zhang_2007,
author = {Fu, X. and Yong, J. and Zhang, X.},
title = {Exact controllability for multidimensional semilinear hyperbolic equations},
journal = {SIAM Journal on Control and Optimization},
volume = {46},
number = {5},
pages = {1578--1614},
year = {2007},
doi = {10.1137/040610222}
}

@article{Lin2006,
  author = {Lin, P. and Zhou, Z. and Gao, H.},
  title = {Exact controllability of the parabolic system with bilinear control},
  journal = {Applied Mathematics Letters},
  year = {2006},
  volume = {19},
  number = {6},
  pages = {568--575},
  doi = {10.1016/j.aml.2005.05.016}
}

@article{Beauchard2011Local,
author = {Beauchard, K.},
title = {Local controllability and non-controllability for a 1D wave equation with bilinear control},
journal = {Journal of Differential Equations},
year = {2011},
volume = {250},
number = {6},
pages = {2064--2098},
doi = {10.1016/j.jde.2010.10.008}
}

@article{Ouzahra2013,
  author = {Ouzahra, M.},
  title = {Comments on ``{C}ontrollability of the wave equation with bilinear controls''},
  journal = {European Journal of Control},
  year = {2013},
  volume = {20},
  number = {2},
  doi = {10.1016/j.ejcon.2013.10.007}
}

@misc{yang2025deepneuralnetworksgeneral,
      title={Deep Neural Networks with General Activations: Super-Convergence in Sobolev Norms}, 
      author={Yang, Y. and He, J.},
      year={2025},
      eprint={2508.05141},
      archivePrefix={arXiv},
      primaryClass={cs.LG},
      url={https://arxiv.org/abs/2508.05141}, 
}

@book {Evans:1998:PDE,
    AUTHOR = {Evans, L. C.},
     TITLE = {Partial differential equations},
    SERIES = {Graduate Studies in Mathematics},
    VOLUME = {19},
 PUBLISHER = {American Mathematical Society},
   ADDRESS = {Providence, RI},
      YEAR = {1998},
     PAGES = {xviii+662},
      ISBN = {0-8218-0772-2},
   MRCLASS = {35-01},
  MRNUMBER = {MR1625845 (99e:35001)},
MRREVIEWER = {Luigi Rodino},
}

@article{Cazenave1980,
  author    = {Cazenave, T. and Haraux, A.},
  title     = {{\'{E}}quations d'{\'{e}}volution avec non lin{\'{e}}arit{\'{e}} logarithmique},
  journal   = {Annales de la Facult{\'{e}} des sciences de Toulouse : Math{\'{e}}matiques},
  volume    = {2},
  number    = {1},
  pages     = {21--51},
  year      = {1980},
  series    = {5e s{\'{e}}rie},
}

@article{BardosLebeauRauch1992,
author = {Bardos, C. and Lebeau, G. and Rauch, J.},
title = {Sharp Sufficient Conditions for the Observation, Control, and Stabilization of Waves from the Boundary},
journal = {SIAM Journal on Control and Optimization},
volume = {30},
number = {5},
pages = {1024-1065},
year = {1992},
doi = {10.1137/0330055},
}

@article{DehmanErvedozaZuazua2025,
     author = {B. Dehman and S. Ervedoza and E. Zuazua},
     title = {Regional and partial observability and control of waves},
     journal = {Comptes Rendus. Math\'ematique},
     pages = {1467--1497},
     year = {2025},
     publisher = {Acad\'emie des sciences, Paris},
     volume = {363},
     doi = {10.5802/crmath.805},
     language = {en},
}

@article{LeRousseau2017,
     author = {Le Rousseau, J. and Lebeau, G. and Terpolilli, P. and Tr\'elat, E.},
     title = {Geometric control condition for the wave equation with a time-dependent observation domain},
     journal = {Analysis \& PDE},
     year = {2017},
     volume = {10},
     number = {4},
     pages = {983--1015},
     doi = {10.2140/apde.2017.10.983},
}

@book{LionsMagenesVol2,
  author    = {J.-L. Lions and E. Magenes},
  title     = {Non-Homogeneous Boundary Value Problems and Applications. Vol. II},
  series    = {Die Grundlehren der mathematischen Wissenschaften},
  volume    = {182},
  publisher = {Springer-Verlag},
  address   = {Berlin and New York},
  year      = {1972},
  isbn      = {978-3-642-65217-2},
  doi       = {10.1007/978-3-642-65217-2}
}

@book{McLean2000,
  author    = {W. McLean},
  title     = {Strongly Elliptic Systems and Boundary Integral Equations},
  publisher = {Cambridge University Press},
  address   = {Cambridge},
  year      = {2000},
  series    = {Cambridge University Press},
  isbn      = {9780521663755},
  isbn10    = {052166375X},
  pages     = {357},
  url       = {https://www.cambridge.org/9780521663755}
}

@article{LuMengMaoKarniadakis2021_DeepXDE,
  author  = {Lu, L. and Meng, X. and Mao, Z. and Karniadakis, George E.},
  title   = {{DeepXDE}: A deep learning library for solving differential equations},
  journal = {SIAM Review},
  volume  = {63},
  number  = {1},
  pages   = {208--228},
  year    = {2021},
  doi     = {10.1137/19M1274067}
}

@article{WangYuPerdikaris2022_NTKPINN,
  author  = {Wang, S. and Yu, X. and Perdikaris, P.},
  title   = {When and why {PINNs} fail to train: A neural tangent kernel perspective},
  journal = {Journal of Computational Physics},
  volume  = {449},
  pages   = {110768},
  year    = {2022},
  doi     = {10.1016/j.jcp.2021.110768}
}

@article{WuZhuTan2023_AdaptiveSamplingPINN,
  author  = {Wu, C. and Zhu, M. and Tan, Q. and Kartha, Y. and Lu, L.},
  title   = {A comprehensive study of non-adaptive and residual-based adaptive sampling for physics-informed neural networks},
  journal = {Computer Methods in Applied Mechanics and Engineering},
  volume  = {403},
  pages   = {115671},
  year    = {2023},
  doi     = {10.1016/j.cma.2022.115671}
}

@article{JagtapKawaguchiKarniadakis2020,
  author = {Jagtap, A. D. and Kawaguchi, K. and Karniadakis, G. E.},
  title = {Locally adaptive activation functions with slope recovery for deep and physics-informed neural networks},
  journal = {Proceedings of the Royal Society A: Mathematical, Physical and Engineering Sciences},
  year = {2020},
  volume = {476},
  number = {2239},
  pages = {20200334},
  doi = {10.1098/rspa.2020.0334},
  publisher = {The Royal Society}
}

@article{SonChoHwang2023_ALPINN,
  author  = {Son, H. J. and Cho, H. J. and Hwang, H. J.},
  title   = {Enhanced physics-informed neural networks with augmented Lagrangian relaxation method},
  journal = {Neurocomputing},
  volume  = {548},
  pages   = {126424},
  year    = {2023},
  doi     = {10.1016/j.neucom.2023.126424}
}

@article{LuPestourieYaoWangVerdugoJohnson2021_hPINN,
  author  = {Lu, L. and Pestourie, R. and Yao, W. and Wang, Z. and Verdugo, F. and Johnson, S. G.},
  title   = {Physics-informed neural networks with hard constraints for inverse design},
  journal = {SIAM Journal on Scientific Computing},
  volume  = {43},
  number  = {6},
  pages   = {B1105--B1132},
  year    = {2021},
  doi     = {10.1137/21M1397908}
}

@article{MadduSturmMuellerSbalzarini2022_InverseDirichlet,
  author = {Maddu, S. and Sturm, D. and M{\"u}ller, C. L. and Sbalzarini, I. F.},
  title = {Inverse Dirichlet weighting enables reliable training of physics-informed neural networks},
  journal = {Machine Learning: Science and Technology},
  year = {2022},
  volume = {3},
  number = {1},
  pages = {015026}
}

@article{WangSankaranPerdikaris2024_CausalPINN,
  author = {Wang, S. and Sankaran, S. and Perdikaris, P.},
  title = {Respecting causality for training physics-informed neural networks},
  journal = {Computer Methods in Applied Mechanics and Engineering},
  year = {2024},
  volume = {421},
  pages = {116813}
}

@article{YuLuMengKarniadakis2022_gPINN,
  author = {Yu, J. and Lu, L. and Meng, X. and Karniadakis, G. E.},
  title = {Gradient-enhanced physics-informed neural networks for forward and inverse {PDE} problems},
  journal = {Computer Methods in Applied Mechanics and Engineering},
  year = {2022},
  volume = {393},
  pages = {114823}
}

@article{KharazmiZhangKarniadakis2019_VPINN,
  author = {Kharazmi, E. and Zhang, Z. and Karniadakis, G. E.},
  title = {Variational physics-informed neural networks for solving partial differential equations},
  journal = {arXiv preprint arXiv:1912.00873},
  year = {2019}
}

@article{KharazmiZhangKarniadakis2021_hpVPINN,
  author = {Kharazmi, E. and Zhang, Z. and Karniadakis, G. E.},
  title = {hp-{VPINN}s: Variational physics-informed neural networks with domain decomposition},
  journal = {Computer Methods in Applied Mechanics and Engineering},
  year = {2021},
  volume = {374},
  pages = {113547}
}

@article{JagtapKarniadakis2020_XPINN,
  author = {Jagtap, A. D. and Karniadakis, G. E.},
  title = {Extended physics-informed neural networks {(XPINNs):} {A} generalized space-time domain decomposition based deep learning framework for nonlinear partial differential equations},
  journal = {Communications in Computational Physics},
  year = {2020},
  volume = {28},
  number = {5},
  pages = {2002--2041},
  doi = {10.4208/cicp.OA-2020-0164},
  publisher = {Global Science Press}
}

@article{MoseleyMarkhamNissenMeyer2023_FBPINN,
  author = {Moseley, B. and Markham, A. and Nissen-Meyer, T.},
  title = {Finite basis physics-informed neural networks ({FBPINN}s): a scalable domain decomposition approach for solving differential equations},
  journal = {Advances in Computational Mathematics},
  year = {2023},
  volume = {49},
  number = {4},
  pages = {62}
}

@article{YangMengKarniadakis2021_BPINN,
  author = {Yang, L. and Meng, X. and Karniadakis, G. E.},
  title = {{B-PINN}s: Bayesian physics-informed neural networks for forward and inverse {PDE} problems with noisy data},
  journal = {Journal of Computational Physics},
  year = {2021},
  volume = {425},
  pages = {109913}
}

@article{PangLuKarniadakis2019_fPINN,
  author = {Pang, G. and Lu, L. and Karniadakis, G. E.},
  title = {{fPINN}s: Fractional physics-informed neural networks},
  journal = {SIAM Journal on Scientific Computing},
  year = {2019},
  volume = {41},
  number = {4},
  pages = {A2603--A2626}
}

@article{ToscanoOommenVargheseZouAhmadiKarniadakis2024_PIKAN,
  author = {Toscano, J. D. and Oommen, V. and Varghese, A. J. and Zou, Z. and A. Daryakenari, N. and Wu, C. and Karniadakis, G. E.},
  title = {From {PINN}s to {PIKAN}s: recent advances in physics-informed machine learning},
  journal = {arXiv preprint arXiv:2410.13228},
  year = {2024}
}

@article{DawBuWangPerdikarisKarniadakis2023_R3,
  author = {Daw, A. and Bu, J. and Wang, S. and Perdikaris, P.},
  title = {Mitigating propagation failures in physics-informed neural networks using retain-resample-release ({R3}) sampling},
  journal = {Proceedings of the 40th International Conference on Machine Learning},
  year = {2023}
}

@article{MowlaviNabi2023JCP,
   author  = {Mowlavi, S. and Nabi, S.},
   title   = {Optimal control of {PDE}s using physics-informed neural networks},
   journal = {Journal of Computational Physics},
   volume  = {473},
   pages   = {111731},
   year    = {2023},
   doi     = {10.1016/j.jcp.2022.111731}
 }

@techreport{BarryStraumeSarsharPopovSandu2022,
   author      = {Barry-Straume, J. and Sarshar, A. and Popov, A. A. and Sandu, A.},
   title       = {Physics-informed neural networks for {PDE}-constrained optimization and control},
  institution = {Computational Science Laboratory, Virginia Tech},
  number      = {CSL-TR-22-2},
  year        = {2022},
 note        = {arXiv:2205.03377}
}
	
	\appendix

\end{document}